\documentclass[11pt]{article}
\input epsf.tex
\usepackage{amssymb,latexsym,times}
\usepackage{mathptmx}
\usepackage[all]{xy}
\catcode`\@=11
\renewcommand\subsection{\@startsection{subsection}{2}{\z@}%
                                     {-3.25ex\@plus -1ex \@minus -.2ex}%
                                     {-0.01 mm}
                                     {\normalfont\large\bfseries}}

\catcode`\@=11
\renewcommand\subsubsection{\@startsection{subsubsection}{2}{\z@}%
                                     {-3.25ex\@plus -1ex \@minus -.2ex}%
                                     {-0.01 mm}
                                     {\normalfont\bfseries}}
\newtheorem{theorem}{Theorem}[section]
\newtheorem{example}[theorem]{Example}
\newtheorem{corollary}[theorem]{Corollary}
\newtheorem{definition}[theorem]{Definition}
\newtheorem{proposition}[theorem]{Proposition}
\newtheorem{lemma}[theorem]{Lemma}
\newtheorem{remark}[theorem]{Remark}
\newtheorem{conjecture}[theorem]{Conjecture}

\def\resp{{\em resp.$\ $}}

\def\cqfd{\hfill $\Box$ \bigskip}
\def\adots{\mathinner{\mkern2mu\raise1pt\hbox{.}
\mkern3mu\raise4pt\hbox{.}\mkern1mu\raise7pt\hbox{.}}}
\def\<{\langle\,}
\def\>{\,\rangle}

\def\cf{{\it cf.$\ $}}

\def\ie{{\em i.e.\,}}
\def\eg{{\em e.g.\,}}

\def\l{\lambda}
\def\a{\alpha}
\def\de{\delta}
\def\b{\beta}
\def\ga{\gamma}

\def\N{{\mathbb N}}
\def\Z{{\mathbb Z}}
\def\C{{\mathbb C}}

\def\Q{{\mathbb Q}}

\def\F{{\cal F}}
\def\P{{\cal P}}

\def\SS{{\cal S}}

\def\g{\mathfrak g}

\def\hg{\widehat{\mathfrak g}}

\def\gl{\mathfrak{gl}}
\def\Sl{\mathfrak{sl}}

\def\slchap{\widehat{\mathfrak{sl}}}

\def\H{\widehat H}

\def\A{{\cal A}}

\def\bar{\overline}

\def\<{\langle}
\def\>{\rangle}

\def\pr{{\rm pr\,}}
\def\PP{{\mathbb P}}

\def\deg{{\rm deg}}

\def\CC{{\cal C}}

\def\M{{\cal M}}
\def\RR{{\cal R}}

\def\le{\leqslant}
\def\ge{\geqslant}

\def\eps{\epsilon}

\def\Si{\Sigma}
\def\si{\sigma}

\def\vp{\varphi}

\def\De{\Delta}

\def\md{{\rm mod}}

\def\D{{\cal D}}
\def\ds{\displaystyle}

\def\uu{{\mathbf u}}
\def\vv{{\mathbf v}}
\def\xx{{\mathbf x}}

\def\zz{{\mathbf z}}
\def\Y{{\mathcal Y}}
\def\vpi{\varpi}
\def\om{\omega}
\def\FM{\mathrm{FM}}
\def\tB{\widetilde{B}}
\def\hy{\widehat{y}}
\def\gg{\mathbf{g}}

\def\Gr{\mathrm{Gr}}

\title{\bf Cluster algebras and quantum affine algebras}
\author{David Hernandez and Bernard Leclerc}
\date{}

\begin{document}
\maketitle

\begin{abstract}
Let $\CC$ be the category of finite-dimensional representations of 
a quantum affine algebra $U_q(\hg)$ of simply-laced type.
We introduce certain monoidal subcategories $\CC_\ell\ (\ell\in\N)$
of $\CC$ and we study their Grothendieck rings using cluster algebras.
\end{abstract}

\setcounter{tocdepth}{1}
{\footnotesize \tableofcontents}

\section{Introduction}
\subsection{} 
Let $\g$ be a simple Lie algebra of type $A_n, D_n$ or $E_n$, and let 
$U_q(\hg)$ denote the corresponding quantum affine algebra,
with parameter $q\in\C^*$ not a root of unity.
The monoidal category $\CC$ of finite-dimensional $U_q(\hg)$-modules
has been studied by many authors from different perspectives
(see \eg \cite{AK,CP,FR,GV,KS,N1}).
In particular its simple objects have been classified by 
Chari and Pressley, and Nakajima has calculated their character 
in terms of the cohomology of certain quiver varieties.

In spite of these remarkable results many basic questions
remain open, and in particular little is known about the 
tensor structure of $\CC$. 

When $\g=\Sl_2$, Chari and Pressley \cite{CP2} have shown
that every simple object is isomorphic to a tensor
product of simple objects 
of a special type called Kirillov-Reshetikhin modules. 
Conversely, they have shown that a tensor product
$S_1\otimes \cdots \otimes S_k$ of Kirillov-Reshetikhin
modules is simple if and only if $S_i\otimes S_j$ is 
simple for every $1\le i<j \le k$.
Moreover, $S_i\otimes S_j$ is simple if and only if 
$S_i$ and~$S_j$ are ``in general position'' (a combinatorial condition
on the roots of the Drinfeld polynomials of $S_i$ and $S_j$). 
Hence, the Kirillov-Reshetikhin modules can be regarded as 
the {\em prime} simple objects of~$\CC$ \cite{CP5}, and one knows which 
products of primes are simple.
As an easy corollary, one can see that the tensor powers 
of any simple object of $\CC$ are simple.

For $\g \not = \Sl_2$, the situation is far more complicated.
Thus, already for $\g=\Sl_3$, we do not know a general factorization 
theorem for simple objects
(see \cite{CP5}, where a tentative list of prime simple objects is conjectured).
In fact, it was shown in \cite{L} that the tensor square of
a simple object of $\CC$ is not necessarily simple in general,
so one should not expect results similar to the $\Sl_2$ case
for other Lie algebras $\g$.

Because of these difficulties, we decide in this paper to
focus on some smaller subcategories. 
We introduce a sequence 
\[
 \CC_0 \subset \CC_1 \subset \cdots \subset \CC_\ell \subset \cdots,
\qquad (\ell\in \N),
\]
of full monoidal subcategories of $\CC$, whose objects are
characterized by certain strong restrictions on the roots
of the Drinfeld polynomials of their composition factors.
By construction, the Grothendieck ring $R_\ell$ of $\CC_{\ell}$ is
a polynomial ring in $n(\ell+1)$ variables, where $n$ is the rank of $\g$.
Our starting point is that $R_\ell$ is naturally equipped with
the structure of a cluster algebra.

Recall that cluster algebras were introduced by Fomin and Zelevinsky
\cite{FZ0} as a combinatorial device for studying canonical bases
and total positivity.
They found immediately lots of applications, including a proof
of a conjecture of Zamolodchikov concerning certain discrete dynamical systems
arising from the thermodynamic Bethe ansatz, called $Y$-systems \cite{FZ1}.
As observed by Kuniba, Nakanishi and Suzuki \cite{KNS}, 
$Y$-systems are strongly related with the representation theory 
of $U_q(\hg)$ via some other systems of functional relations called $T$-systems. 
It was conjectured in \cite{KNS} that the characters of the Kirillov-Reshetikhin
modules are solutions of a $T$-system, and this was later proved
by Nakajima \cite{N2} in the simply-laced case, and by Hernandez in the
general case \cite{H2}. Now it is easy to notice that in the simply-laced
case the equations of a $T$-system are exactly of the same form as
the exchange relations in a cluster algebra. 
This led us to introduce a cluster algebra structure on $R_\ell$ by
using an initial seed consisting of a choice of $n(\ell +1)$ Kirillov-Reshetikhin
modules in $\CC_\ell$. The exchange matrix of this seed encodes 
$n\ell$ equations of the $T$-system satisfied by these Kirillov-Reshetikhin modules.
(Note that the seed contains $n$ frozen variables -- or coefficients -- in the sense
of \cite{FZ0}.)
By definition of a cluster algebra, one can obtain new seeds by
applying sequences of mutations to the initial seed. Then one of our main 
conjectures is that all the new cluster variables produced in this way are
classes of simple objects of $\CC_\ell$. (Note that in general, these
simple objects are no longer Kirillov-Reshetikhin modules.) 

\subsection{}
For $\ell = 0$, the cluster structure of $R_0$ is trivial: there is a 
unique cluster consisting entirely of frozen variables.

The case $\ell = 1$ is already very interesting, and most of this
paper will be devoted to it.
Recall that Fomin and Zelevinsky have classified the cluster algebras
with finitely many cluster variables in terms of finite root systems \cite{FZ2}.
It turns out that for every $\g$ the ring $R_1$ has finitely many cluster
variables, and that its cluster type coincides with the root system of~$\g$.
Therefore, one may expect that the tensor structure of the simple
objects of the category $\CC_1$ can be described in ``a finite way''.
In fact we conjecture that for every $\g$ the category $\CC_1$ behaves as nicely as  
the category $\CC$ for $\Sl_2$, and we prove it for $\g$ of type $A_n$
and $D_4$. 

More precisely, we single out a finite set of simple objects
of $\CC_1$ whose Drinfeld polynomials are naturally labeled by the
set of almost positive roots of $\g$ (\ie, positive roots and negative
simple roots). 
Recall that the almost positive roots are in one-to-one
correspondence with the cluster variables \cite{FZ2}, so we shall
call these objects the {\em cluster simple objects}.
To these objects we add $n$ distinguished simple objects
which we call {\em frozen simple objects}.
Our first claim is that the classes of these objects
in $R_1$ coincide with the cluster variables and frozen 
variables.

Recall also that the cluster variables are grouped into overlapping
subsets of cardinality $n$ called clusters \cite{FZ0}. The number
of clusters is a generalized Catalan number, and they can be identified
with the faces of the dual of a generalized associahedron \cite{FZ1}.
Our second claim is that a tensor product of 
cluster simple objects is
simple if and only if all the objects belong to a common cluster.
Moreover, the tensor product of a frozen simple
object with any simple object is again simple.
It follows that every simple object of $\CC_1$
is a tensor product of cluster simple objects and frozen simple objects.
As a consequence, the tensor powers of any simple object of $\CC_1$
are simple.

To prove this, we show that a tensor product $S_1\otimes \cdots \otimes S_k$
of simple objects of $\CC_1$ is simple if and only if $S_i\otimes S_j$ is
simple for every $i\not =j$. This result is proved uniformly for all types.

Note that only the frozen objects and the cluster objects attached
to positive simple roots and negative simple roots are Kirillov-Reshetikhin
modules. The remaining cluster objects (labelled by the positive
non simple roots) probably deserve to be studied more closely.
 
\subsection{}
When $\ell > 1$ the ring $R_\ell$ has in general infinitely many cluster variables,
grouped into infinitely many clusters.
A notable exception is the case $\g=\Sl_2$, for which $R_\ell$ is a 
cluster algebra of finite type $A_\ell$ in the classification of \cite{FZ2}.
In this special case it follows from \cite{CP2} that, again, 
the classes in $R_\ell$ of the simple objects of $\CC_\ell$ are precisely
the cluster monomials of $R_\ell$.

We conjecture that for arbitrary $\g$ and $\ell$, every cluster monomial of $R_\ell$
is the class of a simple object. We also conjecture that, conversely,
the class of a simple object $S$ in $\CC_\ell$ is a cluster monomial if
and only if $S\otimes S$ is simple. In this case, following \cite{L}, 
we call $S$ a {\em real simple object}. We believe that real simple objects
form an interesting class of irreducible $U_q(\hg)$-modules, and the meaning
of our partial results and conjectures is that their characters are governed 
by the combinatorics of cluster algebras.
 
\subsection{}
Let us now describe the contents of the paper in more detail.

In Section~\ref{sect3} we recall the definition of a cluster algebra,
and we introduce the new notion of monoidal categorification of a cluster
algebra (Definition~\ref{defmoncat}). We show (Proposition~\ref{positivity}) that the existence
of a monoidal categorification gives an immediate answer to some important
open problems in the theory of cluster algebras, like the linear independence
of cluster monomials, or the positivity of the coefficients of their
expansion with respect to an arbitrary cluster. 

In Section~\ref{sect2} we briefly review the theory of finite-dimensional
representations of $U_q(\hg)$ and we introduce the categories $\CC_\ell$.
We also recall the definition of the Kirillov-Reshetikhin modules
and we review the $T$-system of equations that they satisfy.

In Section~\ref{sect4} we introduce some simple objects
$S(\a)$ of $\CC_1$ attached to the almost positive roots $\a$,
and we formulate our conjecture (Conjecture~\ref{mainconjecture})
for the category $\CC_1$.
It states that $\CC_1$ is a monoidal categorification
of a cluster algebra $\A$ with the same Dynkin type as $\g$, 
and that the $S(\a)$ are the cluster simple objects.
We illustrate the conjecture in type $A_3$.

In Section~\ref{sect6} we review the definition and main properties of
the $q$-characters of Frenkel-Reshetikhin. One of the main tools to
calculate them is the Frenkel-Mukhin algorithm which we recall and
illustrate with examples.

In Section~\ref{Secttrunc}, we introduce some truncated versions of the
$q$-characters for $\CC_1$. These new truncated characters are much easier
to calculate and they contain all the information to determine the
composition factors of an object of $\CC_1$.
The main result of this section (Proposition~\ref{prop33}) is an
explicit formula for the truncated $q$-character of $S(\a)$
when $\a$ is a multiplicity-free positive root.
 
In Section~\ref{sectFpol}, we review following \cite{FZ1,FZ4} the $F$-polynomials of
the cluster algebra $\A$. These are variants of the Fibonacci polynomials
of \cite{FZ1}, which are the building blocks of the general solution of
a $Y$-system. They satisfy a functional equation similar to a $T$-system
and each cluster variable can be expressed in terms of its $F$-polynomial
in a simple way (Equation~(\ref{eqFpol1})). 
We show that Conjecture~\ref{mainconjecture}~(i) is equivalent to the
fact that the (normalized) truncated $q$-characters of the cluster simple objects are
equal to the $F$-polynomials, and we prove it for the multiplicity-free roots
(Theorem~\ref{multfreeequal}).

In Section~\ref{sectensprod}, we prove an important tensor product theorem
for the category $\CC_1$ (Theorem~\ref{thtensprod}): if $S_1,\ldots, S_k$ are simple
objects of $\CC_1$, then $S_1\otimes\cdots\otimes S_k$ is simple
if and only if $S_i\otimes S_j$ is simple for every $i\not =j$.

In Section~\ref{sectclusterexp}, we introduce following \cite{FZ1} the notions
of compatible roots and cluster expansion. Because of Theorem~\ref{thtensprod}
and of the existence and uniqueness of a cluster expansion \cite{FZ1}, we reduce
Conjecture~\ref{mainconjecture}~(ii) for a given $\g$ to a finite check:
one has to verify that $S(\a)\otimes S(\b)$ is simple for every pair $(\a,\b)$
of compatible roots. 

In Section~\ref{secttypA} and Section~\ref{secttypD} we prove Conjecture~\ref{mainconjecture} 
in type $A_n$ and $D_4$. Conjecture~\ref{mainconjecture}~{(i)} is proved more generally in type $D_n$,
by showing that the truncated $q$-characters of cluster simple objects are given
by the explicit combinatorial formula of \cite{FZ1} for the Fibonacci polynomials of 
2-restricted roots.

In Section~\ref{sectappl} we present some applications of our results
for $\CC_1$. First we show that the $q$-characters of the simple objects
of $\CC_1$ are solutions of a system of functional equations similar to
a periodic $T$-system. Secondly, we explain that the $l$-weight
multiplicities appearing in the truncated $q$-characters of the 
cluster simple objects are equal to some tensor product
multiplicities.
This is reminiscent of the Kostka duality for representations of
$\Sl_n$,
but in our case it is not limited to type~$A$.
Thirdly, we exploit some known geometric formulas for $F$-polynomials
due to Fu and Keller \cite{FK} to express the coefficients
of the truncated $q$-characters of the simple objects of $\CC_1$
as Euler characteristics of some quiver grassmannians. This is similar
to the Nakajima character formula for standard modules, but our
formula works for simple modules.
  
Finally in Section~\ref{sect9} we state our conjectures for 
$\CC_\ell$ (arbitrary $\ell$)
and illustrate them for $\g = \Sl_2$ (arbitrary $\ell$) where they follow
from \cite{CP2}, and $\g=\Sl_3$ ($\ell = 2$). We also explain how our conjecture
for $\g=\Sl_r$ (arbitrary $\ell$) would essentially follow from a general conjecture 
of \cite{GLS1} about the relation between Lusztig's dual canonical and dual semicanonical bases. 
 
\subsection{}
Kedem \cite{Ke} and Di Francesco \cite{KDF} have studied another connection between 
quantum affine algebras and cluster algebras, based on other
types of functional equations ($Q$-systems and generalized $T$-systems).
Keller \cite{Kel2} has obtained a proof of the periodicity
conjecture for $Y$-systems attached to pairs of simply-laced
Dynkin diagrams using 2-Calabi-Yau categorifications
of cluster algebras.
More recently, Inoue, Iyama, Kuniba, Nakanishi and Suzuki \cite{IIKNS}
have also studied the connection between $Y$-systems, $T$-systems,
Grothendieck rings of $U_q(\hg)$ 
and cluster algebras, motivated by periodicity problems.
These papers do not study the relations
between cluster monomials and irreducible $U_q(\hg)$-modules.

After this paper was submitted for publication, Nakajima \cite{N5}
gave a geometric proof of Conjecture~\ref{mainconjecture} for 
all types $A, D, E$, using a tensor category of perverse sheaves 
on quiver varieties. This category also makes sense for non Dynkin
quivers and Nakajima showed that its Grothendieck ring has a cluster
algebra structure and that all cluster monomials are classes of simple
objects. Thanks to Proposition~\ref{positivity} below, this yields strong 
positivity results for every acyclic cluster algebra with a bipartite seed. 
To the best of our knowledge, our conjecture for $\ell > 1$ remains open.

We have presented our main results 
in several seminars and conferences in 2008
and 2009 (IHP Paris (BL), MSRI Berkeley (BL), 
NTUA Athens (BL), ETH Zurich (DH),
UNAM Mexico (DH), Math. Institute Oxford (DH), MIO Oberwolfach (BL)).
We thank these institutions for their kind invitations. 
Special thanks are due to Arun Ram and MSRI for organizing in spring 2008
a program on combinatorial representation theory where a large part of
this work was done. 
We also thank Keller for his preliminary Oberwolfach report
on this work \cite{Kel}.
Finally we thank Nakajima for helpful comments and stimulating
discussions.

\section{Cluster algebras and their monoidal categorifications}\label{sect3}

\subsection{}
We refer to \cite{FZSurv} for an excellent survey on cluster algebras.  
Here we only recall the main definitions and results.

\subsubsection{}\label{defcluster}
Let $0\le n < r$ be some fixed integers.
If $\widetilde{B} = (b_{ij})$ is an $r \times (r-n)$-matrix
with integer entries, then
the {\it principal part} $B$ of 
$\widetilde{B}$ is the square matrix obtained from 
$\widetilde{B}$ by deleting the last $n$ rows.
Given some $k \in [1,r-n]$ define a new $r \times (r-n)$-matrix 
$\mu_k(\widetilde{B}) = (b_{ij}')$ by
\begin{equation}
b'_{ij}=
\left\{
\matrix{
-b_{ij}\hfill & \mbox{if $i=k$ or $j=k$},\cr 
b_{ij} + {\displaystyle\frac{|b_{ik}|b_{kj} + b_{ik}|b_{kj}|}{2}} & \mbox{otherwise},
}
\right.
\end{equation}
where $i \in [1,r]$ and $j \in [1,r-n]$.
One calls $\mu_k(\widetilde{B})$ the {\it mutation} of the matrix $\widetilde{B}$
in direction $k$.
If $\widetilde{B}$ is an integer matrix whose principal part is
skew-symmetric, then it is 
easy to check that $\mu_k(\widetilde{B})$ is also an integer matrix 
with skew-symmetric principal part.
We will assume from now on that $\widetilde{B}$ has skew-symmetric principal part. 
In this case, one can equivalently encode $\widetilde{B}$ by a quiver $\Gamma$ with vertex set
$\{1,\ldots,r\}$ and with $b_{ij}$ arrows from $j$ to $i$ if $b_{ij}>0$
and $-b_{ij}$ arrows from $i$ to $j$ if $b_{ij}<0$.
Note that $\Gamma$ has no loop nor 2-cycle.

Now Fomin and Zelevinsky define a cluster algebra
$\A(\widetilde{B})$ as follows.
Let $\F = \Q(x_1,\ldots,x_r)$ be the field of rational
functions in $r$ commuting indeterminates 
$\xx = (x_1,\ldots,x_r)$. 
One calls $(\xx,\widetilde{B})$ the {\it initial seed} of~$\A(\widetilde{B})$.
For $1 \le k \le r-n$ define 
\begin{equation}\label{mutationformula}
x_k^* = 
\frac{\prod_{b_{ik}> 0} x_i^{b_{ik}} + \prod_{b_{ik}< 0} x_i^{-b_{ik}}}{x_k}.
\end{equation}
The pair 
$(\mu_k(\xx),\mu_k(\widetilde{B}))$, 
where
$\mu_k(\xx)$ is obtained from $\xx$ by replacing $x_k$ by 
$x_k^*$,
is the {\it mutation} of the seed $(\xx,\widetilde{B})$ in direction $k$. 
One can iterate this procedure and obtain new seeds by mutating 
$(\mu_k(\xx),\mu_k(\widetilde{B}))$ in any direction 
$l\in [1,r-n]$. 
Let $\SS$ denote the set of all seeds obtained from $(\xx,\widetilde{B})$ 
by any finite sequence of mutations.
Each seed of $\SS$ consists of an $r$-tuple of elements of $\F$
called a {\em cluster}, and of a matrix.
The elements of a cluster are its {\em  cluster variables}.
Every seed has $r-n$ neighbours obtained by a single mutation in direction
$1 \le k \le r-n$.
One does not mutate the last $n$ elements of a cluster; they are
called {\em frozen variables} and belong to every cluster.
We then define the {\em cluster algebra}
$\A(\widetilde{B})$ as
the subring of $\F$ generated by all the cluster variables of 
the seeds of $\SS$.
The integer $r-n$ is called the {\em rank} of $\A(\widetilde{B})$.

A {\em cluster monomial} is a monomial in the cluster variables 
of a {\em single} cluster.
Note that the exchange relation (\ref{mutationformula}) is of
the form 
\begin{equation}\label{mutation2}
 x_k x_k^* = m_+ + m_-
\end{equation}
where $m_+$ and $m_-$ are two cluster monomials.

\subsubsection{}\label{clusterres}
The first important result of the theory 
is that every cluster variable $z$
of $\A(\widetilde{B})$ is a Laurent polynomial in
$\xx$ with
coefficients in $\Z$.
It is conjectured that the coefficients are positive.
Note that because of (\ref{mutationformula}), every cluster variable 
can be written as a subtraction free rational expression in $\xx$, 
but this is not enough to ensure the positivity of its Laurent expansion.

The second main result is the classification of {\em cluster algebras
of finite type}, \ie with finitely many different cluster variables.
Fomin and Zelevinsky proved that this happens if and only if there
exists a seed $(\zz, \widetilde{C})$ such that the quiver
attached to the principal part of $\widetilde{C}$
is a Dynkin quiver (that is, an arbitrary orientation of a Dynkin diagram of
type $A, D, E$).
In this case, the cluster monomials form a distinguished subset 
of $\A(\widetilde{B})$, which is conjectured to be a $\Z$-basis
\cite[\S 11]{FZ4}.

If $\A(\widetilde{B})$ is not of finite cluster type, the cluster
monomials do not span it, but it is conjectured that they are linearly
independent. It is an interesting open problem to specify a
``canonical'' $\Z$-basis of $\A(\widetilde{B})$ containing the
cluster monomials. 

\subsection{}
We now propose a natural framework which would yield positive answers
to the above questions.
We say that a simple object $S$ of a monoidal category is {\em prime}
if there exists no non trivial factorization $S\cong S_1 \otimes S_2$.
We say that $S$ is {\em real} if $S\otimes S$ is simple.

\begin{definition}\label{defmoncat}
Let $\A$ be a cluster algebra and let $\M$ be an abelian monoidal category.
We say that $\M$ is a monoidal categorification of $\A$
if the Grothendieck ring of $\M$ is isomorphic to $\A$,
and if 
\begin{itemize}
 \item[{\rm (i)}] 
the cluster monomials of $\A$ are the classes of all the 
real simple objects of $\M$; 
\item[{\rm (ii)}]
the cluster variables of $\A$ (including the frozen ones)
are the classes of all the real prime simple objects of $\M$. 
\end{itemize}
\end{definition}

The existence of a monoidal categorification of a cluster algebra
is a very strong property, as shown by the following result.
\begin{proposition}\label{positivity}
Suppose that the cluster algebra $\A$ has a monoidal categorification $\M$.
Then
\begin{itemize}
 \item[{\rm (i)}]
every cluster variable of $\A$ has a Laurent expansion 
with positive coefficients with respect to any cluster;
\item[{\rm (ii)}]
the cluster monomials of $\A$ are linearly independent.
\end{itemize}
\end{proposition}

\begin{proof}
If $m$ is a cluster monomial, we denote
by $S(m)$ the simple object with class $[S(m)]=m$.
Let $z$ be a cluster variable, and let 
\[
 z = \frac{N(x_1,\ldots,x_n)}{x_1^{d_1}\cdots x_r^{d_r}}
\]
denote its cluster expansion with respect to the cluster
$\xx = (x_1,\ldots,x_r)$.
Here the numerator $N(x_1,\ldots,x_r)$ is a polynomial
with coefficients in $\Z$.
Multiplying both sides by the denominator, we see that 
$N(x_1,\ldots,x_r)$ is the class of the tensor product
\[
P:=S(z)\otimes S(x_1)^{\otimes d_1}\otimes \cdots \otimes S(x_r)^{\otimes d_r}.
\]
Moreover, since $\xx$ is a cluster, every monomial 
$m=x_1^{k_1}\cdots x_r^{k_r}$ is the class of a simple object 
\[
\Sigma = S(x_1)^{\otimes k_1}\otimes \cdots \otimes S(x_r)^{\otimes k_r}
\]
of $\M$. Hence the coefficient of $m$ in 
$N(x_1,\ldots,x_r)$ is equal to the multiplicity of $\Sigma$
as a composition factor of $P$, thus it is nonnegative. This proves (i).

By definition of a monoidal categorification, the cluster
monomials form a subset of the set of classes of all simple
objects of $\M$, which is a $\Z$-basis of the Grothendieck
group. This proves (ii). 
\cqfd 
\end{proof}

\begin{remark}
{\rm
(i)\ 
In recent years, many examples of categorifications of cluster
algebras have been constructed (see \eg  \cite{MRZ, BMRRT, CC, GLS1, BIRS, CK, GLS2}).
They are quite different from the monoidal categorifications introduced
in this paper. Indeed, these categories are only additive and
have no tensor operation. The multiplication of the cluster algebra
reflects the direct sum operation of the category. We shall call
these categorifications {\em additive}.
Note that there is no analogue of Proposition~\ref{positivity} for additive
categorifications. Although additive categorifications have been helpful for proving positivity
of cluster expansions or linear independence of cluster monomials
in some cases,
this always requires some additional work, for example to show the
positivity of some Euler characteristics. 
Finally, to recover the cluster algebra from its additive categorification
one does not consider the Grothendieck group (which would be too small)
but a kind of ``dual Hall algebra'' constructed via a {\em cluster
character}. This is in general a complicated procedure.

(ii)\ 
In view of the strength and simplicity of Proposition~\ref{positivity}, 
one might wonder whether there exist any examples of monoidal categorifications 
of cluster algebras.
One of the aims of this paper is to produce some examples using representations of quantum 
affine algebras.

(iii)\ 
Let $\M$ be an abelian monoidal category. If $\M$ is a monoidal categorification
of a cluster algebra $\A$, then we get new combinatorial insights about 
the tensor structure of $\M$.
In particular if $\A$ has finite cluster type, we can express any simple
object of $\M$ as a tensor product of finitely many prime objects, and this
yields a combinatorial algorithm to calculate the composition factors 
of a tensor product of simple objects of $\M$.
So this can be a fruitful approach
to study certain interesting monoidal categories $\M$. 
This is the point of view we adopt in this paper.
}
\end{remark}

\section{Finite-dimensional representations of $U_q(\hg)$}\label{sect2}

In this section we briefly review some known results in the representation
theory of quantum affine algebras. For more detailed surveys we refer the reader
to the monograph \cite[chap. 12]{CP} and the recent paper \cite{CH}.  

\subsection{}\label{sect3.1}
Let $\g$ be a simple Lie algebra over $\C$ of type $A_n, D_n$ or $E_n$.
We denote by $I=[1,n]$ the set of vertices of the Dynkin diagram,
by $A=[a_{ij}]_{i,j\in I}$ the Cartan matrix of $\g$,
by $h$ the Coxeter number,
by $\Pi=\{\a_i\mid i\in I\}$ the set of simple roots, by $W$ the Weyl
group, with longest element~$w_0$.

Let $U_q(\hg)$ denote the corresponding quantum affine algebra,
with parameter $q\in\C^*$ not a root of unity.
$U_q(\hg)$ has a Drinfeld-Jimbo presentation, which is a $q$-analogue
of the usual presentation of the Kac-Moody algebra $\hg$.
It also has a second presentation, due to Drinfeld, which is
better suited to study finite-dimensional representations.
There are infinitely many generators 
\[
x^+_{i,r},\ x^-_{i,r},\ h_{i,s},\ k_i,\ k_i^{-1},\ c,\ c^{-1}, 
\qquad (i\in I,\ r\in\Z,\ s\in\Z\setminus \{0\}), 
\]
and a list of relations which we will not repeat (see \eg \cite{FR}).
Remember that $\hg$ can be realized as a central extension
of the loop algebra $\g\otimes \C[t,t^{-1}]$.
If $x^+_i,\,x^-_i,\,h_i\ (i\in I)$ denote the Chevalley generators
of $\g$ then $x^\pm_{i,r}$ is a $q$-analogue of $x^\pm_i\otimes t^r$,
$h_{i,s}$ is a $q$-analogue of $h_i\otimes t^s$,
$k_i$ stands for the $q$-exponential of $h_i\equiv h_i\otimes 1$, 
and $c$ for the $q$-exponential of the central element.

For every $a\in\C^*$ there exists an automorphism $\tau_a$ of $U_q(\hg)$
given by
\[
\tau_a(x^\pm_{i,r}) = a^r x^\pm_{i,r},\quad
\tau_a(h_{i,s}) = a^s h_{i,s},\quad
\tau_a(k_i^{\pm1}) = k_i^{\pm1},\quad
\tau_a(c^{\pm1}) = c^{\pm1}. 
\]
There also exists an involutive automorphism $\si$ given by
\[
\si(x^\pm_{i,r}) = x^\mp_{i,-r},\quad
\si(h_{i,s}) = - h_{i,-s},\quad
\si(k_i^{\pm1}) = k_i^{\mp1},\quad
\si(c^{\pm1}) = c^{\mp1}.  
\]

\subsection{}\label{subCC}
We consider the category $\CC$ of finite-dimensional $U_q(\hg)$-modules
(of type 1).
It is easy to see that if $V$ is an object of $\CC$, then $c$ acts
on $V$ as the identity, and the generators $h_{i,s}$ act by pairwise
commuting endomorphisms of $V$. 

Since $U_q(\hg)$ is a Hopf algebra, $\CC$ is an abelian monoidal category.
It is well-known that $\CC$ is not semisimple.

For an object $V$ in $\CC$ and $a\in\C^*$, we denote by $V(a)$
the pull-back of $V$ under $\tau_a$, and by $V^\si$ the
pull-back of $V$ under $\si$.
The maps $V \mapsto V(a)$ and $V \mapsto V^\si$ give auto-equivalences
of $\CC$.

\subsection{}
Let $I$ denote the set of vertices of the Dynkin diagram of $\g$.
It was proved by Chari and Pressley that
the simple objects $S$ of $\CC$ are parametrized by $I$-tuples 
of polynomials $\pi_S = (\pi_{i,S}(u);\ i\in I)$ in one indeterminate $u$
with coefficients in $\C$ and constant term 1, called 
the {\em Drinfeld polynomials} of $S$.
In particular, for $a\in \C^*$ and $i\in I$ we have a {\em fundamental
module} $V_{i,a}$ which is the simple object with Drinfeld polynomials:
\[
 \pi_{j,V_{i,a}}(u)= \left\{
\matrix{1-au &\mbox{if $j=i$,}\cr
              1 &\mbox{if $j\not = i$.}}
\right.
\]

\subsection{}\label{sect23}
Let $S^*$ denote the dual of $S$. 
The Drinfeld polynomials of $S^*$ are easily deduced 
from those of $S$, namely
\[
 \pi_{i,S^*}(u) = \pi_{i^\vee,S}(q^{-h} u), \qquad (i\in I), 
\]
where $i\mapsto i^\vee$ denotes the involution on $I$ given 
by $w_0(\a_i)=-\a_{i^\vee}$.

The Drinfeld polynomials behave simply under the action of the
automorphisms $\tau_a$, namely, for a simple object $S$ of $\CC$
we have
\[
\pi_{i,S(a)}(u) = \pi_{i,S}(au),\quad 
\qquad (i\in I).
\]
They also behave simply under the action of the
involution $\si$ \cite[Prop. 5.1]{CP3} \cite[Cor. 4.11]{H4}, namely, if 
\[
\pi_{i,S}(u) =  \prod_k(1-a_ku),
\]
then 
\[
\pi_{i^\vee,\,S^\si}(u) = \prod_k(1-a_k^{-1}q^{-h}u).
\]

We also have the following compatibility with tensor products.
If $S_1$ and $S_2$ are simple, and if $S$ is the simple object with
Drinfeld polynomials $\pi_{i,S} := \pi_{i,S_1}\pi_{i,S_2}\ (i\in I)$,
then $S$ is a subquotient of the tensor product $S_1\otimes S_2$. 

\subsection{}
Let $R$ denote the Grothendieck ring of $\CC$. 
It is known \cite[Cor. 2]{FR} that $R$ is the polynomial ring over $\Z$
in the classes $[V_{i,a}]\ (i\in I,\ a\in\C^*)$ of 
the fundamental modules.
 
\subsection{}
Since the Dynkin diagram of $\g$ is a bipartite graph, we have a 
partition $I=I_0\sqcup I_1$ such that every edge connects
a vertex of $I_0$ with a vertex of $I_1$.
The following notation will be very convenient and used in
many places. For $i\in I$ we set
\begin{equation}
\xi_i = 
\left\{
\begin{array}{ll}
0& \mbox{if $i\in I_0$,}\\
1& \mbox{if $i\in I_1$,}
\end{array}
\right.
\qquad
\eps_i=(-1)^{\xi_i}.
\end{equation}
Clearly, the map $i\mapsto \xi_i$ is completely determined
by the choice of $\xi_{i_0}\in \{0,1\}$ for a single vertex~$i_0$,
hence there are only two possible such maps.

\subsection{}
Let $\CC_\Z$ be the full subcategory of $\CC$ whose objects 
$V$ satisfy: 
\begin{itemize}
\item[]
for every composition factor $S$ of $V$ and
every $i\in I$, the roots of the Drinfeld poly\-no\-mial $\pi_{i,S}(u)$ 
belong to $q^{2\Z+\xi_i}$.
\end{itemize}
The Grothendieck ring $R_\Z$ of $\CC_\Z$ is the subring of $R$
generated by the classes
$[V_{i,q^{2k+\xi_i}}]\ (i\in I,\ k\in\Z)$.
It is known that every simple object $S$ of $\CC$ can be written 
as a tensor product $S_1(a_1)\otimes\cdots\otimes S_k(a_k)$ 
for some simple objects $S_1, \ldots, S_k$ of $\CC_\Z$ and
some complex numbers $a_1,\ldots,a_k$ such that
\[
\frac{a_i}{a_j} \not \in q^{2\Z}, \qquad (1\le i < j \le k).
\]
(This follows, for example, from the fact that 
such a tensor product satisfies
the irreducibility criterion in \cite{Cha2}.)
Therefore, the description of the simple objects of $\CC$
essentially reduces to the description of the simple
objects of $\CC_\Z$.

\subsection{} \label{sect2.5}
We will now introduce an increasing sequence of subcategories of $\CC_\Z$. 
Let $\ell\in\N$. (Note that in this paper $\N$ denotes the set of nonnegative
integers, that is, $0\in\N$.)
\begin{definition}
The category $\CC_\ell$ is the full subcategory of $\CC_\Z$ 
consisting of those objects $V$ which satisfy:
\begin{itemize}
\item[]
for every composition factor $S$ of $V$ and
every $i\in I$, the roots of the Drinfeld poly\-no\-mial $\pi_{i,S}(u)$ 
belong to $\{q^{-2k-\xi_i}\mid 0\le k \le \ell\}$.
\end{itemize}  
\end{definition}
Note that for every simple object $S$ of $\CC_\Z$, there exists $k\in\Z$ such 
that $S(q^k)$ belongs to $\CC_\ell$ for some $\ell\in\N$.

\begin{proposition}\label{propCCl}
$\CC_\ell$ is an abelian monoidal category, with Grothendieck ring
the polynomial ring
\[
 R_\ell = \Z\left[[V_{i,q^{2k+\xi_i}}];\ 0\le k\le \ell\right].
\]
\end{proposition}
The proof of Proposition~\ref{propCCl} will be given in
\ref{qcharconseq}.
It is an easy consequence of the theory of $q$-characters.

\begin{example}\label{exC0}
{\rm
The simple objects of the category $\CC_0$ have a simple description.
Indeed, it follows from \cite[Prop. 6.15]{FM} that all tensor products
of fundamental modules of $\CC_0$ are simple, hence every simple
object of $\CC_0$ is isomorphic to the tensor product of
fundamental modules corresponding to the irreducible factors of 
its Drinfeld polynomials.
Moreover any tensor product of simple objects of $\CC_0$ is simple.
} 
\end{example}

\subsection{}\label{sectchoice}
The definition of $\CC_{\ell}$ depends on the map $i \mapsto \xi_i$
which can be chosen in two different ways. However, if $\g$
is not of type $A_{2n}$, the image of $\CC_{\ell}$
under the auto-equivalence $V \mapsto V^\si(q^{h+2\ell+1})$
is the category defined like $\CC_{\ell}$ but with the opposite
choice of $\xi_i$.
If $\g$ is of type $A_{2n}$, the image of $\CC_{\ell}$
under $V \mapsto V^*(q^{h})$ is the category defined like $\CC_{\ell}$ 
but with the opposite choice of $\xi_i$.
This shows that the choice of $\xi_i$ is in fact irrelevant.

\subsection{}\label{sectKR}
For $i\in I$, $k\in \N^*$ and $a\in\C^*$, the simple object 
$W^{(i)}_{k,a}$ with Drinfeld polynomials
\[
 \pi_{j,W^{(i)}_{k,a}}= \left\{
\matrix{(1-au)(1-aq^2u)\cdots(1-aq^{2k-2}u)
 &\mbox{if $j=i$,}\cr
             \ 1\hfill &\mbox{if $j\not = i$,}}
\right.
\]
is called a {\em Kirillov-Reshetikhin module}.
In particular for $k=1$,
$W^{(i)}_{1,a}$ coincides with the fundamental module $V_{i,a}$.
By convention, $W^{(i)}_{0,a}$ is the trivial representation for every
$i$ and $a$.

The classes $[W^{(i)}_{k,a}]$ in $R$ satisfy the following system of equations
indexed by $i\in I$, $k\in\N^*$, and $a\in\C^*$, called
{\em the $T$-system}: 
\begin{equation}\label{eqTsystem}
[W^{(i)}_{k,a}][W^{(i)}_{k,aq^2}] = [W^{(i)}_{k+1,a}][W^{(i)}_{k-1,aq^2}]
+ \prod_j [W^{(j)}_{k,aq}], 
\end{equation}
where in the right-hand side, the product runs over all vertices
$j$ adjacent to $i$ in the Dynkin diagram.
This was conjectured in \cite{KNS} and proved in \cite{N2, H2}.
Using these equations, one can calculate inductively the expression
of any $[W^{(i)}_{k,a}]$ as a polynomial in the classes 
$[W^{(i)}_{1,a}]$ of the fundamental modules.

\begin{example}\label{exa3}
{\rm
For $\g=\Sl_2$, we have $I=\{1\}$, and we may drop the index $i$ in $[W^{(i)}_{k,a}]$.
The $T$-system reads
\[
[W_{k,a}][W_{k,aq^2}] = [W_{k+1,a}][W_{k-1,aq^2}]
+ 1, \qquad (a\in\C^*,\ k\in\N^*). 
\]
This easily implies that $[W_{k,a}]$ is given by the $k\times k$ determinant:
\begin{equation}\label{eqdet}
[W_{k,a}] =
\left|
\matrix{[W_{1,a}] & 1 & 0 & \cdots &0 \cr
         1& [W_{1,aq^2}] & 1   &\ddots &\vdots \cr
         0& 1 & [W_{1,aq^4}] & \ddots & \vdots             \cr 
         \vdots &\ddots& \ddots &\ddots &\vdots \cr 
          0 &\cdots & 0 &1 & [W_{1,aq^{2k-2}}]
}
\right|
.
\end{equation}
}
\end{example}

\section{The case $\ell=1$: statement of results and conjectures}\label{sect4}

We now focus on the subcategory $\CC_1$.

\subsection{}\label{sect41}
Let $Q=\Z\Pi$ be the root lattice of $\g$.
Let $\Phi\subset Q$ be the root system.
Following \cite{FZ1} we denote by $\Phi_{\ge -1}$ the subset of almost
positive roots, that is,
\[
\Phi_{\ge -1} = \Phi_{>0} \cup (-\Pi) 
\]
consists of the positive roots together with the negatives of the simple roots.
In this section, we attach to each $\b \in \Phi_{\ge -1}$ a simple object $S(\b)$ of $\CC_1$.
\subsubsection{}
To the negative simple root $-\a_i$ we attach the module $S(-\a_i)$
whose Drinfeld polynomials are all equal to~1, except 
$P_{i,S(-\a_i)}$ which is equal to
\begin{equation}
P_{i,S(-\a_i)}(u)=
\left\{
\begin{array}{ll}
1-uq^{2}&\quad \mbox{if}\quad i\in I_0,
\\
1-uq&\quad \mbox{if} \quad i\in I_1.
\end{array}
\right.
\end{equation}
In other words, $S(-\a_i)$ is equal to the fundamental module
$V_{i,q^2}$ if $i\in I_0$, and to the fundamental module
$V_{i,q}$ if $i\in I_1$. 

\subsubsection{}
To the simple root $\a_i$ we attach the module $S(\a_i)$
whose Drinfeld polynomials are all equal to~1, except 
$P_{i,S(\a_i)}$ which is equal to 
\begin{equation}
P_{i,S(\a_i)}(u)=
\left\{
\begin{array}{ll}
1-u &\quad \mbox{if}\quad i\in I_0,
\\
1-uq^3&\quad \mbox{if}\quad i\in I_1.
\end{array}
\right.
\end{equation}
In other words, $S(\a_i)$ is equal to the fundamental module
$V_{i,1}$ if $i\in I_0$, and to the fundamental module
$V_{i,q^3}$ if $i\in I_1$. 

\subsubsection{}
To $\b = \sum_{i\in I} b_i \a_i\in\Phi_{>0}$ we attach
the module $S(\ga)$ whose Drinfeld polynomials are
\begin{equation}
P_{i,S(\b)} = \left(P_{i,S(\a_i)}\right)^{b_i}\,,\qquad (i\in I).
\end{equation} 
If $\b$ is not a simple root, this is \emph{not} a Kirillov-Reshetikhin
module.

\subsubsection{}
For $i\in I$, we denote by $F_i$ the simple module whose Drinfeld polynomials
are all equal to~1, except $P_{i,F_i}$ which is equal to
\begin{equation}
P_{i,F_i}(u)=
\left\{
\begin{array}{ll}
(1-uq^0)(1-uq^2) &\quad \mbox{if}\quad i\in I_0,
\\
(1-uq)(1-uq^{3}) &\quad \mbox{if}\quad i\in I_1.
\end{array}
\right.
\end{equation}
This is a Kirillov-Reshetikhin module. 
The classes $[F_i]$ will play the r\^ole of frozen cluster variables
in the Grothendieck ring $R_1$.

\subsection{}\label{sectgraph}
We define a quiver $\Gamma$ with $2n$ vertices as follows.
We start from a copy of the Dynkin diagram oriented in such a way
that every vertex of $I_0$ is a source, and every vertex of $I_1$
is a sink. For every $i\in I$ we then add a new vertex $i'$ and
an arrow $i \leftarrow i'$ if $i\in I_0$ (\resp $i \rightarrow i'$ if $i\in I_1$). 

\begin{example}
{\rm
Let $\g$ be of type $A_3$. We choose $I_0 = \{1,3\}$.
The quiver $\Gamma$ is 
\[
\matrix{1& \leftarrow & 1' \cr 
         \downarrow &&  \cr
 2& \rightarrow & 2' \cr 
         \uparrow &&\cr
3& \leftarrow & 3' 
}
\]
}
\end{example} 

Put $I'=\{i' \mid i\in I\}$. 
It is often convenient to identify $I$ with $[1,n]$, $I'$ with $[n+1,2n]$,
and $i'$ with $i+n$. This will be done freely in the sequel.
As explained in \ref{defcluster} we can attach a matrix $\widetilde{B} = (b_{ij})$
to $\Gamma$. This is a $2n\times n$-matrix with set of column indices $I$, 
and set of row indices $I \cup I'$, the vertex set of $\Gamma$.
The entry $b_{ij}$ is equal to 1 if there is an arrow from $j$ to $i$
in $\Gamma$, to -1 if there is an arrow from $i$ to $j$
in $\Gamma$, and to 0 otherwise.

\begin{example}
{\rm
We continue the previous example. We have
\[
 \widetilde{B}=
\pmatrix{0 & -1 & 0\cr 
         1 & 0 & 1\cr 
         0 & -1 & 0 \cr 
         -1 & 0 & 0 \cr 
         0 & 1 & 0 \cr 
         0 & 0 & -1
}.
\]

}
\end{example}  

\subsection{}\label{clusterA1}
Let $\A=\A(\widetilde{B})$ be the cluster algebra attached as 
in \ref{defcluster} to the initial seed $(\xx,\widetilde{B})$,
where $\xx=(x_1,\ldots,x_{2n})$.
This is a cluster algebra of rank $n$ with $n$ frozen variables.
By construction, the principal part of $\widetilde{B}$ is
a skew-symmetric matrix encoded by a Dynkin quiver (with 
sink-source orientation) of the same Dynkin type as $\g$.
As recalled in \ref{clusterres}, it follows from \cite{FZ2}
that $\A$ has finitely many cluster variables.
Moreover, the (non-frozen) cluster variables are naturally
labelled by $\Phi_{\ge -1}$ via their denominator \cite[Th.1.9]{FZ2}. 
Let $x[\b]$ denote the cluster
variable attached to $\b\in\Phi_{\ge -1}$, 
and let $f_i=x_{i+n}$ denote the frozen variable attached to 
$i'\equiv i+n \in I'$. 

\begin{example}\label{exa9}
{\rm
We continue the previous examples. 
We have 
\[
 \Phi_{\ge -1} = \{-\a_1,\,-\a_2,\,-\a_3,\,\a_1,\,\a_2,\,\a_3,\,\a_1+\a_2,\,\a_2+\a_3,\,\a_1+\a_2+\a_3\}.
\]
The cluster algebra $\A$ is the subring of $\F=\Q(x_1,x_2,x_3,f_1,f_2,f_3)$
generated by 
\[
x[-\a_1]=x_1,\ x[-\a_2]=x_2,\ x[-\a_3]=x_3,\ f_1,\ f_2,\ f_3,
\] 
and the following cluster variables
\[
 \begin{array}{rll}
x[\a_1] &=& \ds\frac{x_2+f_1}{x_1},\\[4mm]
x[\a_2] &=& \ds\frac{x_1x_3+f_2}{x_2},\\[4mm]
x[\a_3] &=& \ds\frac{x_2+f_3}{x_3},\\[4mm]
x[\a_1+\a_2] &=& \ds\frac{f_2x_2+f_1x_1x_3+f_1f_2}{x_1x_2},\\[4mm]
x[\a_2+\a_3] &=& \ds\frac{f_2x_2+f_3x_1x_3+f_2f_3}{x_2x_3},\\[4mm]
x[\a_1+\a_2+\a_3] &=& \ds\frac{f_2x_2^2+f_1f_3x_1x_3+f_1f_2x_2+f_2f_3x_2+f_1f_2f_3}{x_1x_2x_3}.\\[4mm]
\end{array}
\]
}
\end{example} 

\begin{lemma}\label{lem10}
The cluster algebra $\A$ is equal to the polynomial ring in the
$2n$ variables 
\[
x[-\a_i],\ x[\a_i],\quad (i\in I).
\] 
\end{lemma}
\begin{proof}
This follows from \cite{BFZ}.
Indeed, $\A$ is an acyclic cluster algebra, and $x[-\a_i],\ x[\a_i]$
are the generators denoted by $x_i,\, x'_i$ in \cite{BFZ}.
The monomials in $x_1,\, x'_1,\ldots, x_n,\, x'_n$ which contain no product
of the form $x_jx'_j$ are called {\em standard}, and by \cite[Cor. 1.21]{BFZ} they form a basis of $\A$
over the ring $\Z[f_i \mid i\in I]$. 
It then follows from the relations $f_i=x_ix'_i-\prod_{j\not = i}x_j^{-a_{ij}}$
that the set of all monomials in $x_1,\, x'_1,\ldots, x_n,\, x'_n$ is a basis of $\A$
over $\Z$. 
\cqfd 
\end{proof}

\begin{example}\label{exa10}
{\rm
We continue Example~\ref{exa9}.
The generators of $\A$ can be expressed as polynomials in 
\[
 x[-\a_1],\ x[-\a_2],\ x[-\a_3],\ x[\a_1],\ x[\a_2],\ x[\a_3],\ 
\]
as follows: 
\[
\begin{array}{rll}
f_1 &=& x[-\a_1]x[\a_1]-x[-\a_2],\\[2mm]
f_2 &=& x[-\a_2]x[\a_2]-x[-\a_1]x[-\a_3],\\[2mm]
f_3 &=& x[-\a_3]x[\a_3]-x[-\a_2],\\[2mm]
x[\a_1+\a_2] &=& x[\a_1]x[\a_2]-x[-\a_3],\\[2mm]
x[\a_2+\a_3] &=& x[\a_2]x[\a_3]-x[-\a_1],\\[2mm]
x[\a_1+\a_2+\a_3] &=& x[\a_1]x[\a_2]x[\a_3]-x[-\a_1]x[\a_1]-x[-\a_3]x[\a_3]+x[-\a_2].
\end{array}
\]
}
\end{example}

\subsection{}\label{sect44}
We have seen in \ref{sect2.5} that the Grothendieck ring $R_1$
is the polynomial ring in the classes of the $2n$ fundamental modules
of $\CC_1$:
\[
 S(-\a_i),\ S(\a_i),\qquad (i\in I).
\]
By Lemma~\ref{lem10}, the assignment 
\[
 x[-\a_i]\mapsto [S(-\a_i)],\quad x[\a_i] \mapsto [S(\a_i)],\qquad (i\in I),
\]
extends to a ring isomorphism $\iota$ from $\A$ to $R_1$. 

We can now state the main theorem-conjecture of this section.
\begin{conjecture}\label{mainconjecture}
\begin{itemize}
 \item[\rm(i)] 
We have 
\[
\iota(x[\b]) = [S(\b)],\quad \iota(f_i) = [F_i],\qquad (\b\in\Phi_{\ge -1},\ i\in I).
\]
\item[\rm(ii)]
If we identify $\A$ to $R_1$ via $\iota$, 
$\CC_1$ becomes a monoidal categorification of $\A$. 
Moreover, the class in $R_1$ of any simple object of
$\CC_1$ is a cluster monomial (in other words, every simple
object of $\CC_1$ is real). 
\end{itemize}
\end{conjecture}

The conjecture will be proved in Section~\ref{secttypA} for $\g$ of type $A_n$. 
In Section~\ref{secttypD} we prove it for $\g$ of type $D_4$,
and we prove (i) for $\g$ of type $D_n$.

Conjecture~\ref{mainconjecture}~(ii) immediately implies the following
\begin{corollary}\label{cor13}
\begin{itemize}
\item[\rm (i)] The category $\CC_1$ has finitely many prime simple objects,
namely, the cluster simple objects $S(\b)\ (\b\in \Phi_{\ge -1})$, and
the frozen simple objects $F_i\ (i\in I)$.
\item[\rm (ii)] Each simple object of $\CC_1$ is a tensor product of primes
whose classes belong to one of the clusters of $\A$.

\item[\rm (iii)] All the tensor powers of a simple object of $\CC_1$ are simple.

\item[\rm (iv)] The cluster monomials form a $\Z$-basis of $\A$.
\cqfd
\end{itemize}
\end{corollary}

\begin{example}\label{exa14cl}
{\rm
We continue the previous examples. 
By Corollary~\ref{cor13}, for $\g$ of type $A_3$ the category $\CC_1$
has 12 prime simple objects:
\[
\begin{array}{c}
 S(-\a_1),\ S(-\a_2),\ S(-\a_3),\ S(\a_1),\ S(\a_2),\ S(\a_3),\ S(\a_1+\a_2),\ S(\a_2+\a_3),\ S(\a_1+\a_2+\a_3),
\\[2mm]
F_1,\ F_2,\ F_3.
\end{array} 
\]
The  first 6 modules are fundamental representations, $F_1, F_2, F_3$ are Kirillov-Reshetikhin
modules, $S(\a_1+\a_2)$, $S(\a_2+\a_3)$ are minimal affinizations
(in the sense of \cite{Cha}), but $S(\a_1+\a_2+\a_3)$ is {\em not} a minimal affinization.
Its underlying $U_q(\g)$-module is isomorphic to 
$V(\varpi_1+\varpi_2+\varpi_3)\oplus V(\varpi_2)$.
(Here we denote by $\varpi_i \ (i\in I)$ the fundamental weights
of $\g$, and by $V(\l)$ the irreducible $U_q(\g)$-module with
highest weight $\l$.)

Using the known dimensions of the fundamental modules and 
the formulas of Example~\ref{exa10}, we can easily calculate their
dimensions, namely (in the same order):
\[
 4,\ 6,\ 4,\ 4,\ 6,\ 4,\ 20,\ 20,\ 70,\ 10,\ 20,\ 10.
\]

The cluster algebra $\A$ has 14 clusters:
\[
 \begin{array}{c}
\{-\a_1,\,-\a_2,\,-\a_3\},\ 
\{\a_1,\,-\a_2,\,-\a_3\},\ 
\{-\a_1,\,\a_2,\,-\a_3\},\ 
\{-\a_1,\,-\a_2,\,\a_3\},\ 
\{\a_1,\,-\a_2,\,\a_3\},\ 
\\[2mm]
\{-\a_1,\,\a_2,\,\a_2+\a_3\},\
\{-\a_1,\,\a_3,\,\a_2+\a_3\},\
\{-\a_3,\,\a_2,\,\a_1+\a_2\},\ 
\{-\a_3,\,\a_1,\,\a_1+\a_2\},\
\\[2mm]
\{\a_1+\a_2,\,\a_2,\,\a_2+\a_3\},\
\{\a_1,\,\a_3,\,\a_1+\a_2+\a_3\},\
\{\a_1,\,\a_1+\a_2,\,\a_1+\a_2+\a_3\},\ 
\\[2mm]
\{\a_3,\,\a_2+\a_3,\,\a_1+\a_2+\a_3\},\ 
\{\a_1+\a_2,\,\a_2+\a_3,\,\a_1+\a_2+\a_3\}.
\end{array}
\]
Here we have written for short $\b$ instead of $x[\b]$ and
we have omitted the three frozen variables which belong to every
cluster.

The simple objects of $\CC_1$ are exactly
all tensor products of the form
\[
S(\b_1)^{\otimes k_1}\otimes S(\b_2)^{\otimes k_2}\otimes S(\b_3)^{\otimes k_3}\otimes 
F_1^{\otimes l_1}\otimes F_2^{\otimes l_2}\otimes F_3^{\otimes l_3},
\quad (k_1,k_2,k_3,l_1,l_2,l_3)\in\N^6,
\]
in which $\{\b_1,\,\b_2,\,\b_3\}$ runs over the 14 clusters listed above. 
Moreover, simple objects multiply (in the Grothendieck ring) as the
corresponding cluster monomials in $\A$.
For instance, the exchange formula
\[
 x[-\a_2]x[\a_2] = f_2 + x[-\a_1]x[-\a_3]
\]
shows that the tensor product $S(-\a_2)\otimes S(\a_2)$
has two composition factors isomorphic to $F_2$ and
$S(-\a_1)\otimes S(-\a_3)$. 
}
\end{example} 

\section{$q$-characters}

\label{sect6}

\subsection{}
An essential tool for our proof of Conjecture~\ref{mainconjecture}
in type $A$ and $D$ is the theory of $q$-characters
of Frenkel and Reshetikhin \cite{FR}. We shall recall briefly 
their definition and main properties. For more details see \eg \cite{CH}.

\subsubsection{}
Recall from \ref{subCC} that if
$V$ is a finite-dimensional $U_q(\hg)$-module, the endomorphisms
of $V$ which represent the generators $h_{i,m}$ commute with each other.
Moreover, by Drinfeld's presentation the $k_i$ are pairwise
commutative, and they commute with all the $h_{i,m}$.
Hence we can write $V$ as a finite direct sum of common generalized eigenspaces
for the simultaneous action of the $k_i$ and of the $h_{i,m}$.
These common generalized eigenspaces are called the
{\em l-weight-spaces} of $V$. (Here {\em l} stands for ``loop'').
The {\em $q$-character of $V$} is a Laurent polynomial with positive
integer coefficients in some indeterminates
$Y_{i,a}\ (i\in I, a\in\C^*)$,
which encodes the decomposition of~$V$
as the direct sum of its {\em l}-weight-spaces .

More precisely, the eigenvalues of the $h_{i,m}\ (m>0)$ 
in an {\em l}-weight-space $W$ of $V$ 
are always of the form
\begin{equation}\label{eq24}
\frac{q^m-q^{-m}}{m(q-q^{-1})}
\left(
\sum_{r=1}^{k_i}(a_{ir})^m-\sum_{s=1}^{l_i}(b_{is})^m
\right)
\end{equation}
for some nonzero complex numbers $a_{ir}, b_{is}$.
Moreover, they completely determine 
the eigenvalues of the $h_{i,m}\ (m<0)$ and of the $k_i$ on $W$.
We encode this collection of eigenvalues with 
the Laurent monomial
\begin{equation}\label{eq23}
 \prod_{i\in I}\left(
\prod_{r=1}^{k_i} Y_{i,a_{ir}} \prod_{s=1}^{l_i} Y_{i,b_{is}}^{-1}
\right),
\end{equation}
and the coefficient
of this monomial in the $q$-character of $V$ is the dimension of 
$W$ \cite[Prop. 2.4]{FM}.
The collection of eigenvalues (\ref{eq24}) is called the {\em l-weight}
of $W$. By a slight abuse, we shall often say that the {\em l}-weight
of $W$ is the monomial (\ref{eq23}).

Let $\Y = \Z[Y_{i,a}^\pm ; i\in I, a\in\C^*]$, and denote by
$\chi_q(V)\in\Y$ the $q$-character of $V\in\CC$.
One shows that $\chi_q(V)$ depends only on the class of $V$
in the Grothendieck ring $R$, and that the induced map
$\chi_q : R \to \Y$ is an injective ring homomorphism.

\begin{example}\label{exsl2}
{\rm
Let $\g=\Sl_2$.
Then $I=\{1\}$, and we can drop the index $i\in I$.
Hence 
\[
\Y = \Z[Y_a^\pm ; a\in\C^*].
\] 
One calculates easily the $q$-character of the fundamental module 
$W_{1,a}$. It is two-dimensional, and decomposes as a sum of 
two common eigenspaces:
\begin{equation}\label{eqW1}
 \chi_q(W_{1,a})=Y_a+Y_{aq^2}^{-1}.
\end{equation}
From the identity 
\[
[W_{1,a}][W_{1,aq^2}] = [W_{2,a}][W_{0,aq^2}]+ 1=[W_{2,a}]+1,
\]
(see Example~\ref{exa3}), one deduces that
\[
\chi_q(W_{2,a})=Y_aY_{aq^2}+Y_aY_{aq^4}^{-1}+Y_{aq^2}^{-1}Y_{aq^4}^{-1}.
\]
One can calculate similarly the $q$-character of every Kirillov-Reshetikhin
module for $U_q(\slchap_2)$ (see below Example~\ref{exKRsl2}).
}
\end{example}

\subsubsection{}
$U_q(\hg)$ has a natural subalgebra isomorphic to $U_q(\g)$,
hence every $V\in\CC$ can be regarded as a $U_q(\g)$-module by restriction.
The {\em l}-weight-space decomposition of $V$ is a refinement
of its decomposition as a direct sum of $U_q(\g)$-weight-spaces. 
Let $P$ be the weight lattice of $\g$, with basis given by the fundamental
weights $\vpi_i\ (i\in I)$.
We denote by $\om$ the $\Z$-linear map from $\Y$ to $\Z[P]$ defined by
\begin{equation}
 \om\left(\prod_{i,a}Y_{i,a}^{u_{i,a}}\right) = \sum_i \left(\sum_a u_{i,a}\right) \vpi_i.
\end{equation}
If $W$ is an {\em l}-weight-space of $V$ with {\em l}-weight the 
Laurent monomial $m\in \Y$, then $W$ is a subspace of the $U_q(\g)$-weight-space
with weight $\om(m)$.
Hence, the image $\om(\chi_q(V))$ of the $q$-character of $V$ is
the ordinary character of the underlying $U_q(\g)$-module.

For $i\in I$ and $a\in\C^*$ define
\begin{equation}\label{eqrootA}
 A_{i,a} = Y_{i,aq}Y_{i,aq^{-1}}\prod_{j\not = i}Y_{j,a}^{a_{ij}}.
\end{equation}
(Note that, because of our general assumption that $\g$ is simply-laced 
(see \S\ref{sect3.1}), for $i\not = j$ we have $a_{ij}\in\{0,-1\}$.)
Thus $\om(A_{i,a})=\a_i$, and the $A_{i,a}\ (a\in \C^*)$ should be viewed as affine 
analogues of the simple root~$\a_i$.
Following \cite{FR}, we define a partial order on the set $\M$ 
of Laurent monomials in the $Y_{i,a}$ by setting:
\[
m \le m'\quad \Longleftrightarrow \quad \mbox{$\ds\frac{m'}{m}$ is a monomial in the $A_{i,a}$.}
\]
This is an affine analogue of the usual partial order on $P$,
defined by $\l \le \l'$ if and only if $\l'-\l$ is a sum of simple roots $\a_i$.
 
Let $S$ be a simple object of $\CC$ with Drinfeld polynomials
\begin{equation}\label{eqDr}
 \pi_{i,S}(u) = \prod_{k=1}^{n_i}(1-ua_k^{(i)}),\qquad (i\in I).
\end{equation}
Then the subset $\M(S)$ of $\M$ consisting of all the monomials occuring in $\chi_q(S)$ has a 
unique maximal element with respect to $\le$, which is equal to
\begin{equation}\label{eqmon}
m_S=\prod_{i\in I}\prod_{k=1}^{n_i}Y_{i,a_k^{(i)}}
\end{equation}
and has coefficient $1$ \cite[Th. 4.1]{FM}. This is the {\em highest weight monomial}
of $\chi_q(S)$.
The one-dimensional {\em l}-weight-space of $S$ with {\em l}-weight
$m_S$ consists of the {\em highest-weight vectors} of $S$, that is,
the {\em l}-weight vectors $v\in S$ such that $x_{i,r}^+v=0$ for every 
$i\in I$ and $r\in\Z$. 

A monomial $m\in\M$ is called {\em dominant} if it does not contain
negative powers of the variables $Y_{i,a}$. The highest weight
monomial $m_S$ of an irreducible $q$-character $\chi_q(S)$ is always dominant.
We will denote by $\M_+$ the set of dominant monomials.

Conversely, we can associate to any dominant monomial as in (\ref{eqmon})
a unique set of Drinfeld polynomials given by (\ref{eqDr}).
Hence, we can equivalently parametrize the isoclasses of simple objects
of $\CC$ by $\M_+$. 
In the sequel, the simple module whose $q$-character
has highest weight monomial $m\in\M_+$ will be denoted by $L(m)$.
To summarize, we have
\[
 \chi_q(L(m)) = m\left(1 + \sum_p M_p\right),
\]
where all the $M_p$ are monomials in the variables $A_{i,a}^{-1}$.
It will sometimes be convenient to renormalize the $q$-characters
and work with
\begin{equation}\label{eqchit}
 \widetilde{\chi}_q(L(m)) := \frac{1}{m}\chi_q(L(m)) = 1 + \sum_p M_p.
\end{equation}
This is a polynomial with positive integer coefficients in the variables
$A_{i,a}^{-1}$.
\begin{example}\label{exKRsl2}
{\rm 
We continue Example~\ref{exsl2} and describe the $q$-characters
of all the simple objects of $\CC$ for $\g=\Sl_2$.
For $a\in\C^*$, we have $A_a = Y_{aq}Y_{aq^{-1}}$.
For $k\in\N$, put 
$m_{k,a}=Y_{a}Y_{aq^2}\cdots Y_{aq^{2k-2}}$.
It follows from (\ref{eqdet}) and (\ref{eqW1}) that for the Kirillov-Reshetikhin modules we have \cite{FR}:
\begin{equation}\label{eqKR}
 \chi_q(W_{k,a})=m_{k,a}\left(1+A_{aq^{2k-1}}^{-1}\left(1+A_{aq^{2(k-1)-1}}^{-1}
\left(1+\cdots\left(1+A_{aq^3}^{-1}(1+A_{aq}^{-1})\right)\cdots \right)\right)\right).
\end{equation}
We call {\em $q$-segment of origin $a$ and length $k$}
the string of complex numbers 
\[
\Si(k,a)=\{a,\ aq^2,\ \cdots, aq^{2k-2}\}.
\]
Two $q$-segments are said to be in {\em special position} if one does 
not contain the other, and their union is a $q$-segment.
Otherwise we say that there are in {\em general position}.
It is easy to check that every finite multi-set 
$\{b_1,\ldots, b_s\}$ of elements of $\C^*$ 
can be written uniquely as a union of segments $\Si(k_i,a_i)$ 
in such a way that
every pair $(\Si(k_i,a_i),\,\Si(k_j,a_j))$ is in general position. 
Then, Chari and Pressley \cite{CP2} have proved that the simple 
module $S$
with Drinfeld polynomial 
\[
\pi_{S}(u)=\prod_{m=1}^s(1-ub_m)
\] 
is isomorphic to the tensor product of Kirillov-Reshetikhin modules $\bigotimes_i W_{k_i,a_i}$.
Hence $\chi_q(S)$ can be calculated using (\ref{eqKR}).
} 
\end{example}

An important consequence of the existence and uniqueness of
the highest $l$-weight of a simple module is:
\begin{proposition} {\rm \cite{FM}} \label{prop21}
Let $V$ and $W$ be two objects of $\CC$. If $\chi_q(V)$ and 
$\chi_q(W)$ have the same dominant monomials with the same multiplicities,
then $\chi_q(V)=\chi_q(W)$.
\end{proposition}
Indeed, to express $\chi_q(V)$ as a sum of irreducible $q$-characters, one
can use the following simple procedure. Pick a dominant 
monomial $m$ in $\chi_q(V)$ which is maximal with respect
to the partial order $\le$. Then $\chi_q(V)-\chi_q(L(m))$ is
a polynomial with nonnegative coefficients.
If $\chi_q(V)-\chi_q(L(m))\not = 0$, then pick a maximal dominant
monomial $m'$ in $\chi_q(V)-\chi_q(L(m))$ and consider
$\chi_q(V)-\chi_q(L(m))-\chi_q(L(m'))$. And so forth.
After a finite number of steps one gets the decomposition of
$\chi_q(V)$ into irreducible characters using only its subset
of dominant monomials.

\subsection{}
We now recall an algorithm due to Frenkel and Mukhin \cite{FM} which
attaches to any $m\in\M_+$ a polynomial $\FM(m)$, equal
to the $q$-character of $L(m)$ under certain conditions.

\subsubsection{}
We first introduce some notation.
Given $i\in I$, we say that $m\in\M$ is {\em $i$-dominant} if every 
variable $Y_{i,a}\ (a\in\C^*)$ occurs in $m$ with non-negative
exponent, and in this case we write $m\in\M_{i,+}$.
Using the known irreducible $q$-characters of $U_q(\slchap_2)$
(see Example~\ref{exKRsl2}) we define for $m\in\M_{i,+}$ a polynomial $\vp_i(m)$ as follows.
Let $\bar{m}$ be the monomial obtained from $m$ by replacing
$Y_{j,a}$ by $Y_a$ if $j = i$ and by $1$ if $j\not = i$. 
Then there is a unique irreducible $q$-character 
$\chi_q(\bar{m})$ of $U_q(\slchap_2)$
with highest weight monomial $\bar{m}$.
Write $\chi_q(\bar{m})=\bar{m}(1 + \sum_p \bar{M}_p)$, where
the $\bar{M}_p$ are monomials in the variables $A_a^{-1}\ (a\in\C^*)$.
Then one sets $\vp_i(m) := m(1 + \sum_p M_p)$ where each $M_p$
is obtained from the corresponding $\bar{M}_p$ by replacing
each variable $A_a^{-1}$ by $A_{i,a}^{-1}$.

Suppose now that $m\in\M_+$.
We define the subset $D_m$ of $\M$ as follows.
A monomial $m'$ belongs to $D_m$ if there is a finite
sequence $(m_0=m, m_1,\ldots, m_t=m')$ such that for all
$r=1,\ldots,t$ there exists $i\in I$ such that $m_{r-1}\in\M_{i,+}$
and $m_r$ is a monomial occuring in $\vp_i(m_{r-1})$.
Clearly, $D_m$ is countable and every $m'\in D_m$ satisfies $m'\le m$.
We can therefore write 
\[
D_m=\{m_0=m,\ m_1,\ m_2,\ \dots\ \}
\] 
in such a way that if $m_t\le m_r$ then $t\ge r$.

Finally, we define inductively some sequences of integers 
$(s(m_r))_{r\geq 0}$ and $(s_i(m_r))_{r\geq 0}\ (i\in I)$
as follows. 
The initial condition is $s(m_0)=1$ and $s_i(m_0)=0$ for all $i\in I$.
For $t\geq 1$, we set
\begin{eqnarray}
s_i(m_t)&=&\ds\sum_{r<t,\ m_r\in\M_{i,+}}(s(m_{r})-s_i(m_{r}))[\vp_i(m_{r}) : m_t],\quad (i\in I),\label{eqsim}
\\[2mm]
s(m_t) &=& \max\{s_i(m_t)\mid i\in I\},\label{eqsm}
\end{eqnarray}
where $[\vp_i(m_{r}) : m_t]$ denotes the coefficient of $m_t$
in the polynomial $\vp_i(m_r)$.
We then put
\begin{equation}\label{eqFM}
\FM(m) := \sum_{r\geq 0}s(m_r) m_r.
\end{equation} 

\begin{example}\label{exa22}
{\rm
Take $\g$ of type $A_2$ and $m=Y_{1,1}Y_{2,q^3}$.
We have 
\[
\begin{array}{l}
\vp_1(m) = Y_{1,1}Y_{2,q^3} + Y_{1,q^2}^{-1}Y_{2,q}Y_{2,q^3},\\[2mm]
\vp_2(m) = Y_{1,1}Y_{2,q^3} + Y_{1,1}Y_{1,q^4}Y_{2,q^5}^{-1},
\end{array}
\]
so we put $m_1=Y_{1,q^2}^{-1}Y_{2,q}Y_{2,q^3}$ and 
$m_2=Y_{1,1}Y_{1,q^4}Y_{2,q^5}^{-1}$.
The monomial $m_1$ is 2-dominant, and we have
\[
\vp_2(m_1)= Y_{1,q^2}^{-1}Y_{2,q}Y_{2,q^3} + Y_{1,q^2}^{-1}Y_{1,q^4}Y_{2,q}Y_{2,q^5}^{-1}
+ Y_{1,q^4}Y_{2,q^3}^{-1}Y_{2,q^5}^{-1};
\]
similarly, $m_2$ is 1-dominant and we have
\[
 \vp_1(m_2)= Y_{1,1}Y_{1,q^4}Y_{2,q^5}^{-1} + Y_{1,q^2}^{-1}Y_{1,q^4}Y_{2,q}Y_{2,q^5}^{-1}
+ Y_{1,1}Y_{1,q^6}^{-1}+Y_{1,q^2}^{-1}Y_{1,q^6}^{-1}Y_{2,q}.
\]
We set $m_3=Y_{1,q^2}^{-1}Y_{1,q^4}Y_{2,q}Y_{2,q^5}^{-1}$, 
$m_4 = Y_{1,q^4}Y_{2,q^3}^{-1}Y_{2,q^5}^{-1}$,
$m_5 = Y_{1,1}Y_{1,q^6}^{-1}$, and 
$m_6 = Y_{1,q^2}^{-1}Y_{1,q^6}^{-1}Y_{2,q}$.
We see that $m_3$ is neither 1-dominant nor 2-dominant.
The monomial $m_4$ is 1-dominant and 
\[
 \vp_1(m_4)=Y_{1,q^4}Y_{2,q^3}^{-1}Y_{2,q^5}^{-1}+Y_{1,q^6}^{-1}Y_{2,q^3}^{-1};
\]
similarly, $m_5$ and $m_6$ are 2-dominant and 
\[
 \vp_2(m_5)= Y_{1,1}Y_{1,q^6}^{-1}, \qquad 
\vp_2(m_6)=Y_{1,q^2}^{-1}Y_{1,q^6}^{-1}Y_{2,q} + Y_{1,q^6}^{-1}Y_{2,q^3}^{-1}.
\]
Finally the monomial $m_7= Y_{1,q^6}^{-1}Y_{2,q^3}^{-1}$ is neither 1-dominant nor 2-dominant.
So 
\[
D_m = \{m,\ m_1,\ m_2,\ m_3,\ m_4,\ m_5,\ m_6,\ m_7\}.
\]
It is easy to check that $\FM(m)$ is the sum of all the elements of $D_m$
with coefficient 1, and that we have
\[
\FM(m)=\vp_1(m)+\vp_1(m_2)+\vp_1(m_4)=\vp_2(m)+\vp_2(m_1)+\vp_2(m_5)+\vp_2(m_6). 
\]
} 
\end{example}

\subsubsection{}
Let $m\in\M_+$. We say that the simple module $L(m)$ is {\em minuscule} 
if $m$ is the only dominant monomial of $\chi_q(L(m))$. (In \cite{N3}
these modules are called {\em special}.)

\begin{theorem}{\rm\cite{FM}\ }
If $L(m)$ is minuscule then $\chi_q(L(m))=\FM(m)$. Moreover, all the fundamental
modules are minuscule.
\end{theorem}

It was proved in \cite{N2} that Kirillov-Reshetikhin modules are also
minuscule.
Unfortunately, there exist simple modules for which the Frenkel-Mukhin
algorithm fails, as shown by the next example. For another example in type $C_3$
see \cite{NN}.
\begin{example}\label{exa24}
{\rm 
Take $\g$ of type $A_2$ and $m=Y_{1,1}^2Y_{2,q^3}$.
Clearly, $L(m)$ is a simple object of $\CC_1$, and using
Conjecture~\ref{mainconjecture}~(ii) (which will be proved in Section~\ref{secttypA} in
type $A_n$), we have in the notation of \ref{sect41}
\[
 L(m)\cong S(\a_1)\otimes S(\a_1+\a_2) = L(Y_{1,1})\otimes L(Y_{1,1}Y_{2,q^3}).
\]
It is easy to calculate that for $\g$ of type $A_{n+1}$ and $a\in\mathbb{C}^*$, we have
\[
\chi_q(L(Y_{1,a}))=Y_{1,a} + Y_{1,aq^2}^{-1}Y_{2,aq} + Y_{2,aq^3}^{-1}Y_{3,aq^2} + \cdots + Y_{n,aq^{n+1}}^{-1}.
\]
Hence, for our example in type $A_2$, $\chi_q(L(m))=\chi_q(L(Y_{1,1}))\chi_q(L(Y_{1,1}Y_{2,q^3}))$ contains
the monomial 
\[
Y_{2,q^3}^{-1} Y_{1,1}Y_{2,q^3} = Y_{1,1}=mA_{1,q}^{-1}A_{2,q^2}^{-1}
\]
with coefficient 1.
On the other hand, $\vp_1(m)=m(1+2A_{1,q}^{-1}+A_{1,q}^{-2})$
and $\vp_2(m)=m(1+A_{2,q^4}^{-1})$.
Now $m_1:=mA_{1,q}^{-1}=Y_{1,1}Y_{1,q^2}^{-1}Y_{2,q}Y_{2,q^3}$ is
2-dominant and $\vp_2(m_1)= m_1(1+A_{2,q^4}^{-1}+A_{2,q^2}^{-1}A_{2,q^4}^{-1})$. 
Thus $\FM(m) - m -2m_1$ is of the form $mP$ where
$P\in\Z[A_{1,a}^{-1},A_{2,a}^{-1};\ a\in\C^*]$ contains
only monomials divisible by $A_{1,q}^{-2}$ or $A_{2,q^4}^{-1}$.
Hence $\FM(m)$ does not contain the monomial $mA_{1,q}^{-1}A_{2,q^2}^{-1}$,
thus $\FM(m)$ is not equal to $\chi_q(L(m))$.
} 
\end{example}
\begin{remark}\label{rem25}
{\rm
(i)\ 
When $\FM(m)=\chi_q(L(m))$, this polynomial
can be written for every $i\in I$ as a positive sum of polynomials 
of the form $\vp_i(m')$ for some $m'\in\M_{i,+}$. 
For $m'\in D_m \cap \M_{i,+}$, the integer $s(m')-s_i(m')$
is then the coefficient of $\vp_i(m')$ in this sum.

(ii)\ 
One can slightly generalize \cite[Th. 5.9]{FM} as follows:
a sufficient condition for having $\FM(m)=\chi_q(L(m))$
is that $\FM(m)$ contains all the dominant monomials of $\chi_q(L(m))$.
The proof is essentially the same as in \cite[Th. 5.9]{FM}.

(iii)\ 
In \cite{H1}, a polynomial $F(m)$ defined by formulas
similar to (\ref{eqsim}), (\ref{eqsm}), (\ref{eqFM}),
has been defined for any $m\in\M_+$.
The only difference between the definitions of $F(m)$ and $\FM(m)$
is the formula giving $s(m_t)$.
If $L(m)$ is minuscule, then $F(m)=\FM(m)=\chi_q(L(m))$.
Otherwise $F(m)$ may not be equal to $\FM(m)$.
The polynomial $\FM(m)$ has nonnegative coefficients,
may contain several dominant monomials, and it may not
belong to the image of the map $\chi_q : R \to \Y$.
On the other hand, $F(m)$ belongs to this image, has
a unique dominant monomial (equal to $m$), but may have
negative coefficients. 

This is well illustrated by Example~\ref{exa24}. In that case
$L(m)$ has dimension 24, and we have
\[
\FM(m) = \chi_q(L(m)) - Y_{1,1},\qquad
F(m) =  \chi_q(L(m)) - Y_{1,1} - Y_{1,q^2}^{-1}Y_{2,q} - Y_{2,q^3}^{-1},
\]
and $F(m)$ contains the monomial $Y_{2,q^3}^{-1}$ with coefficient $-1$.

In this paper when we refer to the Frenkel-Mukhin algorithm we will always mean the polynomial $\FM(m)$.
}
\end{remark}

\subsubsection{}\label{conseqFM}
Let us note a useful consequence of the Frenkel-Mukhin algorithm. 

\begin{proposition}\label{FMmon}
Let $m=\prod_{k=1}^r Y_{i_k,a_k}\in \M_+$ and let $mM$ be a monomial
occuring in $\chi_q(L(m))$, where $M$ is a monomial in the variables
$A_{i,a}^{-1} \ (i\in I,\ a\in\C^*)$.
If $A_{j,b}^{-1}$ occurs in $M$ there exist $k\in \{1,\ldots, r\}$
and $l\in\N\setminus\{0\}$ such that $b=a_kq^l$. Moreover, there also exist
finite sequences 
\[
(j_1=i_k, j_2, \ldots, j_s=j) \in I^s,\quad
(l_1=1\le l_2\le \cdots \le l_s=l) \in \N^s
\]
such that $a_{j_t,j_{t+1}}=-1$ and $A_{j_t,a_kq^{l_t}}^{-1}$ 
occurs in $M$ for every $t=1,\ldots, s-1$.
Finally, $l_t$ is odd if $\eps_{j_t}=\eps_{i_k}$ and even otherwise. 
\end{proposition}
\begin{proof}
If $m=Y_{i,a}$ is the highest weight monomial of a fundamental
module, then $\chi_q(L(m))=\FM(m)$, and the proposition holds
by definition of the algorithm $\FM$.
In the general case, by \ref{sect23} $L(m)$ is a subquotient
of the tensor product $\otimes_{k=1}^r L(Y_{i_k,a_k})$, and
therefore its $q$-character is contained in the product of
the $q$-characters of the factors, which proves the proposition.
\cqfd
\end{proof}

\subsubsection{}\label{qcharconseq}
We can now prove Proposition~\ref{propCCl}.
Let $L(m)$ and $L(m')$ be in $\CC_{\ell}$.
This means that $m$ and $m'$ are monomials in the variables
$Y_{i,q^{\xi_i+2k}}\ (0\le k\le \ell)$.
If $L(m'')$ is a composition factor of $L(m)\otimes L(m')$
then $m''$ is a product of monomials of $\chi_q(L(m))$ and
$\chi_q(L(m'))$.
So we have $m''=m m' M$ where $M$ is a monomial in the $A_{i,a}^{-1}$. 
We claim that, for $m''$ to be dominant,
the spectral parameters $a$
have to be of the form $a=q^{\xi_i+2k+1}$ with $0\le k \le \ell-1$.
Indeed, by Proposition~\ref{FMmon} we know that
these parameters belong to $q^{\xi_i+2\N+1}$.
Let 
\[
 s = \max\{r \mid A_{i,q^{\xi_i + 2r+1}}^{-1} \mbox{occurs in $M$ for some $i$} \}. 
\]
If $s \ge \ell$ then all variables $Y_{i,q^{\xi_i+2s+2}}$ occuring in $m''$
have a negative exponent, and $m''$ cannot be dominant. 
Hence $a=q^{\xi_i+2k+1}$ with $0\le k\le \ell-1$.
It follows that $m''$ depends only on the variables 
$Y_{i,q^{\xi_i+2k}}\ (0\le k\le \ell)$.
Thus $L(m'')$ is in $\CC_{\ell}$ and $\CC_{\ell}$ is closed under
tensor products. The description of the Grothendieck ring $R_\ell$
immediately follows, and this finishes the proof of Proposition~\ref{propCCl}.

\subsubsection{}\label{cons2}
As a consequence of Proposition~\ref{FMmon}, all
simple modules in the category $\mathcal{C}_0$ are minuscule. 
Indeed consider such a module $S$ with highest monomial $m$. 
A monomial $m'$ occuring in $\chi_q(S) - m$ contains at least one 
$A_{j,q^r}^{-1}$ with $r\geq 1$ and $j\in I$. 
As $m\in\Z[Y_{i,q^{\xi_i}}]_{i\in I}$, we can
conclude as in Section \ref{qcharconseq} that $m'$ is not dominant.

This property also holds for other subcategories equivalent to $\mathcal{C}_0$ 
obtained by shifting the spectral parameter by $a\in \C^*$. 
More explicitly, consider the
category of finite-dimensional representations $V$ which satisfy : for every composition factor $S$ of $V$ and
every $i\in I$, the Drinfeld poly\-no\-mial $\pi_{i,S}(u)$ belongs to $\Z[(1 - a^{-1}q^{-\xi_i}u)]$. 
Then one proves exactly as above
that in this category any simple object is minuscule, and every tensor product of simple objects is simple.

\subsection{}\label{belong}
The next proposition is often helpful to prove that certain
monomials belong to the $q$-character of a module.
\begin{proposition}{\rm \cite[Prop.~3.1]{H3}}
Let $V$ be an object of $\CC$ and fix $i\in I$. Then there is a 
unique decomposition of $\chi_q(V)$ as a finite sum
\[
 \chi_q(V) = \sum_{m\in\M_{i,+}} \l_m \vp_i(m), 
\]
and the $\l_m$ are nonnegative integers. 
\end{proposition}

Suppose that we know that $m\in\M_{i,+}$ occurs in $\chi_q(V)$.
Assume also that if $m'\in\M_{i,+}\setminus\{m\}$ is such that $\vp_i(m')$
contains $m$ then $m'$ does not occur in $\chi_q(V)$.
Then the proposition implies that all the monomials of $\vp_i(m)$
occur in $\chi_q(V)$. 
In particular, let $m\in \M_+$ and let $mM$ be a monomial
of $\chi_q(L(m))$, where $M$ is a monomial in the 
$A_{j,a}^{-1}\ (j\in I)$.
If $M$ contains no variable $A_{i,a}^{-1}$ then $mM\in \M_{i,+}$
and $\varphi_i(mM)$ is contained in $\chi_q(L(m))$.


\subsection{}\label{belongJ}
We will sometimes need a natural generalization of \ref{belong}
in which the singleton $\{i\}$ is replaced by an arbitrary subset $J$ of~$I$.
To formulate it we first need to imitate the definition 
of $\vp_i(m)$ and introduce some polynomials $\vp_J(m)$.
We say that $m\in\M$ is {\em $J$-dominant}, and we write $m\in\M_{J,+}$,
if $m$ does not contain negative powers of $Y_{j,a}\ (j\in J,\ a\in\C^*)$.
For $m\in\M$ we denote by $\bar{m}$ the monomial obtained from $m$
by replacing $Y_{i,a}^\pm$ by $1$ if $i\not\in J$. 
If $m\in\M_{J,+}$ then $\bar{m}$ can be regarded as a dominant monomial for the 
subalgebra $U_q(\hg_J)$ of $U_q(\hg)$ generated by the Drinfeld generators
attached to the vertices of~$J$.
Denote by $\chi_{q}(\bar{m})$ the $q$-character of the unique irreducible
$q$-character of $U_q(\hg_J)$ with highest weight monomial $\bar{m}$.
Write $\chi_q(\bar{m})=\bar{m}(1 + \sum_p \bar{M}_p)$, where
the $\bar{M}_p$ are monomials in the variables $\bar{A_{j,a}^{-1}}\ (j\in J,\ a\in\C^*)$.
Then one sets $\vp_J(m) := m(1 + \sum_p M_p)$ where each $M_p$
is obtained from the corresponding $\bar{M}_p$ by replacing
each variable $\bar{A_{j,a}^{-1}}$ by $A_{j,a}^{-1}$.
In particular, if $m\in\M_+$ then $\vp_J(m)$ is the sum of all
the monomials of $\chi_q(L(m))$ of the form $mM$ where
$M$ is a monomial in the $A_{j,a}^{-1}\ (j\in J)$.
We can now state
\begin{proposition}{\rm \cite[Prop.~3.1]{H3}}
Let $V$ be an object of $\CC$ and let $J$ be an arbitrary subset of $I$. 
Then there is a unique decomposition of $\chi_q(V)$ as a finite sum
\[
 \chi_q(V) = \sum_{m\in\M_{J,+}} \l_m \vp_J(m), 
\]
and the $\l_m$ are nonnegative integers. 
\end{proposition}
In particular, let $m\in\M_+$ and 
let $mM$ be a monomial
of $\chi_q(L(m))$, where $M$ is a monomial in the 
$A_{i,a}^{-1}\ (i\in I)$.
If $M$ contains no variable $A_{j,a}^{-1}$ with $j\in J$
then $mM\in \M_{J,+}$
and $\varphi_J(mM)$ is contained in $\chi_q(L(m))$.
 

\section{Truncated $q$-characters}\label{Secttrunc}

The Frenkel-Mukhin algorithm is an important tool because
there is no general formula (like the Weyl character formula)
for calculating an irreducible $q$-character of $\CC$.
Unfortunately, even when the Frenkel-Mukhin algorithm is successful,
the full expansion of the irreducible $q$-character is in general impossible
to handle because it contains too many monomials.
For example, the $q$-character of the 5th fundamental representation
for $\g$ of type $E_8$ has 6899079264 monomials 
(counted with their multiplicities) \cite{N4}.
However, when dealing with the subcategory $\CC_1$, we can work with 
certain truncations of the $q$-characters, as we shall explain in this section.
Thus, we will see that in the category $\CC_1$ in type~$D_4$, the
$q$-character $\chi_q(L(Y_{1,q^3}Y_{2,1}^2Y_{3,q^3}Y_{4,q^3}))$ 
which contains $167237$ monomials (counted with their multiplicities)
can be controlled by its truncated $q$-character which has
only 14 monomials.

\subsection{}\label{sect621}
From now on we will work in the subcategory $\CC_\Z$.
It follows from Proposition~\ref{FMmon} that the
$q$-characters of objects of $\CC_\Z$ involve only monomials in the
variables $Y_{i,q^r}\ (i\in I,\, r\in\Z)$.
To simplify notation, {\em we will henceforth write $Y_{i,r}$ instead
of $Y_{i,q^r}$}. Similarly, {\em we will write $A_{i,r}$ instead of
$A_{i,q^r}$}.

\subsection{}\label{secttrunc}
Let $V$ be an object of $\CC_1$. We can write
\[
 \chi_q(V) = \sum_k m_k(1+\sum_p M_p^{(k)})
\]
where the $m_k$ are dominant monomials in
the variables $Y_{i,\xi_i}, Y_{i,\xi_i+2}\ (i\in I)$, 
and the $M_p^{(k)}$ are certain monomials in the $A_{i,r}^{-1}$.
The factorization of a monomial 
$m=m_kM_p^{(k)}$ is in general not unique. For example,
in type $A_2$ we have $Y_{1,0}Y_{1,2}A_{1,1}^{-1} = Y_{2,1}$.
However, if $m_kM_p^{(k)}$ is such that $M_p^{(k)}$ contains
a negative power of $A_{i,r}$ for some $i\in I$ and some $r\ge 3$, 
because of the restriction on the variables $Y$ occuring
in $m_k$ and of the formula
\[
A_{i,r}^{-1}=Y_{i,r+1}^{-1}Y_{i,r-1}^{-1}\prod_{j:\ a_{ij}=-1}Y_{j,r},
\]
for any other expression $m_kM_p^{(k)} = \widetilde{m}_k\widetilde{M}_p^{(k)}$
the monomial $\widetilde{M}_p^{(k)}$ also contains 
a negative power of $A_{i,r}$ for some $i\in I$ and some $r\ge 3$.  
We define the {\em truncated $q$-character of} $V$ to be the Laurent
polynomial obtained from $\chi_q(V)$ by keeping only the monomials $M_p^{(k)}$
which do not contain any $A_{i,r}^{-1}$ with $r\ge 3$.
We denote this truncated polynomial by $\chi_q(V)_{\le 2}$.
Our motivation for introducing this truncation is the following
\begin{proposition}\label{prop28}
The map $V \mapsto \chi_q(V)_{\le 2}$ is an injective homomorphism 
from the Grothendieck ring $R_1$ of $\CC_1$ to $\Y$.
\end{proposition}
\begin{proof}
It is clear from the definition that our truncated
$q$-character is additive and multiplicative, hence induces
a homomorphism from $R_1$ to $\Y$.
Let us prove injectivity. 
Let $S$ be a simple object of $\CC_1$, 
and $\chi_q(S)=m(1+\sum_k M_k)$ where the $M_k$ are monomials in the
$A_{i,r}^{-1}$.
If $M_k$ contains a variable $A_{i,r}^{-1}$ with $r\ge 3$ then
$mM_k$ can not be dominant.
Indeed, if 
\[
s = \max\{r \mid \mbox{$A_{i,r}^{-1}$ occurs in $M_k$ for some $i$ }\},
\]
then all variables $Y_{i,s+1}$ occuring in $mM_k$ have a negative exponent, and therefore
$mM_k$ is not dominant. (In the terminology of \cite[\S 6.1]{FM},
$mM_k$ is a {\em right negative} monomial.)
Hence, the dominant monomials of 
$\chi_q(S)$ are all contained in its truncated version $\chi_q(S)_{\le 2}$.
Thus, if for two objects $V$ and $W$ of $\CC_1$ we have
{$\chi_q(V)_{\le 2}=\chi_q(W)_{\le 2}$}
then $\chi_q(V)$ and $\chi_q(W)$ have the same dominant monomials, and 
the claim follows from Proposition~\ref{prop21}.
\cqfd 
\end{proof}

\begin{remark}
{\rm
One might consider a different truncated $q$-character,
obtained by keeping only the dominant monomials. 
By Proposition~\ref{prop21}, this truncation is
injective on $R_\Z$, but it is difficult to use because 
it is not multiplicative.  
}
\end{remark}

\begin{example}\label{exa28}
{\rm We continue Example~\ref{exa22} in type $A_2$. 
With our new simplified notation, we have $m= Y_{1,0}Y_{2,3}$, 
and one checks easily that 
\[
\chi_q(L(m))_{\le 2}=m+m_1=Y_{1,0}Y_{2,3}+Y_{1,2}^{-1}Y_{2,1}Y_{2,3}
=Y_{1,0}Y_{2,3}\left(1+A_{1,1}^{-1}\right).
\] 
On the other hand, it follows from Example~\ref{exa24}
that 
\[
\begin{array}{lcl}
 \chi_q(L(Y_{1,0}^2Y_{2,3}))_{\le 2} &=&
\chi_q(L(Y_{1,0}))_{\le 2}\,\chi_q(L(Y_{1,0}Y_{2,3}))_{\le 2}\\[1.5mm]
& = &Y_{1,0}\left(1+A_{1,1}^{-1}+A_{1,1}^{-1}A_{2,2}^{-1}\right)
Y_{1,0}Y_{2,3}\left(1+A_{1,1}^{-1}\right)\\
&=&\ Y_{1,0}^2Y_{2,3}\left(1+ 2A_{1,1}^{-1}+A_{1,1}^{-2}+A_{1,1}^{-1}A_{2,2}^{-1}+A_{1,1}^{-2}A_{2,2}^{-1}\right).
\end{array}
\]
} 
\end{example}

\begin{example}\label{exa31}
{\rm 
Let $\g$ be of type $A, D, E$.
Then for $i\in I$
\[
\chi_q(L(Y_{i,2}))_{\le 2}=Y_{i,2},\ \
\chi_q(L(Y_{i,1}))_{\le 2}=Y_{i,1}\left(1+A_{i,2}^{-1}\right),\]
\[\chi_q(L(Y_{i,0}))_{\le 2}=Y_{i,0}\left(1+A_{i,1}^{-1}\prod_{j\in I, a_{ij} = -1}(1 + A_{j,2}^{-1})\right),\]
\[\chi_q(L(Y_{i,\xi_i}Y_{i,\xi_i+2}))_{\le 2}=Y_{i,\xi_i}Y_{i,\xi_i+2}. 
\]
Indeed, for all these modules the $q$-character is given by the Frenkel-Mukhin
algorithm (the first three are fundamental modules, and the last one is a
Kirillov-Reshetikhin module.)
}
\end{example}

Let us introduce for $P, Q \in \Z[Y_{i,r}^\pm; i\in I, r\in\Z]$ the notation 
\[
 P \le Q \quad \Longleftrightarrow \quad Q - P \in \N[Y_{i,r}^\pm; i\in I, r\in\Z].
\]
\begin{corollary}\label{corregular}
Let $S=L(m)$ be a simple object of $\CC_1$.
Suppose that $\FM(m)\ge\chi_q(S)_{\le 2}$.
Then $\FM(m)=\chi_q(S)$. 
\end{corollary}
\begin{proof}
By the proof of Proposition~\ref{prop28}, $\chi_q(S)_{\le 2}$
contains all the dominant monomials of $\chi_q(S)$,
and the claim follows from Remark~\ref{rem25}~(ii).
\cqfd 
\end{proof}

\subsection{}

Let $\ga=\sum_{i\in I} c_i \a_i$ be an element of the root lattice with nonnegative
coordinates $c_i$. 
We denote by 
$J=\{j\in I\mid c_j \not = 0\}$ the {\em support of $\ga$},
and we assume that $J$ is connected.
Let 
\[
m = \prod_{i\in I_0}Y_{i,0}^{c_i}\prod_{i\in I_1}Y_{i,3}^{c_i}.
\]
This is the highest $l$-weight of 
a simple object $L(m)$ of $\CC_1$.
The next proposition shows how to calculate the truncated $q$-character
of $L(m)$ in terms of the truncation of $\vp_J(m)$.

Let $K=\{k\in I-J \mid a_{kj}=-1 \mbox{ for some $j\in J$}\}$
be the subset of vertices adjacent to $J$. 
For $k\in K$ denote by $j_k$ the unique $j\in J$ such that
$a_{kj}=-1$.
(The uniqueness of $j_k$ follows from the fact that the Dynkin graph has 
no cycle and $J$ is connected.)
Write
\[
\vp_J(m)_{\le 2}\ = m\left(1+\sum_pM_p\right),
\]
where the $M_p$ are monomials in $A_{j,1}^{-1},\,A_{j,2}^{-1}\ (j\in J)$.
In fact, by Proposition~\ref{FMmon}, it is easy to see that each $M_p$ is of the form
$M_p = \prod_{j\in J} A_{j,1+\xi_j}^{-\mu_{j,p}}$
for some $\mu_{j,p}\in\N$. 
\begin{proposition}\label{prop30}
We have
\[
 \chi_q(L(m))_{\le 2}\ = m\left(1+ \sum_pM_p\prod_{k\in K\cap I_1}(1+A_{k,2}^{-1})^{\mu_{j_k,p}}\right).
\]
\end{proposition}
\begin{proof}
From \ref{belongJ}, we have that $\vp_J(m) \le \chi_q(L(m))$.
If $k\in K\cap I_1$ then $j_k\in J\cap I_0$,
and if a monomial $M_p$ contains $A_{j_k,1}^{-1}$
with exponent $\mu_{j_k,p}>0$,
then $mM_p$ contains $Y_{k,1}^{\mu_{j_k,p}}$ and no other $Y_{k,r}$
for $r\not =1$. Therefore $mM_p$ is $k$-dominant. 
Using \ref{belong} we deduce that
\[
\vp_k(mM_p)=M_p(1+A_{k,2}^{-1})^{\mu_{j_k,p}}\le \chi_q(L(m)),
\]
and this implies that
\begin{equation}\label{eq29}
m\left(1+ \sum_pM_p\prod_{k\in K\cap I_1}(1+A_{k,2}^{-1})^{\mu_{j_k,p}}\right)
\le 
\chi_q(L(m))_{\le 2}.
\end{equation}
Note that
if $k\in K\cap I_0$ then $j_k\in J\cap I_1$, 
and if $mM_p$ is $k$-dominant then it contains $Y_{k,2}^{\mu_{j_k,p}}$.
Hence, in this case
$\vp_k(mM_p)_{\le 2} = mM_p$ does not contribute anything
new to the truncated $q$-character. 

Let us now show that (\ref{eq29}) is an equality.
Suppose on the contrary that it is a strict inequality,
and let $mM$ be a monomial in $\chi_q(L(m))_{\le 2}$ which
appears in the left-hand side of (\ref{eq29}) with a strictly smaller multiplicity.
Here $M$ is a monomial in the variables $A_{i,\xi_i+1}^{-1}$.
Moreover, by definition of $\vp_J(m)$ all the monomials of $\chi_q(L(m))$
obtained from $m$ by multiplying by variables $A_{j,\xi_j+1}^{-1}$
with $j\in J$ appear in $\vp_J(m)$ with the same multiplicity.
Hence $M$ has to contain at least one variable $A_{k,\xi_k+1}^{-1}$
with $k\not\in J$. By Proposition~\ref{FMmon} we can assume that $k\in K\cap I_1$.
We will also assume that $mM$ is maximal with these properties.
Denote by $r$ the multiplicity of $mM$ in $\chi_q(L(m))$,
and by $v$ the exponent of $A_{k,2}^{-1}$ in $M$.
Since $mM$ belongs to a truncated $q$-character, $M$ does
not contain any variable $A_{i,3}^{-1}$ with $a_{ik}=-1$,
hence $mM$ has to contain $Y_{k,3}^{-v}$ and therefore it
can not be $k$-dominant.
Hence, by \ref{belong}, there exist $k$-dominant monomials 
$m_1, \ldots, m_s,$ in $\chi_q(L(m))$ (counted with their
multiplicities) such that each $\vp_k(m_i)$ contains $mM$ 
with multiplicity $r_i$ and $r_1+\cdots+r_s=r$. 
Since $m_i>mM$ for every $i$, the multiplicities of $m_i$
in both sides of (\ref{eq29}) are equal. 
Moreover, by our construction, 
the left-hand side of (\ref{eq29}) also contains $\sum_i \vp_k(m_i)$, 
hence the multiplicity of $mM$ in this left-hand side 
is at least $r$, which contradicts our assumption.
Thus (\ref{eq29}) is an equality. 
\cqfd
\end{proof}

\subsection{}
We now calculate $\chi_q(S(\a))_{\le 2}$ for a {\em multiplicity-free positive root} $\a$,
\ie a root of the form $\a=\a_J=\sum_{j\in J} \a_j$ for some connected subset $J$ of $I$.
If $\g$ is not of type $A$, we assume that the trivalent node of the Dynkin diagram
belongs to $I_0$ (this is no loss of generality, see \ref{sectchoice}). 

\begin{proposition}\label{prop33}
Let $J$ be a connected subset of $I$, and $\a=\a_J$ the corresponding 
multiplicity-free positive root.
Let $\eta \in \{0,1\}^I$ be the characteristic function of $J$,
\ie $\eta_i = 1$ if $i\in J$ and $\eta_i=0$ otherwise.
Let $m$ be the highest weight monomial of $\chi_q(S(\a))$.
We have
\[
\chi_q(S(\a))_{\le 2}\ =\ m\sum_{\nu \in \SS}\prod_{i\in I} A_{i,1+\xi_i}^{-\nu_i}, 
\]
where $\SS$ denotes the set of all finite sequences $\nu \in \{0,1\}^I$
such that $\nu_i \le \eta_i$ for every $i\in I_0$, and  for every $i\in I_1$
\begin{equation}\label{eqnu}
\nu_i \le \max\left(0,\ -\eta_i+\sum_{j\ :a_{ij}=-1}\nu_j\right).
\end{equation}
\end{proposition}
\begin{proof}
We first prove the result for $J=I=\{1,\ldots,n\}$. In this case $\eta_i \equiv 1$
and the condition $\nu_i\le\eta_i$ is always satisfied.
We use induction on $n$.
For $n = 1$ we have $\chi_q(S(\a))_{\le 2}\ = Y_{1,0}(1+A_{1,1}^{-1})$
if $I=\{1\}=I_0$,
and $\chi_q(S(\a))_{\le 2}\ = Y_{1,3}$ if $I=I_1$.
Assume now that $n\ge 2$, and let $m_1$ be a monomial 
in the truncated $q$-character of $S(\a)$.
We can suppose that the vertex labelled $1$ is monovalent and adjacent
to the vertex labelled $2$.
Put $I'=\{2,\ldots,n\}$.
There are two cases. 

(a) \ If $1\in I_1$ then $m = Y_{1,3}Y_{2,0}Y_{3,3}\cdots$. Put $m'=mY_{1,3}^{-1}$.
As $L(m)$ appears as a subquotient of $L(Y_{1,3})\otimes L(m')$ and
$$\chi_q(L(Y_{1,3}))_{\le 2}\chi_q(L(m'))_{\le 2} = Y_{1,3} \chi_q(L(m'))_{\le 2},$$
we have $m_1=Y_{1,3}m'_1$ where $m'_1$ is a monomial in $\chi_q(L(m'))_{\le 2}$. 
By Proposition~\ref{prop30}, $\chi_q(L(m'))_{\le 2}$ can be calculated from
$\vp_{I'}(m')_{\le 2}$, and by induction on $n$ we may assume that
\[
 \vp_{I'}(m')_{\le 2} = \ m'\sum_{\nu \in \SS'}\prod_{i\in I'} A_{i,1+\xi_i}^{-\nu_i}, 
\]
where $\SS'$ is defined like $\SS$ replacing $I$ by $I'$.
By Proposition~\ref{prop30}, $\chi_q(L(m'))_{\le 2}$ can only differ from
$\vp_{I'}(m')_{\le 2}$ by certain summands $M_p(1+A_{1,2}^{-1})^{\mu_{2,p}}$
so this shows that the exponent $\nu_i$ of $A_{i,1+\xi_i}^{-1}$ in
$m_1/m$ satisfies the condition (\ref{eqnu}) for $i\in I'$. 
It remains to check that $\nu_1 = 0$.
For this we can use the fact that 
$\chi_q(L(m)) \le \chi_q(L(Y_{1,3}Y_{2,0}))\chi_q(L(m'Y_{2,0}^{-1}))$.
Now by Proposition \ref{prop30} and the first formula in Example \ref{exa28}
(with the nodes $1$ and $2$ of the $A_2$ Dynkin diagram exchanged), we get 
\[
\chi_q(L(Y_{1,3}Y_{2,0}))_{\le 2} = Y_{1,3}Y_{2,0}(1+A_{2,1}^{-1}\prod_{j\neq 1, a_{j2} = -1}(1+A_{j,2}^{-1})).
\]
This does not contain $A_{1,2}^{-1}$, neither $\chi_q(L(m'Y_{2,0}^{-1}))_{\le 2}$.

(b) \ If $1\in I_0$ then $2\in I_1$ is not trivalent, so $\nu_2\le 1$.
We can assume that the vertex $3$ is adjacent to~$2$.
We thus have $m = Y_{1,0}Y_{2,3}Y_{3,0}\cdots$. Put $m'=mY_{1,0}^{-1}$.
Then $m_1=m_0m'_1$ where $m_0$ is a monomial in $\chi_q(L(Y_{1,0}))_{\le 2}$
and $m'_1$ is a monomial in $\chi_q(L(m'))_{\le 2}$.
By Proposition~\ref{prop30}
the monomial $m_0$ belongs to
$\{Y_{1,0},\ Y_{1,0}A_{1,1}^{-1},\ Y_{1,0}A_{1,1}^{-1}A_{2,2}^{-1}\}$.
On the other hand, by Proposition~\ref{prop30} again, 
$\chi_q(L(m'))_{\le 2} = \vp_{I'}(m')_{\le 2}$ in this case, which is known by induction.
If $m_0 \not = Y_{1,0}A_{1,1}^{-1}A_{2,2}^{-1}$ we see that the exponent 
$\nu_i$ of $A_{i,1+\xi_i}^{-1}$ in $m_1/m$ satisfies condition (\ref{eqnu}) for 
every $i\in I$.
If $m_0 = Y_{1,0}A_{1,1}^{-1}A_{2,2}^{-1}$ condition (\ref{eqnu}) would be
violated if $m'_1/m'$ had $\nu_3=0$. But in this case $m_0m'_1$ could not be a monomial
of $\chi_q(L(m))_{\le 2}$. Indeed let us prove that if $A_{2,2}^{-1}$ appears, then $A_{3,1}^{-1}$
must appear. We have the inequality
$\chi_q(L(m))_{\le 2}\ \le\ \chi_q(L(Y_{1,0}Y_{2,3}Y_{3,0}))_{\le 2}\,\chi_q(L(m'Y_{2,3}^{-1}Y_{3,0}^{-1}))_{\le 2}$.
For type $A_3$, again by using the first formula in Example \ref{exa28} and Proposition~\ref{prop30}, we have 
$$\chi_q(L(Y_{2,3}Y_{3,0}))_{\le 2} = Y_{2,3}Y_{3,0}(1 + A_{3,1}^{-1})\ , \ {\chi_q(L(Y_{1,0}Y_{2,3}))_{\le 2} = Y_{1,0}Y_{2,3}(1 + A_{1,1}^{-1})}.$$
Then we have the two inequalities $\chi_q(L(Y_{1,0}Y_{2,3}Y_{3,0}))_{\le 2} \le \chi_q(L(Y_{1,0}))_{\le 2}\chi_q(L(Y_{2,3}Y_{3,0}))_{\le 2}$ and $\chi_q(L(Y_{1,0}Y_{2,3}Y_{3,0})) \le \chi_q(L(Y_{1,0}Y_{2,3}))_{\le 2}\chi_q(Y_{3,0}))_{\le 2}$ which follow as above from a subquotient argument. As a consequence we get
\begin{equation}\label{exa3eq}
\chi_q(L(Y_{1,0}Y_{2,3}Y_{3,0}))_{\le 2} \le
Y_{1,0} Y_{2,3} Y_{3,0}\left(1+A_{1,1}^{-1}+A_{3,1}^{-1}+A_{1,1}^{-1}A_{3,1}^{-1}(1+A_{2,2}^{-1})\right). 
\end{equation}
Now we get the result by Proposition~\ref{prop30}.

Finally, the general case of a root $\a = \a_J$ for a subinterval $J$ of $I$
follows from the case $J=I$. 
Indeed what we have already proved gives us 
$\vp_J(m)_{\le 2}$, and Proposition~\ref{prop30} then yields
the value of $\chi_q(S(\a))_{\le 2}$. 
\cqfd 
\end{proof}

\begin{corollary}\label{correg}
Assume that the trivalent node is in $I_0$ if $\g$ is of type $D$ or $E$.
Let $\a_J$ be a multiplicity-free root. 
Then $\chi_q(S(\a_J))$ is given by the Frenkel-Mukhin algorithm. 
\end{corollary}

\begin{proof}
Let $m$ be the highest $l$-weight of $S(\a)$.
It follows easily from the explicit formula of Proposition~\ref{prop33}
that $\FM(m)\ge\chi_q(S(\a))_{\le 2}$. The only point
which needs to be checked reduces to type $A_3$ for $m = Y_{1,0}Y_{2,3}Y_{3,0}$. By (\ref{exa3eq}), it suffices to check that 
\[
m(1+A_{1,1}^{-1}+A_{3,1}^{-1}+A_{1,1}^{-1}A_{3,1}^{-1}(1+A_{2,2}^{-1}))\le \FM(m).
\]
We have clearly $m(1+A_{1,1}^{-1}+A_{3,1}^{-1}+A_{1,1}^{-1}A_{3,1}^{-1})\le \FM(m)$. Then 
\[
\varphi_2(mA_{1,1}^{-1}A_{3,1}^{-1}) = \varphi_2(Y_{2,1}^2Y_{2,3}) 
= \varphi_2(Y_{2,1})\varphi_2(Y_{2,1}Y_{2,3})
\] 
where $mA_{1,1}^{-1}A_{3,1}^{-1}A_{2,2}^{-1}$ appears.
Now the claim follows from Corollary~\ref{corregular}.
\cqfd 
\end{proof}

If $\g$ is of type $A$, all its positive roots are multiplicity-free, thus 
Proposition~\ref{prop33} and Example~\ref{exa31} give explicit formulae for all the 
truncated $q$-characters $\chi_q(S(\a))_{\le 2}$ in that case.

\begin{example}
{\rm
We take $\g$ of type $A_3$. We assume that $I_0=\{1,3\}$ and $I_1=\{2\}$.
The truncated $q$-characters of the modules $S(\a)\ (\a\in\Phi_{\ge -1})$ are 
\[
\begin{array}{lcllcl}
\chi_q(S(-\a_1))_{\le 2}&=&Y_{1,2},
&\chi_q(S(\a_1))_{\le 2}&=&Y_{1,0}(1+A_{1,1}^{-1}+A_{1,1}^{-1}A_{2,2}^{-1}),
\\[2mm]
\chi_q(S(-\a_2))_{\le 2}&=&Y_{2,1}(1+A_{2,2}^{-1}),\ \
&\chi_q(S(\a_2))_{\le 2}&=&Y_{2,3},
\\[2mm]
\chi_q(S(-\a_3))_{\le 2}&=&Y_{3,2},
&\chi_q(S(\a_3))_{\le 2}&=&Y_{3,0}(1+A_{3,1}^{-1}+A_{2,2}^{-1}A_{3,1}^{-1}),
\end{array}
\]
\[
\begin{array}{lcl}
\chi_q(S(\a_1+\a_2))_{\le 2}&=&Y_{1,0}Y_{2,3}(1+A_{1,1}^{-1}),\\[2mm]
\chi_q(S(\a_2+\a_3))_{\le 2}&=&Y_{2,3}Y_{3,0}(1+A_{3,1}^{-1}),\\[2mm]
\chi_q(S(\a_1+\a_2+\a_3))_{\le 2}&=&Y_{1,0}Y_{2,3}Y_{3,0}(1+A_{1,1}^{-1}
+A_{3,1}^{-1}+A_{1,1}^{-1}A_{3,1}^{-1}+A_{1,1}^{-1}A_{2,2}^{-1}A_{3,1}^{-1}).\\[2mm]
\end{array}
\]
Moreover, by Example~\ref{exa31}, the frozen simple objects $F_i$ have the following truncated $q$-characters
\[
 \chi_q(F_1)_{\le 2} = Y_{1,0}Y_{1,2},\quad
 \chi_q(F_2)_{\le 2} = Y_{2,1}Y_{2,3},\quad
 \chi_q(F_3)_{\le 2} = Y_{3,0}Y_{3,2}.
\]
} 
\end{example}

\begin{corollary}\label{corprime}
Let $\a_J$ be a multiplicity-free root. 
Then $S(\a_J)$ is a prime simple object. 
\end{corollary}

\begin{proof}
Let $m$ be the highest $l$-weight of $S(\a_J)$.
We have $m=\prod_{i\in J\cap I_0} Y_{i,0}\prod_{i\in J\cap I_1}Y_{i,3}$.
Suppose that $S(\a_J)$ is not prime. 
Then 
\begin{equation}\label{wrong}
S(\a_J)\cong L(m_1)\otimes\cdots\otimes L(m_k),
\end{equation}
where $m_1\cdots m_k = m$.
Clearly, there must exist
$i\in J\cap I_0$ and $j\in J\cap I_1$ 
with $a_{ij}=-1$ and such 
that $Y_{i,0}$ and $Y_{j,3}$ do not occur in the 
same monomial $m_r\ (1\le r\le k)$. Let $m_s$ be the
monomial containing $Y_{i,0}$. Then 
$\chi_q(L(m_s))$ contains $m_sA_{i,1}^{-1}A_{j,2}^{-1}$
by Proposition~\ref{prop30},
hence $mA_{i,1}^{-1}A_{j,2}^{-1}$ occurs in
$\chi_q(L(m_1)\otimes\cdots\otimes L(m_k))$.
On the other hand, it follows from Proposition~\ref{prop33}
that $mA_{i,1}^{-1}A_{j,2}^{-1}$ is not a monomial of
$\chi_q(S(\a_J))$, which contradicts (\ref{wrong}).
\cqfd 
\end{proof}

\section{$F$-polynomials}
\label{sectFpol}
In \cite{FZ4}, Fomin and Zelevinsky have shown that the cluster
variables of $\A$ have a nice expression in terms of certain
polynomials called the $F$-polynomials, which are closely related
to the Fibonacci polynomials of \cite{FZ1}.
Moreover, \cite{FZ1} gives some explicit formulae for the 
Fibonacci polynomials in type $A$ and $D$.
We will recall these results in a form suitable to our present
purpose.

\subsection{}\label{sectFpol1}
In Section~\ref{sect4}, we have used 
$\xx=\{x[-\a_i]\mid i\in I\}$ as our reference cluster.
It will be convenient here to work with a slightly different
cluster denoted by $\zz=\{z_i \mid i\in I\}$,
and given by 
\[
z_i = \left\{ 
\begin{array}{ll}
x[-\a_i] & \mbox{if $i\in I_0$,}\\[2mm]
x[\a_i] & \mbox{if $i\in I_1$.}
\end{array}
\right. 
\]
It is easy to see that one passes from $\xx$ to $\zz$
by applying the sequence of mutations $\prod_{k\in I_1}\mu_k$
(it does not matter in which order since they pairwise
commute).
A straightforward calculation shows that
the exchange matrix $\tB_\zz=[b_{ij}^{\zz}]$
at $\zz$ is given by
\[
b_{ij}^\zz = \left\{ 
\begin{array}{ll}
\eps_j a_{ij} & \mbox{if $i,j\in I$ and $i\not = j$,}\\[2mm]
-1 & \mbox{if $j\in I$ and $i=j+n\in I'$,}\\[2mm]
-a_{kj} & \mbox{if $j\in I_0$, and $i=k+n\in I'$ with $k\not = j$,}\\[2mm] 
0 & \mbox{otherwise.}
\end{array}
\right. 
\] 
Here the column set is indexed by $I$ and the row set is
indexed by $I\cup I'\equiv [1,n]\cup [n+1,2n]$.
Following \cite[\S 6]{FZ4} we define the following elements
of $\F$:
\begin{equation}\label{equay}
y_j = \prod_{i\in I}f_i^{b_{i+n,j}^\zz},\qquad
\hy_j = y_j\prod_{i\in I}z_i^{b_{ij}^\zz}.
\end{equation}
\begin{example}
{\rm
We take $\g$ of type $A_3$ and $I_0=\{1,3\}$. We have
\[
 \tB_\zz=
\pmatrix{0 & 1 & 0\cr 
         -1 & 0 & -1\cr 
         0 & 1 & 0 \cr 
         -1 & 0 & 0 \cr 
         1 & -1 & 1 \cr 
         0 & 0 & -1
}.
\]
We have 
\[
y_1 = f_1^{-1}f_2,\quad
y_2 = f_2^{-1},\quad
y_3 = f_2f_3^{-1},\quad
\hy_1 = z_2^{-1} f_1^{-1}f_2,\quad
\hy_2 = z_1z_3 f_2^{-1},\quad
\hy_3 = z_2^{-1} f_2f_3^{-1}. 
\]
}
\end{example}  
Following \cite[\S 10]{FZ4} we denote by $E$ the linear automorphism
of the root lattice $Q$ given by
\[
E(\a_i)=-\eps_i\a_i,\qquad (i\in I).
\]
We also define piecewise-linear involutions $\tau_\eps\ (\eps=\pm1)$
of $Q$ by
\begin{equation}
[\tau_\eps(\ga) : \a_i] = 
\left\{
\begin{array}{ll}
-[\ga : \a_i] - \displaystyle\sum_{j\not = i} a_{ij} \max(0, [\ga : \a_j])
&\mbox{if} \  \eps_i = \eps,\\[5mm]
[\ga : \a_i] &\mbox{if} \ \eps_i\not = \eps.
\end{array}
\right.
\end{equation} 
Here, for $\ga\in Q$, we denote by $[\ga : \a_i]$ the coefficient of 
$\a_i$ in the expansion of $\ga$ on the basis of simple roots.
It is easy to see that $\tau_{\eps}$ preserves $\Phi_{\ge -1}$.
For $\a\in\Phi_{\ge -1}$, one then defines the {\em g-vector}
\begin{equation}
\gg(\a)=E\tau_{-}(\a). 
\end{equation}
The involution $\tau_{-}$ relates the natural labellings of the cluster 
variables with respect to $\xx$ and $\zz$. 
We shall write $z[\a]=x[\tau_{-}(\a)]$.
In particular, since $\tau_{-}(-\a_i)=-\eps_i\a_i$, we have
$z_i=z[-\a_i]\ (i\in I)$.

\subsection{}\label{trop}
Consider the multiplicative group $\PP$ of all Laurent monomials in 
the variables $f_i\ (i\in I)$. 
As in \cite[Def. 2.2]{FZ4}, we introduce
the addition $\oplus$ given by
\[
\prod_i f_i^{a_i} \oplus \prod_i f_i^{b_i} =
\prod_i f_i^{\min(a_i,b_i)}.
\]
Endowed with this operation and its ordinary multiplication and
division, $\PP$ becomes a semifield, called the {\em tropical semifield}.
If $F(t_1,\ldots,t_n)$ is a subtraction-free rational expression
with integer coefficients in some variables $t_i$,
then we can evaluate it in $\PP$ by specializing the $t_i$ to 
some elements $p_i$ of $\PP$. This will be denoted by
$F|_{\PP}(p_1,\ldots,p_n)$.

\subsection{}
To define the $F$-polynomials, we need to introduce a variant
of $\A$ called the {\em cluster algebra with principal coefficients}.
We shall denote it by $\A_{\pr}$. It is given by the initial seed
$(\uu,\tB_{\pr})$, where $\uu=(u_1,\ldots,u_{n},v_1,\ldots,v_n)$, and
$\tB_{\pr}$ is the $2n \times n$ matrix with the same principal part
as $\tB_\zz$ and with lower part equal to the $n\times n$ identity matrix.
Thus $(v_1,\ldots,v_n)$ are the frozen variables of $\A_{\pr}$.
By \cite{FZ2} every cluster variable of $\A_{\pr}$ is of the form
\begin{equation}
u[\a] = \frac{N_\a(u_1,\ldots,u_n,v_1,\ldots,v_n)}{u_1^{a_1}\cdots u_n^{a_n}} 
\end{equation}
for some $\a = \sum_{i\in I} a_i\a_i\in\Phi_{\ge -1}$.
Here $N_\a$ is a polynomial, and $N_{-\a_i}\equiv 1\ (i\in I)$.
Following \cite[\S 3]{FZ4}, we can now define
the {\em $F$-polynomials} by specializing all the $u_i$ to 1:
\begin{equation}
F_\a(v_1,\ldots,v_n) = N_\a(1,\ldots,1,v_1,\ldots,v_n),\qquad (\a\in\Phi_{\ge -1}). 
\end{equation}
For example, for the simple root $\a_i$ we have $F_{\a_i}(v_1,\ldots,v_n) = 1+v_i$.
It is known that $F_\a$ is a polynomial with positive integer coefficients
\cite[Cor.~11.7]{FZ4}.
The main formula is then \cite[Cor.~6.3]{FZ4}:
\begin{equation}\label{eqFpol1}
z[\a] = \frac{F_\a(\hy_1,\ldots,\hy_n)}{F_\a|_{\PP}(y_1,\ldots,y_n)}\zz^{\gg(\a)},  
\end{equation}
where, if $\gg(\a)=(g_1,\ldots,g_n)$, we write for short $\zz^{\gg(\a)}=z_1^{g_1}\cdots z_n^{g_n}$. 
This means that the cluster variable $z[\a]$ is completely determined
by the corresponding $F$-polynomial $F_\a$ and $g$-vector $\gg(\a)$.

\subsection{}
Recall from \S\ref{sect44} the ring isomorphism $\iota : \A \to R_1$.
Note that by comparing the relation 
\[
x[\a_i]x[-\a_i] = f_i + \prod_{j: a_{ij}=-1} x[-\a_j] 
\]
with the $T$-system equation (\ref{eqTsystem}) satisfied by the Kirillov-Reshetikhin
modules $S(\a_i)$ and $S(-\a_i)$, we obtain
immediately 
\[
\iota(f_i) = [F_i], \qquad (i\in I).
\]
Taking into account Proposition~\ref{prop28}, we can regard $\iota$
as an isomorphism from $\A$ to the subring of $\Y$ generated by
the truncated $q$-characters of objects of $\CC_1$.
We then have (\cf Example~\ref{exa31})
\begin{equation}\label{eqatrtr}
\iota(z_i) = Y_{i,\xi_i+2},\qquad \iota(f_i) = Y_{i,\xi_i}Y_{i,\xi_i+2},
\qquad (i\in I). 
\end{equation}
Moreover, extending $\iota$ to a homomorphism from $\F$
to the fraction field of $\Y$, we can consider 
the elements $\iota(\hy_j)$. 
\begin{lemma}\label{lem82}
For $j\in I$, we have 
\[
\iota(\hy_j) = A_{j,\xi_j+1}^{-1}.
\] 
\end{lemma}
\begin{proof}
If $j\in I_1$, we have
\[
\iota(\hy_j) = \iota(f_j)^{-1}\prod_{i\not = j}\iota(z_i)^{-a_{ij}}
= Y_{j,\xi_j}^{-1}Y_{j,\xi_j+2}^{-1} \prod_{i\not = j}Y_{i,\xi_i+2}^{-a_{ij}}
= A_{j,\xi_j+1}^{-1},
\]
since $\xi_i+2=\xi_j+1$ if $j\in I_1$ and $i\in I_0$.
On the other hand, if $j\in I_0$, we have
\[
\iota(\hy_j) = \iota(f_j)^{-1}\prod_{i\not = j}\iota\left(\frac{z_i}{f_i}\right)^{a_{ij}}
= Y_{j,\xi_j}^{-1}Y_{j,\xi_j+2}^{-1} \prod_{i\not = j}Y_{i,\xi_i}^{-a_{ij}}Y_{i,\xi_i+2}^{-a_{ij}}Y_{i,\xi_i+2}^{a_{ij}}
= Y_{j,\xi_j}^{-1}Y_{j,\xi_j+2}^{-1} \prod_{i\not = j}Y_{i,\xi_i}^{-a_{ij}}
= A_{j,\xi_j+1}^{-1},
\]
since $\xi_i=\xi_j+1$ if $j\in I_0$ and $i\in I_1$.
\cqfd 
\end{proof}

\begin{lemma}\label{lem83}
Let $\a\in\Phi_{\ge -1}$ and set $\b=\tau_{-}\a=\sum_ib_i\a_i$.
We have
\[
\iota\left(\frac{\zz^{\gg(\a)}}{F_\a|_{\PP}(y_1,\ldots,y_n)}\right)=
\left\{
\begin{array}{ll}
\ds\prod_{i\in I_0}Y_{i,0}^{b_i} \prod_{i\in I_1}Y_{i,3}^{b_i} & \mbox{if $\b > 0$,}\\[5mm]
Y_{i,2-\xi_i} &\mbox{if $\b=-\a_i$.}
\end{array}
\right.
\]
\end{lemma}
\begin{proof}
Write $\a=\sum_ia_i\a_i$.
If $\a=-\a_i$ then $F_\a = 1$. 
Moreover, $\b=-\eps_i\a_i$, $\gg(\a)=\a_i$, so 
\[
\iota\left(\frac{\zz^{\gg(\a)}}{F_\a|_{\PP}(y_1,\ldots,y_n)}\right)=
\iota(z_i)=Y_{i,\xi_i+2},
\]
which proves the formula in this case.

Otherwise if $\a>0$,
by \cite[Cor.~10.10]{FZ4} the polynomial $F_\a$ has a unique monomial
of maximal degree, which is divisible by all the other occuring monomials
and has coefficient 1. This monomial is $m=\prod_i v_i^{a_i}$.
When we evaluate $m$ at 
$v_i = y_i = \prod_{k\in I}f_k^{b_{k+n,i}^\zz}$
we obtain
\[
\prod_{i\in I} f_i^{-a_i}\prod_{i\in I_1} f_i^{-\sum_{j\not = i}a_ja_{ij}}
=
\prod_{i\in I_0} f_i^{-a_i}\prod_{i\in I_1} f_i^{-a_i-\sum_{j\not = i}a_ja_{ij}}
=
\prod_{i\in I_0} f_i^{-b_i}\prod_{i\in I_1} f_i^{b_i}.
\]
Moreover $F_\a$ has constant term $1$.
Therefore, if $\b>0$, then $b_i>0$ for every $i\in I$ and 
\[
 F_\a|_{\PP}(y_1,\ldots,y_n) = \prod_{i\in I_0} f_i^{-b_i}.
\]
Otherwise, if $\b=-\a_i$ with $i\in I_1$ we have $F_\a|_{\PP}(y_1,\ldots,y_n) = f_i^{-1}$
(the case $\b=-\a_i$ with $i\in I_0$ has already been dealt with).
On the other hand, $\gg(\a)=E(\b)=-\sum_{i\in I} b_i\eps_i\a_i$.
Hence, if $\b>0$ then
\[
\iota\left(\frac{\zz^{\gg(\a)}}{F_\a|_{\PP}(y_1,\ldots,y_n)}\right)=
\prod_{i\in I_0} \iota(f_i)^{b_i} \prod_{i\in I_0} \iota(z_i)^{-b_i}
\prod_{i\in I_1} \iota(z_i)^{b_i} 
=
\prod_{i\in I_0}Y_{i,0}^{b_i} \prod_{i\in I_1}Y_{i,3}^{b_i}.
\]
If $\b=-\a_i$ and $i\in I_1$, we have
\[
\iota\left(\frac{\zz^{\gg(\a)}}{F_\a|_{\PP}(y_1,\ldots,y_n)}\right)=
\iota(f_i)\iota(z_i^{-1})=Y_{i,1}.
\]
\cqfd 
\end{proof}

\begin{corollary}
Let $\b\in\Phi_{\ge -1}$ and set $\a=\tau_{-}\b$.
Let $Y^\b$ denote the highest weight monomial of $\chi_q(S(\b))$.
We have
\[
 \iota(x[\b]) = Y^\b F_\a\left(A_{1,\xi_1+1}^{-1},\ldots,A_{n,\xi_n+1}^{-1}\right).
\]
\end{corollary}
\begin{proof}
This follows immediately from Lemma~\ref{lem82}, Lemma~\ref{lem83}, and
the relation $x[\b]=z[\tau_{-}\b]$.
\cqfd 
\end{proof}

Using the notation of (\ref{eqchit}), we see that 
the proof of Conjecture~\ref{mainconjecture}~(i)
is now reduced to establishing the following polynomial identity
in $\N[A_{i,\xi_i+1}^{-1}]$:
\begin{equation}\label{eq2prove}
\widetilde{\chi}_q(S(\b))_{\le 2} = F_{\tau_{-}(\b)},\qquad (\b\in \Phi_{>0}), 
\end{equation}
that is, {\em the normalized truncated $q$-character of $S(\b)$ should coincide with
the $F$-polynomial attached to the root $\tau_{-}(\b)$}.

\subsection{}
We now recall an explicit formula of Fomin and Zelevinsky for
the $F$-polynomials.
This is obtained by combining \cite[Prop. 2.10]{FZ1}
with \cite[Th. 11.6]{FZ4}.
It covers all the $F$-polynomials in type $A$ and $D$.

\subsubsection{}
Let $\a=\sum_i a_i\a_i \in \Phi_{>0}$.
Assume that $a_i \le 2$ for every $i\in I$. 
Let $\ga = \sum_i c_i\a_i$. We say that $\ga$ is {\em $\a$-acceptable}
if 
\begin{itemize}
\item[(i)] $0\le c_i \le a_i$ for every $i\in I$;
\item[(ii)] if $i\in I_1$ and $j\in I_0$ are adjacent then $c_i\le (2-a_j)+c_j$; 
\item[(iii)] there is no simple
path $(i_0,\ldots,i_m)$ of length
$m\ge 1$ contained in the support of $\a$ such that $a_{i_0}=a_{i_m}=1$
and for $k=0,\ldots, m$
\[
c_{i_k}=\left\{
\begin{array}{ll}
 1 & \mbox{if $i_k\in I_1$,}\\
 a_{i_k}-1 & \mbox{if $i_k\in I_0$.}
\end{array}
\right.
\]
\end{itemize}
In condition (iii) above, by a simple path we mean any path in the Dynkin
diagram whose vertices are pairwise distinct.
Finally, let $e(\ga,\a)$ be the number of connected components
of the set
\[
\{i\in\ I\mid c_i = 1 \mbox{ if } i\in I_1 \mbox{ and } c_i = a_i-1 \mbox{ if } i\in I_0\}
\]
that are contained in $\{i\in I \mid a_i = 2\}$.
Then
\begin{proposition}{\rm\cite{FZ1,FZ4}}\label{propFZ}
\begin{equation}\label{eqFpol}
F_\a(v_1,\ldots,v_n) = \sum_{\ga} 2^{e(\ga,\a)} \vv^\ga,
\end{equation}
where the sum is over all $\a$-acceptable $\ga\in Q$, and
we write $\vv^\ga = \prod_{i\in I}v_i^{c_i}$.
\end{proposition}
\begin{example}
{\rm
Take $\g$ of type $D_4$ and choose $I_0=\{2\}$, where $2$
labels the trivalent node. Let $\a=\a_1+2\a_2+\a_3+\a_4$
be the highest root.
Then
\[
F_\a = 1 + 2v_2 + v_2^2 + 
v_1v_2 + v_2v_3 + v_2v_4 + v_1v_2^2 + v_2^2v_3
+ v_2^2v_4 + v_1v_2^2v_3 + v_1v_2^2v_4 + v_2^2v_3v_4 
+v_1v_2^2v_3v_4.  
\]
} 
\end{example}

\subsubsection{}\label{ssmfree}
If $\a$ is a {\em multiplicity-free root}, that is, if
$a_i\le 1$ for all roots, (\ref{eqFpol}) simplifies
greatly. Indeed, in the definition of an $\a$-acceptable
vector $\ga$ condition~(ii) is a consequence of condition~(i),
and condition~(iii) reduces to
\begin{itemize}
\item[(iv)]\qquad for every $i\in I_1$, 
$c_i \le \min\{c_j \mid \mbox{$j\in\ $supp$(\a)$ and $j$ adjacent to $i$}\}$.
\end{itemize} 

\begin{example}
{\rm
Take $\g$ of type $A_3$ and choose $I_0=\{1,3\}$.
Then
\[
\begin{array}{lll}
F_{\a_1} &=& 1 + v_1,\\
F_{\a_2} &=& 1 + v_2,\\
F_{\a_3} &=& 1 + v_3,\\
F_{\a_1+\a_2} &=& 1 + v_1 + v_1v_2,\\
F_{\a_2+\a_3} &=& 1 + v_3 + v_2v_3,\\
F_{\a_1+\a_2+\a_3} &=& 1 + v_1 + v_3 + v_1v_3 + v_1v_2v_3.
\end{array}
\]
} 
\end{example}

\subsection{}
We can now prove the following particular case of Conjecture~\ref{mainconjecture}~(i)
valid for all Lie types (see Eq.~(\ref{eq2prove})).
We assume (as we may, \cf \ref{sectchoice}) that the
trivalent node is in $I_0$ if $\g$ is of type $D$ or $E$.

\begin{theorem}\label{multfreeequal}
Let $\b$ be a multiplicity-free positive root. Then
\[
\widetilde{\chi}_q(S(\b))_{\le 2} = F_{\tau_{-}(\b)}. 
\]
\end{theorem}
\begin{proof}
Suppose first that $\b=\a_i$ with $i\in I_1$. Then
$\chi_q(S(\b))_{\le 2} = Y_{i,3}$, and $\tau_-(\b)=-\a_i$
so that $F_{\tau_{-}(\b)} = 1$. Hence the claim is verified
in this case.

So we can assume that $\b=\a_J$ is the multiplicity-free
root supported on a connected subset $J$ of $I$ which is not
reduced to a single element of $I_1$. 
Set
\[
\begin{array}{lcl}
J' &=& \{ j\in J \mid j\in I_1 \mbox{ and } a_{ij}=-1 \mbox{ for some } i\in I\setminus J\}, \\[2mm]
K' &=& \{ i\in I\setminus J \mid i\in I_1 \mbox{ and } a_{ij}=-1 \mbox{ for some } j\in J\},
\end{array}
\]
and define $K=(J\setminus J')\sqcup K'$.  
It follows immediately from the definition of $\tau_-$ that the 
multiplicity-free root $\a_K$ supported on $K$ is equal to $\tau_-(\a_J)$.

Let us now compare the sequences $\nu\in\SS$ of Proposition~\ref{prop33} for $\a_J$,
with the $\a_K$-acceptable vectors $\ga = \sum_i c_i\a_i$ of \ref{ssmfree}. 
First we note that if $i\not\in K \cup J = K'\sqcup (J\setminus J') \sqcup J'$
then both $\nu_i$ and $c_i$ are equal to 0 (this is a straightforward 
consequence of the definitions of $\SS$ and of an $\a_K$-acceptable vector).  
For $j\in J\cap I_0 = K\cap I_0$ the only condition satisfied by $\nu_j$
and $c_j$ is that they should belong to $\{0,1\}$.
If $j\in (J\setminus J')\cap I_1$ then $j$ must have two neighbours 
$j'$ and $j''$ in $J\cap I_0$, and the condition on $v_j$ is
$0\le \nu_j \le \max\{0, \nu_{j'}+\nu_{j''}-1\}$, while the condition on $c_j$
is $0\le c_j \le \min\{c_{j'}, c_{j''}\}$. 
Clearly, since $c_{j'}$ and $c_{j''}$ belong to $\{0,1\}$, the 
second condition can be rewritten $0\le c_j \le \max\{0, c_{j'}+c_{j''}-1\}$.
If $j\in K'$ then $j$ has a unique neighbour $j'\in J$, and the 
condition on $\nu_j$ is $0\le \nu_j \le \nu_{j'}$ while the condition
on $c_j$ is $0\le c_j \le c_{j'}$. Finally, if $j\in J'$ then 
$j$ has a unique neighbour $j'\in J$, and the 
condition on $\nu_j$ is $0\le \nu_j \le \max\{0,-1+v_{j'}\}$ while the condition
on $c_j$ is $c_j=0$. Clearly the condition on $\nu_j$ forces $\nu_j=0$.
In conclusion, the exponents $\nu$ and the vectors $\ga$ must satisfy
the same conditions, and the theorem follows from Proposition~\ref{prop33}
and \ref{ssmfree}.
\cqfd
\end{proof}

\section{A tensor product theorem}
\label{sectensprod}
\subsection{}
In a cluster algebra $\A$, a product $y_1\cdots y_k$ of 
cluster variables is a cluster monomial if and only if
for every $1\le i<j\le k$ the product $y_iy_j$ is a cluster monomial.  
In this section we show the following theorem, which is consistent
with our categorification Conjecture~\ref{mainconjecture}~{(ii)},
and will be needed to prove it.
\begin{theorem}\label{thtensprod}
Let $S_1,\ldots , S_k$, be simple objects of $\CC_1$.
Suppose that
$S_i\otimes S_j$ is simple for every $1\le i<j\le k$.
Then $S_1\otimes \cdots\otimes S_k$ is simple. 
\end{theorem}
Note that we may have $S_i \cong S_j$ for some $i<j$, in which
case our assumption includes the simplicity of $S_i^{\otimes 2}$.
The proof will be given in \ref{proofth}.

A similar result for a special class of modules of the Yangian of $\gl_n$
attached to skew Young diagrams was given by Nazarov and Tarasov \cite{NT}.

\subsection{}
We need to recall some results about tensor products
of objects of $\CC$.
Let $\De$ be the comultiplication of $U_q(\hg)$.
In general one does not know explicit formulae for 
calculating $\De$ in terms of the Drinfeld generators.
However, the next proposition contains partial information
which will be sufficient for our purposes.
We have a natural grading of $U_q(\hg)$ by the
root lattice $Q$ of $\g$ given by
\[
\deg(x^\pm_{i,r}) = \pm\a_i,\quad
\deg(h_{i,s})=\deg(k_i^\pm) = \deg(c^\pm) = 0,
\qquad
(i\in I,\ r\in\Z,\ s\in\Z\setminus \{0\}). 
\]
Let $U^+$ (\resp $U^-$) be the subalgebra of $U_q(\hg)$
consisting of elements of positive 
(\resp negative) $Q$-degree.
\begin{proposition}{\rm \cite[Prop. 7.1]{D}}\label{Dami}
For $i\in I$ and $r>0$ we have
\[
\De(h_{i,r}) - h_{i,r}\otimes 1 - 1\otimes h_{i,r}
\in U^-\otimes U^+. 
\]
\end{proposition}

The next result is about tensor products of fundamental modules.
For a fundamental module $L(Y_{i,r}) = V_{i,q^r}$ in $\CC_\Z$
we denote by $u_{i,r}$ a highest weight vector (it is unique
up to rescaling by a non-zero complex number).

\begin{theorem}{\rm\cite{Cha2,K,VV}}\label{thK}
Let $r_1,\ldots, r_s$ be integers with $r_1\le \cdots \le r_s$.
Then, for any $i_1,\ldots,i_s$, the tensor product 
of fundamental modules $L(Y_{i_s,r_s})\otimes \cdots \otimes L(Y_{i_1,r_1})$
is a cyclic module generated by the tensor product
of highest weight vectors
$u_{i_s,r_s}\otimes \cdots \otimes u_{i_1,r_1}$.
Moreover, there is a unique homomorphism 
\[
\phi : L(Y_{i_s,r_s})\otimes \cdots \otimes L(Y_{i_1,r_1})
\to L(Y_{i_1,r_1})\otimes \cdots \otimes L(Y_{i_s,r_s}) 
\]
with $\phi(u_{i_s,r_s}\otimes \cdots \otimes u_{i_1,r_1})
= u_{i_1,r_1}\otimes \cdots \otimes u_{i_s,r_s}$, and
its image is the simple module $L(Y_{i_1,r_1}\cdots Y_{i_s,r_s})$.
\end{theorem}

\begin{example}
{\rm (\cf Example~\ref{exsl2}.)
Take $\g=\Sl_2$ and consider the tensor product $L(Y_0)\otimes L(Y_2)$.
We have 
\[
L(Y_0)=\C u_0 \oplus \C v_0,\qquad
L(Y_2)=\C u_2 \oplus \C v_2,
\]
where $u_0$ and $u_2$ are the highest weight vectors, and 
$
v_0=x_0^-u_0,\ v_2=x^-_0u_2.
$ 
Each of these four vectors spans a one-dimensional $U_q(\Sl_2)$-weight-space,
hence also an {\em l}-weight-space, and we have
\[
h_r u_0 = \frac{[r]_q}{r}u_0,\quad
h_r u_2 = \frac{[r]_q}{r}q^{2r}u_2,\quad
h_r v_0 = -\frac{[r]_q}{r}q^{2r}v_0,\quad
h_r v_2 = -\frac{[r]_q}{r}q^{4r}v_2,\quad 
\]
where $[r]_q = (q^r-q^{-r})/(q-q^{-1})$.
By Theorem~\ref{thK}, the submodule of $L(Y_0)\otimes L(Y_2)$ generated
by $u_0\otimes u_2$ is the three-dimensional simple module $L(Y_0Y_2)$, with
basis
\[
u_0\otimes u_2,\quad 
x_0^-(u_0\otimes u_2) = u_0\otimes v_2 + q v_0\otimes u_2,\quad
v_0\otimes v_2.
\]
The $U_q(\Sl_2)$-weight-spaces of $L(Y_0Y_2)$ are also one-dimensional and 
coincide with its {\em l}-weight-spaces.
By Proposition~\ref{Dami} we have
\[
h_r(x^-_0(u_0\otimes u_2)) = \frac{[r]_q}{r}(1-q^{4r})u_0\otimes v_2 + \l v_0\otimes u_2
\]
for some $\l\in\C$.
Hence $u_0\otimes v_2 + q v_0\otimes u_2$ has {\em l}-weight $Y_0Y_4^{-1}$,
which is the product of the {\em l}-weights of $u_0$ and $v_2$.
On the other hand, again by Proposition~\ref{Dami}, we have $h_r(v_0\otimes u_2)=0$
for every $r>0$, hence $v_0\otimes u_2$ is an {\em l}-weight-vector 
with {\em l}-weight $1$. This shows that $u_0\otimes v_2$
is {\em not} an {\em l}-weight vector.
}
\end{example}

\subsection{}
To an object $V$ of $\CC_1$ which is generated
by a highest weight vector $v$, we attach as in \ref{secttrunc}
a truncated {\em below} $q$-character $\chi_q(V)_{\ge 3}$
as follows.
We have $\chi_q(V)=m(1+\sum_p M^{(p)})$, where $m$ is the {\em l}-weight
of $v$, and the $M^{(p)}$ are monomials in the 
$A_{i,r}^{-1}\ (i\in I,\ r\in\Z_{>0})$.
Define
\[
\chi_q(V)_{\ge 3}\ := m\left(1+\sum_p{}^* M^{(p)}\right) 
\]
where $\sum^*$ means the sum restricted to the $M^{(p)}$
which have no variable 
$A_{i,1}^{-1},\ A_{i,2}^{-1}\ (i\in I)$.
We shall also denote by $V_{\ge 3}$ the subspace of $V$ 
obtained by taking the direct sum of the corresponding {\em l}-weight-spaces.

\subsection{}
Let $S$ be a simple object in $\CC_1$ and let 
$m = \prod_{i\in I} Y_{i,\xi_i}^{a_i}Y_{i,\xi_i+2}^{b_i}$
be its highest weight monomial.
We define
\[
m^-=\prod_{i\in I} Y_{i,\xi_i}^{a_i},\quad
S^-=L(m^-),\quad
m^+=\prod_{i\in I} Y_{i,\xi_i+2}^{b_i},\quad 
S^+=L(m^+).
\]
Then $S^-$ is a simple object of the subcategory $\CC_0$, and
therefore 
$S^- \cong \bigotimes_{i\in I} L(Y_{i,\xi_i})^{\otimes a_i}$
is a tensor product of fundamental modules (\cf Example~\ref{exC0}).
Similarly, 
$S^+ \cong \bigotimes_{i\in I} L(Y_{i,\xi_i+2})^{\otimes b_i}$.
Applying Theorem~\ref{thK}, we get a surjective homomorphism
$\phi_S : S^+\otimes S^- \twoheadrightarrow S \subset S^-\otimes S^+$.

\begin{lemma}\label{lemtp}
Let $u^-$ be a highest weight vector of $S^-$.
We have $\phi_S^{-1}(S_{\ge 3}) = S^+\otimes u^-$
and $\phi_S$ restricts to a bijection from
$S^+\otimes u^-$ to $S_{\ge 3}$. 
\end{lemma}
\begin{proof}
First we have 
$\chi_q(S^+\otimes S^-)_{\ge 3} = \chi_q(S^+)_{\ge 3}\ \chi_q(S^-)_{\ge 3}$.
Then by Proposition \ref{FMmon} we have $\chi_q(S^+)_{\ge 3} = \chi_q(S^+)$
and $\chi_q(S^-)_{\ge 3} = m^-$. So $\chi_q(S^+\otimes S^-)_{\ge 3} = \chi_q(S^+)m^-$.
Then we note that
$\chi_q(S)_{\ge 3}=m^-\chi_q(S^+)$.
Indeed, the inequality $\chi_q(S)_{\ge 3}\ \le m^-\chi_q(S^+)$
follows from the factorization $m=m^-m^+$.
On the other hand, $S^+$ is minuscule by \S~\ref{cons2}, 
so $\chi_q(S^+) = \FM(m^+)$.
By using inductively \ref{belongJ} for $\chi_q(S)$ we get
the converse inequality.
Hence we get that $\chi_q(S^+\otimes S^-)_{\ge 3} = \chi_q(S)_{\ge 3}$.
Since $\phi_S$ is a homomorphism of $U_q(\hg)$-modules, it preserves
{\em l}-weight-spaces, hence it follows that it restricts to a
vector space isomorphism from $(S^+\otimes S^-)_{\ge 3}$
to $S_{\ge 3}$. 
Now since $u^-$ is a highest-weight vector of $S^-$,
Proposition~\ref{Dami}
shows that for every {\em l}-weight vector $v\in S^+$, 
$v\otimes u^-$ is an {\em l}-weight vector of $S^+\otimes S^-$
with {\em l}-weight equal to the product of the {\em l}-weight
of $v$ by $m^-$. This implies that 
$S^+\otimes u^-\subset (S^+\otimes S^-)_{\ge 3}$.
Hence, since these two vector spaces have the same 
$q$-character, they are equal and the lemma is proved. 
\cqfd 
\end{proof}

\subsection{}\label{proofth}
We can now give the proof of Theorem~\ref{thtensprod}.
For $i=1,\ldots,k$, define $m_i^\pm$, 
$S_i^\pm$, $u_i^\pm$, $\phi_{S_i}$ as above.
First consider a pair $1\le i<j \le k$.
Note that by Example~\ref{exC0} and Section~\ref{cons2}, 
$S_i^\pm \otimes S_j^\pm$ is minuscule and simple.
Now, since $S_i\otimes S_j$ is simple, and 
$S_j^+\otimes S_i^+\otimes S_i^-\otimes S_j^-$
can be written as a product of fundamental modules
with non-increasing spectral parameters, 
by Theorem~\ref{thK}, we have a surjective homomorphism
$\phi : S_j^+\otimes S_i^+\otimes S_i^-\otimes S_j^- \twoheadrightarrow S_i\otimes S_j$,
which is unique up to a scalar multiple.
Using three times Theorem~\ref{thK}, this homomorphism, composed with the inclusion $S_i\otimes S_j\subset S_i^-\otimes S_i^+\otimes S_j^-\otimes S_j^+$, can be factored as follows
\[
S_j^+\otimes (S_i^+\otimes S_i^-)\otimes S_j^- 
\to 
(S_j^+\otimes S_i^-)\otimes S_i^+\otimes S_j^- 
\to 
S_i^-\otimes (S_j^+\otimes S_i^+)\otimes S_j^- 
\cong
\]
\[
\cong
S_i^-\otimes S_i^+\otimes (S_j^+\otimes S_j^-)
\to
S_i^-\otimes S_i^+\otimes S_j^-\otimes S_j^+.
\]
The map from the second module to the fourth one can be written
as $\a \otimes {\rm id}_{S_j^-}$, where
\[ 
\a : S_j^+\otimes S_i^-\otimes S_i^+ \to S_i^-\otimes S_i^+\otimes S_j^+
\]
restricts to a homomorphism $\bar{\a}$ from $S_j^+\otimes S_i$ to
$S_i\otimes S_j^+$.  
Since $\phi$ is surjective, by applying Lemma~\ref{lemtp} to $S=S_j$,
we get that ${\rm im} (\bar{\a}) \otimes u^-_j$ contains $S_i\otimes S_j^+\otimes u^-_j$,
hence $\bar{\a}$ is surjective. 
To summarize, we have obtained the existence of a unique homomorphism
$S_j^+\otimes S_i \to S_i\otimes S_j^+$ mapping $u_j^+\otimes u_i$
to $u_i\otimes u_j^+$, which is surjective.  

Now we proceed by induction on $k\ge 2$ to prove Theorem \ref{thtensprod}.
For $k=2$ there is nothing to prove, so take $k\ge 3$
and assume that $S_1\otimes \cdots \otimes S_{k-1}$ is simple.
By Theorem~\ref{thK} we obtain a surjective homomorphism
\[
S_k^+\otimes(S_1^+\otimes \cdots \otimes S_{k-1}^+)
\otimes
(S_1^-\otimes \cdots \otimes S_{k-1}^-)\otimes S_k^-
\twoheadrightarrow
S_k^+\otimes(S_1\otimes \cdots \otimes S_{k-1})
\otimes S_k^-.
\]
Using the above result, we then have a sequence of surjective 
homomorphisms
\[
S_k^+\otimes(S_1\otimes \cdots \otimes S_{k-1})
\twoheadrightarrow
S_1\otimes S_k^+\otimes S_2\otimes\cdots \otimes S_{k-1}
\twoheadrightarrow
\cdots
\twoheadrightarrow
(S_1\otimes \cdots \otimes S_{k-1})\otimes S^+_k, 
\]
hence
\[
S_k^+\otimes(S_1^+\otimes \cdots \otimes S_{k-1}^+)
\otimes
(S_1^-\otimes \cdots \otimes S_{k-1}^-)\otimes S_k^-
\twoheadrightarrow
(S_1\otimes \cdots \otimes S_{k-1})\otimes S^+_k
\otimes S^-_k
\twoheadrightarrow
S_1\otimes \cdots \otimes S_{k}. 
\]
Thus we have shown that $V := S_1\otimes \cdots \otimes S_{k}$ is 
a quotient of a tensor product of fundamental modules, which
by Theorem~\ref{thK} 
is a cyclic module generated by its highest weight vector. 
Hence $V$ has no proper submodule containing its highest
weight-space.

We can now conclude, as in \cite[\S 4.10]{CP2}, by considering
the dual module $V^* = S_k^*\otimes \cdots \otimes S_{1}^*$.
By our assumption, $S_j^*\otimes S_i^* \cong (S_i \otimes S_j)^*$
is simple for every $1\le i<j \le k$. Moreover, by \ref{sect23}, 
the modules $S_i^*$ belong to a
category defined like $\CC_1$ except for a shift of all
spectral parameters by $q^{-h}$.
Thus, by using the same proof, 
$V^*$ also has no
proper submodule containing its highest weight-space.
But if $W$ was a proper submodule of $V$ not containing 
its highest weight-space, then the annihilator $W^{\circ}$
of $W$ in $V^*$ would be a proper submodule of $V^*$ containing its lowest
weight-space. Let
$\lambda_i$ (resp. $\mu_i$) be the lowest (resp. highest)
weight of $S_i^*$ considered as a $U_q(\g)$-module. 
Then $\lambda = \lambda_1 + \cdots + \lambda_n$
(resp. $\mu = \mu_1+\cdots +\mu_n$)
is the lowest (resp. highest) weight of $V^*$.
In the direct sum decomposition of $V^*$
as a $U_q(\g)$-module, there is a unique simple module $L$
with lowest weight $\lambda$, and by our assumption, $L$
is contained in $W^{\circ}$.
By \cite[Proposition 5.1 (b)]{CP4prime} (see also \cite[Theorem 1.3 (3)]{FM}), 
$\mu$ is the highest weight of $L$, and therefore $W^{\circ}$ must also contain the
highest weight-space of $V^*$, which is impossible.
This finishes the proof of Theorem~\ref{thtensprod}. \cqfd

\section{Cluster expansions}\label{sectclusterexp}

\subsection{}
Following \cite{FZ1, FZ2}, we say that two roots $\a, \b\in \Phi_{\ge -1}$
are {\em compatible} if the cluster variables $x[\a]$ and $x[\b]$ belong 
to a common cluster of $\A$. Let $\ga$ be an element of the root lattice~$Q$.
A {\em cluster expansion} of $\ga$ is a way to express $\ga$
as
\begin{equation}\label{clustexp}
 \ga = \sum_{\a\in\Phi_{\ge -1}} n_\a \a,
\end{equation}
where all $n_\a$ are nonnegative integers, and $n_\a n_\b = 0$
whenever $\a$ and $\b$ are not compatible.
In plain words, a cluster expansion is an expansion into a
sum of pairwise compatible roots in $\Phi_{\ge -1}$.

\begin{theorem}{\em \cite[Th. 3.11]{FZ1}}\label{thFZcl}
Every element of the root lattice has a unique cluster expansion. 
\end{theorem}

\subsection{}\label{clusterexpa}
For $\ga=\sum_i c_i\a_i \in Q$, define
\[
Y^\ga := \prod_{c_i<0} Y_{i,2-\xi_i}^{|c_i|} \prod_{c_i>0} Y_{i,\xi_i-\eps_i+1}^{c_i}  
\in \M_+.
\]
This definition is such that for $\a\in\Phi_{\ge -1}$, the monomial $Y^\a$
is the highest $l$-weight of $S(\a)$.

Let $S$ be a simple object in $\CC_1$. Its highest $l$-weight
is of the form
\[
m = \prod_{i\in I} Y_{i,\xi_i}^{a_i} Y_{i,\xi_i+2}^{b_i}
\]
for some nonnegative integers $a_i, b_i$.
Clearly we can write $m$ as
\[
m=  \left(\prod_{i\in I} (Y_{i,\xi_i}Y_{i,\xi_i+2})^{\min(a_i,b_i)}\right)Y^\ga,
\]
for some unique $\ga\in Q$.
Hence using Theorem~\ref{thFZcl} and (\ref{clustexp}), we have a unique factorization
\[
m= \prod_{i\in I} (Y_{i,\xi_i}Y_{i,\xi_i+2})^{\min(a_i,b_i)}
\prod_{\a\in\Phi_{\ge -1}} (Y^\a)^{n_\a},
\]
where the $\a$ for which $n_\a>0$ are pairwise compatible.
Suppose that Conjecture~\ref{mainconjecture}~{(i)}
is established.
Since $S(\a)=L(Y^\a)$ and $F_i=L(Y_{i,\xi_i}Y_{i,\xi_i+2})$, 
to prove Conjecture~\ref{mainconjecture}~{(ii)}
we then need to prove that
\[
L(m) \cong  \bigotimes_{i\in I} F_i^{\otimes\min(a_i,b_i)}
\bigotimes_{\a\in\Phi_{\ge -1}} S(\a)^{\otimes n_\a}.
\]
Given that the highest $l$-weights of the two sides coincide,
this amounts to prove that the tensor product of the right-hand side 
is a simple module. 
Because of Theorem~\ref{thtensprod}, this will be the case if
and only if every pair of factors has a simple tensor product.
Thus, assuming Conjecture~\ref{mainconjecture}~{(i)}, the
proof of Conjecture~\ref{mainconjecture}~{(ii)} reduces 
to check that

\begin{itemize}
\item[(i)] $F_i\otimes S(\a)$ and $F_i\otimes F_j$ are simple for every $\a\in\Phi_{\ge -1}$
and $i,j\in I$; 
\item[(ii)] if $\a,\b\in\Phi_{\ge -1}$ are compatible,
then $S(\a)\otimes S(\b)$ is simple;
\item[(iii)] for every $\a\in\Phi_{>0}$, $S(\a)$ is prime.
\end{itemize} 
Note that the simple objects $F_i$ and $S(-\a_i)$ are clearly prime. 
Indeed, if $F_i$ was not prime we could only have $F_i \cong S(\a_i)\otimes S(-\a_i)$,
in contradiction with the $T$-system
\[
 [S(\a_i)\otimes S(-\a_i)] = [F_i] + \prod_{j\,:\,a_{ij}=-1} [S(-\a_j)].  
\]

\section{Type $A$}\label{secttypA}

In this section we assume that $\g$ is of type $A_n$.
The vertices of the Dynkin diagram are labelled by 
the interval $I=[1,n]$ in linear order.

\subsection{}\label{proofA(i)}
The positive roots are all multiplicity-free, labelled by the subintervals of $I$.
For $[i,j]\subset I$, we denote by $\a_{[i,j]} := \sum_{k=i}^j \a_k$
the corresponding positive root.
By Theorem~\ref{multfreeequal}, Eq.~(\ref{eq2prove}) is verified for all
positive roots. This and (\ref{eqatrtr}) proves Conjecture~\ref{mainconjecture}~{(i)} 
in type $A$.

\subsection{}\label{proofA(ii)}
We will now prove \ref{clusterexpa}~(i) and (ii). Note that \ref{clusterexpa}~(iii)
follows from Corollary~\ref{corprime}. 

\subsubsection{}\label{proof92i}
We first dispose of (i), which is easy.
Indeed, by Example~\ref{exa31},
$\chi_q(F_i\otimes F_j)_{\le 2}$ is equal to a single dominant monomial
hence $F_i\otimes F_j$ is simple.
Next, we use the fact that $\chi_q(S(\a))_{\le 2}$ is given
by the Frenkel-Mukhin algorithm (see Example~\ref{exa31} and Corollary~\ref{correg}). 
Thus, if $m$ is the highest $l$-weight of $S(\a)$, we have 
\[
\FM(Y_{i,\xi_i}Y_{i,\xi_i+2}m)_{\le 2} = Y_{i,\xi_i}Y_{i,\xi_i+2}\FM(m)_{\le 2}
= Y_{i,\xi_i}Y_{i,\xi_i+2}\chi_q(L(m))_{\le 2}.
\]
Therefore $\chi_q(L(Y_{i,\xi_i}Y_{i,\xi_i+2}m))$
contains $\chi_q(F_i\otimes S(\a))_{\le 2}$, and
$F_i\otimes S(\a)$ is simple.

\subsubsection{}
To describe explicitly the pairs of compatible roots
we are going to use the geometric model of \cite[\S 3.5]{FZ1}.
The set $\Phi_{\ge-1}$ has cardinality $n(n+1)/2 + n = n(n+3)/2$.
We identify the elements of $\Phi_{\ge -1}$ with the diagonals
of a regular convex $(n+3)$-gon $\P_{n+3}$ as follows. 
Let $1,\ldots,n+3$  be the vertices of $\P_{n+3}$,
labelled counterclockwise. The negative simple roots are
identified with the following diagonals:
\[
-\a_k \equiv
\left\{
\begin{array}{ll}
[i,\ n+3-i] &\mbox{if $k=2i-1$,}\\[3mm]
[i+1,\ n+3-i] &\mbox{if $k=2i$.}
\end{array}
\right.
\]
These diagonals form a ``snake'', as shown in Figure~\ref{snake}.
\begin{figure}
\begin{center}
\leavevmode
\epsfxsize =8cm
\epsffile{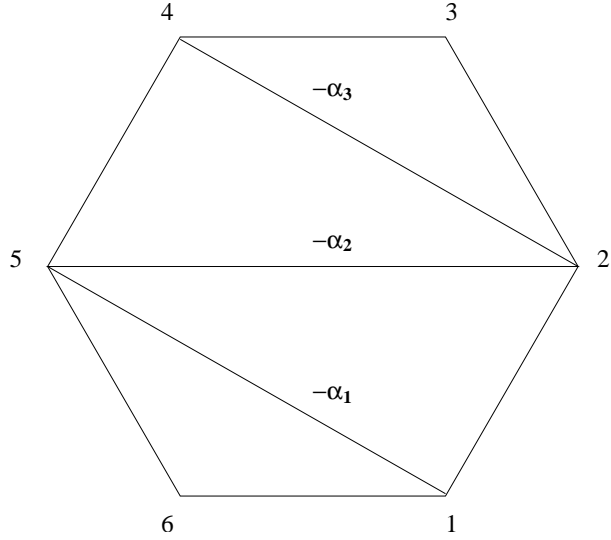}
\end{center}
\caption{\small \label{snake}
{\it The snake in type $A_3$.}}
\end{figure} 
To identify the remaining diagonals (not belonging to the snake)
with positive roots, we associate each $\a_{[i,j]}$ with the unique
diagonal that crosses the diagonals $-\a_i,\,-\a_{i+1},\,\ldots,\,-\a_j$
and does not cross any other diagonal $-\a_k$ of the snake. 
Under this identification the roots $\a$ and $\b$ are compatible
if and only if they are represented by two non-crossing diagonals.
Hence the clusters are in one-to-one correspondence with the
triangulations of $\P_{n+3}$.
\begin{example}
{\rm
Take $n=3$.
The identification of the negative simple roots with the snake diagonals
of a hexagon is shown in Figure~\ref{snake}.
The positive roots are represented by the remaining diagonals as
follows:
\[ 
\a_1 \equiv [2,6],\ 
\a_2 \equiv [1,4],\
\a_3 \equiv [3,5],\ 
\a_1+\a_2 \equiv [4,6],\ 
\a_2+\a_3 \equiv [1,3],\ 
\a_1+\a_2+\a_3 \equiv [3,6].\  
\]
The 14 clusters listed in Example~\ref{exa14cl} correspond to the 14 
triangulations of the hexagon.
} 
\end{example}

\subsubsection{} \label{proof92iiA}
It follows from the geometric model that all
pairs of compatible roots $(\a,\b)$ of $\Phi_{\ge -1}$ are 
of one of the following forms:
\begin{itemize}
\item[(a)] $\a= -\a_i$, $\b = -\a_j$;
\item[(b)] $\a= -\a_i$, $\b\in\Phi_{>0}$ with $[\b : \a_i]=0$;
\item[(c)] $\a=\a_{[i,j]}$, $\b=\a_{[k,l]}$ with $k > j+1$ (two disjoint intervals);
\item[(d)] $\a=\a_{[i,j]}$, $\b=\a_{[i,l]}$ or $\a=\a_{[i,k]}$, $\b=\a_{[j,k]}$ 
(two intervals with a common end);
\item[(e)] $\a=\a_{[i,k]}$, $\b=\a_{[j,l]}$ with $i<j\le k<l$ and $k-j$ even
(two overlapping intervals);
\item[(f)] $\a=\a_{[i,l]}$, $\b=\a_{[j,k]}$ with $i<j< k<l$ and $k-j$ odd
(one interval strictly contained into the other).
\end{itemize}
Let us prove that $S(\a)\otimes S(\b)$ is simple in all these cases.

\medskip
(a) This follows from Example~\ref{exC0}. Indeed, $S(-\a_i)$ and $S(-\a_j)$
are fundamental modules in a shift of the subcategory $\CC_0$.

\medskip
(b) Write $\b=\a_J$. Let $m=Y_{i,\xi_i+1}\prod_{j\in J\cap I_0}Y_{j,0}\prod_{j\in J\cap I_1}Y_{j,3}$
be the highest weight monomial of $\chi_q(S(-\a_i)\otimes S(\b))$.
Since by assumption $i\not \in J$,
it is easy to check that $\FM(m)$ contains the product
\[
\chi_q(L(Y_{i,\xi_i+1}))_{\le 2}\,\chi_q(S(\b))_{\le 2}
\ =\ \chi_q(S(-\a_i)\otimes S(\b))_{\le 2}
\]
given by Example~\ref{exa31} and Proposition~\ref{prop33}.
Hence, by \ref{sect23}, it contains 
in particular $\chi_q(L(m))_{\le 2}$.
By Corollary~\ref{corregular}, it follows that $\FM(m)=\chi_q(L(m))$.
Thus, 
$\chi_q(S(-\a_i)\otimes S(\b))=\chi_q(L(m))$, and $S(-\a_i)\otimes S(\b)$
is irreducible.

\medskip
(c) Same reasoning as (b).

\medskip
(d) We assume that $i\le j\le l$ and argue by induction on $s=l-i$.
If $s=0$, we have to check that $S(\a_i)\otimes S(\a_i)$ is simple.
This is clear, since if we choose $\xi_i=1$ then 
$\chi_q(S(\a_i))_{\le 2}= Y_{i,3}$  and its square contains a single
dominant monomial.
Suppose the claim is proved when $l-i=s$, and take $l-i=s+1$.
Let $m_{[i,j]}$ and $m_{[i,l]}$ denote the highest weight monomials
of $S(\a_{[i,j]})$ and $S(\a_{[i,l]})$.
Choose again $\xi_i=1$. 
Then, by Proposition~\ref{prop33}, the truncated $q$-characters
of $S(\a_{[i,j]})$ and $S(\a_{[i,l]})$ do not contain any $A_{k,\xi_k+1}^{-1}$ 
with $k\le i$. So we have
\[
\chi_q(S(\a_{[i,j]}))_{\le 2} = \vp_{[i+1,j+1]}(m_{[i,j]})_{\le 2},
\qquad
\chi_q(S(\a_{[i,l]}))_{\le 2} = \vp_{[i+1,l+1]}(m_{[i,l]})_{\le 2}.
\]
Hence $\chi_q(S(\a_{[i,j]})\otimes S(\a_{[i,l]}))_{\le 2} = \vp_{[i+1,l+1]}(m_{[i,j]}m_{[i,l]})_{\le 2}$. 
So by using the induction hypothesis and \ref{belongJ}, we get
\[
\chi_q(S(\a_{[i,j]})\otimes S(\a_{[i,l]}))_{\le 2} \le \chi_q(L(m_{[i,j]}m_{[i,l]}))_{\le 2}. 
\]
This shows that $S(\a_{[i,j]})\otimes S(\a_{[i,l]})\cong L(m_{[i,j]}m_{[i,l]})$ is 
simple.
The case of $S(\a_{[i,j]})\otimes S(\a_{[k,j]})$ with $1\le k\le j$ is similar.

\medskip
(e) (f) 
Let $m=m_{[i,k]}m_{[j,l]}=m_{[i,l]}m_{[j,k]}$.
We want to show that 
\begin{equation}\label{eq2prov}
\chi_q(L(m))_{\le2} =
\left\{
\begin{array}{ll}
\chi_q(L(m_{[i,k]}))_{\le2}\ \chi_q(L(m_{[j,l]}))_{\le2} & \mbox{if $k-j$ is odd,}\\[2mm]
\chi_q(L(m_{[i,l]}))_{\le2}\ \chi_q(L(m_{[j,k]}))_{\le2} & \mbox{if $k-j$ is even.} 
\end{array}
\right.
\end{equation}
To do this, it is enough to show that all the monomials in the right-hand
side occur in the left-hand side with the same multiplicity. 
We argue by induction on $l-i\ge 2$.
For $l-i=2$ we are necessarily in the first case with $j=k=i+1=l-1$.
Let us choose $\xi_i=1$. Then $\xi_l=1$, and by Proposition~\ref{prop30}
\[
\chi_q(L(m_{[i,k]}))_{\le2}\ 
=\vp_{[i,l]}(m_{[i,k]})_{\le 2},
\qquad
\chi_q(L(m_{[j,l]}))_{\le2}
= \vp_{[i,l]}(m_{[i,k]})_{\le 2}.
\]
Thus we can assume for simplicity of notation that $I=[i,l]=[1,3]$.
Writing for short
$v_m=A_{m,\xi_m+1}^{-1}$, we have
\[
\chi_q(L(m_{[1,2]}))_{\le2}\ \chi_q(L(m_{[2,3]}))_{\le2}
=
m(1+v_2+v_2v_3)(1+v_2+v_1v_2). 
\]
If a monomial in the right-hand side does not contain $v_1$,
then it belongs to
\[
\vp_{[2,3]}(m_{[1,2]})_{\le2}\ \vp_{[2,3]}(m_{[2,3]})_{\le2}
\ = \
\vp_{[2,3]}(m)_{\le2}
\ \le \ \chi_q(L(m))
\]
(the equality follows from (d)).
For the same reason, if a monomial does not contain $v_3$,
it belongs to $\chi_q(L(m))$.
The only monomial in the right-hand side which contains both $v_1$
and $v_3$ is $m'=mv_1v_2^2v_3$.
Now $m'':=mv_2^2v_3= Y_{1,1}^2Y_{1,3}Y_{2,2}^{-1}Y_{3,1}$ belongs to $\chi_q(L(m))$ by what 
we have just said, and it is $1$-dominant.
We have $\vp_1(m'')_{\le 2} = m''+m'$, thus, by \ref{belong}, $m'$ 
also belongs to $\chi_q(L(m))$.
This proves (\ref{eq2prov}) when $l-i=2$. 

Assume now that (\ref{eq2prov}) holds when $l-i=s$ and take $l-i=s+1$.
Choose $\xi_i=1$.
Using again Proposition~\ref{prop30}, we can also assume
without loss of generality that $I=[i,l]=[1,s+1]$.
Arguing as in the case $l-i=2$ and using $(c)$ or $(d)$
or the induction hypothesis, we see that every monomial $m'$ occuring
in the right-hand side of (\ref{eq2prov}) occurs in $\chi_q(L(m))$
with the same multiplicity if $m'm^{-1}$
does not contain $v_1v_2\cdots v_{s+1}$.
So we only need to consider the monomials $m'=mv_1^{c_1}\cdots v_{s+1}^{c_{s+1}}$
with $c_k>0$ for every $k\in I$.
For such a monomial $m'$, 
it can be checked using Proposition~\ref{prop33} that 
there exists $k$ with $c_k=2$, and that the smallest
such $k$ belongs to $I_0$. 
Let us denote it by $k_{\min}$.
Then Proposition~\ref{prop33} shows that
$m'':=m'/v_{k_{\min}-1}$ appears in the right-hand side of (\ref{eq2prov}),
and since $m''$ does not contain $v_{k_{\min}-1}$, by what we 
said above, $m''$ also appears in $\chi_q(L(m))$ with the same multiplicity.
Now $m''$ is $(k_{\min}-1)$-dominant
and $\vp_{k_{\min}-1}(m'')=m''+m'$ is contained in $\chi_q(L(m))$
by \ref{belong}. Moreover, the multiplicities of $m'$ and $m''$ in $\vp_{k_{\min}-1}(m'')$
are the same. Hence, their multiplicities are also the same on both sides of (\ref{eq2prov}).
This finishes the proof of (e) and (f).

\medskip
This finishes the proof of \ref{clusterexpa}~(ii) in type $A$.
Hence Conjecture~\ref{mainconjecture}~{(ii)} is proved in type $A$.

\subsection{}
The truth of Conjecture~\ref{mainconjecture} in type $A$ has the 
following interesting consequence, which will be used to validate
Conjecture~\ref{mainconjecture}~(i) in type $D$.

Let $\ga=\sum_i c_i\a_i$, with $c_i\ge 0$.
We define $Y^\ga\in \M_+$ as in \ref{clusterexpa}.
Write $\tau_-(\gamma)=\de=\sum_i d_i\a_i$.

\begin{corollary}\label{corA}
If $0\le d_i \le 2$ for every $i\in I$, the truncated $q$-character
of $L(Y^\ga)$ is equal to
\[
\chi_q(L(Y^\ga))_{\le 2} = Y^\ga\, F_\de\left(A_{1,\xi_1+1}^{-1},\ldots,A_{n,\xi_n+1}^{-1}\right), 
\]
where the $F$-polynomial $F_\de$ is given by the explicit formula~(\ref{eqFpol}). 
\end{corollary}
\begin{proof}
Let $\ga = \sum_{\a \in \Phi_{>0}}n_\a \a$ be the cluster expansion 
of $\gamma$.
By \ref{proofA(ii)}, we have
\[
\chi_q(L(Y^\ga))_{\le 2} = \prod_{\a\in\Phi_{>0}} \chi_q(S(\a))_{\le 2}^{n_\a}
= Y^\ga \prod_{\a\in\Phi_{>0}} F_{\tau_-(\a)}(A_{i,\xi_i+1}^{-1})^{n_\a}. 
\]
Since $\tau_-$ is linear on $\oplus_i \N\a_i$, we have 
$\tau_-(\ga)=\sum_{\a\in\Phi_{>0}}n_\a\tau_-(\a) = \de$, and this is the
cluster expansion of $\de$.
Now if we define 
$F_\de(v_1,\ldots,v_n) := \prod_{\a\in\Phi_{>0}} F_{\tau_-(\a)}(v_1,\ldots,v_n)^{n_\a}$,
the proof of Proposition~\ref{propFZ} given in \cite{FZ1} shows that,
since $\de$ is 2-restricted, $F_\de$ can still be calculated by formula~(\ref{eqFpol}).
Indeed, the two main steps (Lemmas 2.11 and 2.12) are proved for a 2-restricted
vector which is not necessarily a root. 
\cqfd 
\end{proof}

\begin{example}\label{ex121}
{\rm
Let $\g$ be of type $A_3$ and take $I_0=\{2\}$ and $I_1=\{1,3\}$.
Choose $\ga=\a_1+2\a_2+\a_3$. Then $\tau_{-}(\ga)=\ga$.
Writing $v_i = A_{i,\xi_i+1}^{-1}$, we have
\[
\chi_q(L(Y_{1,3}Y_{2,0}^2Y_{3,3}))_{\le 2}
\ =\  Y_{1,3}Y_{2,0}^2Y_{3,3}
(1+2v_2+v_2^2+v_1v_2+v_2v_3+v_1v_2^2+v_2^2v_3+v_1v_2^2v_3),
\]
where the monomials of the right-hand side are given by the 
combinatorial rule of Proposition~\ref{propFZ}.
} 
\end{example}

\section{Type $D$}\label{secttypD}

In this section we take $\g$ of type $D_n$,
and we label the Dynkin diagram as in \cite{B}.
 
\subsection{}
Let $\b\in\Phi_{>0}$ and let $\a=\tau_-(\b)$.
We want to prove
\begin{theorem}\label{th2proveD}
The normalized truncated $q$-character of $S(\b)$ is equal to
\[
\widetilde{\chi}_q(S(\b))_{\le 2} \ = \ F_\a(v_1,\ldots,v_n), 
\]
where we write for short $v_i = A_{i,\xi_i+1}^{-1}\ (i\in I)$,
and the $F$-polynomial $F_\a$ is given by the combinatorial
formula of Proposition~\ref{propFZ}.
\end{theorem}

We assume that the trivalent node $n-2$ belongs to $I_0$
(by \ref{sectchoice}, this is no loss of generality).
We first establish some formulas in type~$D_4$.
\begin{lemma}\label{longD4}
Let $\g$ be of type $D_4$. We have
\[
\begin{array}{rcl}
\chi_q(L(Y_{1,3}Y_{2,0}))_{\le 2} &=& Y_{1,3}Y_{2,0}
\left(1+v_2(1+v_3)(1+v_4)\right),\\[3mm]
\chi_q(L(Y_{2,0}Y_{3,3}Y_{4,3}))_{\le 2} &=& Y_{2,0}Y_{3,3}Y_{4,3}
\left(1+v_2(1+v_1)\right),\\[3mm]
\chi_q(L(Y_{1,3}Y_{2,0}^2Y_{3,3}Y_{4,3}))_{\le 2} &=& Y_{1,3}Y_{2,0}^2Y_{3,3}Y_{4,3}
\left(1+v_2(2+v_1+v_3+v_4)
\right.\\[2mm]
&&\left.
+\ v_2^{2}(1+v_1)(1 + v_3)(1+v_4)\right).
\end{array}
\]
\end{lemma}
\begin{proof}
The first two formulas concern multiplicity-free roots, and thus follow from Proposition~\ref{prop33}.
For the third one,
taking $J=\{1,2\}$ and setting $m=Y_{1,0}Y_{2,3}^2Y_{3,0}Y_{4,0}$,
we have by Example~\ref{exa28} (or Corollary~\ref{corA})
\[
 \vp_J(m)_{\le 2}\ = m\left(1+v_2(2+v_1)+ v_2^2(1+v_1)\right),
\]
and by \ref{belongJ} we know that $\vp_J(m) \le \chi_q(L(m))$.
Replacing $J$ by $\{2,3\}$ and by $\{2,4\}$ we get that
\[
m\left(1+v_2(2+v_1+v_3+v_4)
+ v_2^2(1+v_1+ v_3+v_4)\right)
\le \chi_q(L(m))_{\le 2}.
\]
Now $m_1:=mv_2^{2}v_1=Y_{1,1}Y_{2,2}^{-1}Y_{3,1}^2Y_{3,3}Y_{4,1}^2Y_{4,3}$
is $J$-dominant for $J=\{3,4\}$, and 
\[
 \vp_J(m_1)_{\le 2} = m_1(1+v_3)(1+v_4),
\]
so, by \ref{belongJ}, we see that $m_1(1+v_3)(1+v_4) \le \chi_q(L(m))$.
Arguing similarly with $m_2:=mv_2^{2}v_3$ and $m_3:=mv_2^2v_4$
we obtain that 
\begin{equation}\label{lowerb}
m\left(1+v_2(2+v_1+v_3+v_4)
+v_2^{2}(1+v_1)(1 + v_3)(1+v_4)\right) 
\le \chi_q(L(m))_{\le 2}.
\end{equation}
Conversely, we have 
$\chi_q(L(m)) \le \chi_q(L(Y_{2,0}))\chi_q(L(Y_{1,3}Y_{2,0}Y_{3,3}Y_{4,3})$.
Using Proposition~\ref{prop33}, we get 
\[
\chi_q(L(m))_{\le 2}\ \le 
m\left(1+v_2(1+v_1)(1+v_3)(1+v_4)\right)(1+v_2). 
\]
This upper bound is equal to the lower bound of (\ref{lowerb}) plus 
\[
m\left(v_2v_1v_3+v_2v_1v_4
+v_2v_3v_4+v_2v_1v_3v_4\right).
\] 
We also have
$\chi_q(L(m)) \le \chi_q(L(Y_{1,3}Y_{2,0}))\chi_q(Y_{2,0}Y_{3,3}Y_{4,3})$.
Using the first two formulas, we get
\[
\chi_q(L(m))_{\le 2}\ \le 
m\left(1+v_2(1+v_3)(1+v_4)\right)\left(1+v_2(1+v_1)\right). 
\]
This second upper bound does not contain the monomials 
$mv_2v_1v_3$
and $mv_2v_1v_3v_4$.
We can rule out the two remaining monomials by using the
three-fold symmetry $1 \leftrightarrow 3 \leftrightarrow 4$, and we obtain
an upper bound equal to the lower bound.
\cqfd 
\end{proof}

\noindent
{\em Proof of Th.~\ref{th2proveD} ---}\quad
If $\b$ is a multiplicity-free positive root,
the result follows from Theorem~\ref{multfreeequal}.
So we assume that $\b=\sum_i b_i\a_i$ has some multiplicity.
In view of the list of roots for $D_n$,
this implies that $b_{n-2}=2$, and $b_{n-1}=b_n=1$.
Since $n-2\in I_0$, we also have $\a=\sum_i a_i\a_i$ with $a_{n-2}=2$
and $a_{n-1}=a_n=1$.

Consider the subgraph $\De'$ of the Dynkin diagram supported
on $[1,n-1]$, of type $A_{n-1}$.
Set $\b'=\b-\a_n$ and $\a'=\a-\a_n$.
Then $\a'=\tau_-'(\b')$, where $\tau_-'$ is defined like
$\tau_-$, but for $\De'$.
Let $Y^\b$ be the highest $l$-weight of $S(\b)$.
By Corollary~\ref{corA},
\[
\vp_{[1,n-1]}(Y^\b) = Y^\b F_{\a'}(v_1,\ldots,v_{n-1}),
\]
where $F_{\a'}$ is the $F$-polynomial for $\De'$,
given by the explicit formula~(\ref{eqFpol}).
Since $n\in I_1$, this formula shows that
$F_{\a'}(v_1,\ldots,v_{n-1})=F_{\a}(v_1,\ldots,v_{n-1},0)$.
Hence the specialization at $v_n=0$ of the polynomial
$\widetilde{\chi}_q(S(\b))_{\le 2}$ is equal to
$F_{\a}(v_1,\ldots,v_{n-1},0)$.
Similarly, the specialization at $v_{n-1}=0$ of 
$\widetilde{\chi}_q(S(\b))_{\le 2}$ is equal to
$F_{\a}(v_1,\ldots,v_{n-2},0,v_n)$.

Thus we are reduced to prove that the monomials
of $\widetilde{\chi}_q(S(\b))_{\le 2}$
and $F_{\a}(v_1,\ldots,v_n)$
containing both $v_{n-1}$ and $v_n$ are the
same and have the same coefficients.

We first show that such a monomial is 
of the form $mv_{n-2}^2v_{n-1}v_n$ where $m$
is a monomial in $v_1,\ldots,v_{n-3}$.
For the polynomial $F_{\a}(v_1,\ldots,v_n)$ this follows
easily from the definition of an $\a$-acceptable
vector and the values of $a_{n-2}, a_{n-1}, a_n$.
For $\widetilde{\chi}_q(S(\b))_{\le 2}$, we have
\[
\widetilde{\chi}_q(S(\b))_{\le 2} \
\le\ 
\widetilde{\chi}_q(S(\b-2\a_{n-2}-\a_{n-1}-\a_n))_{\le 2}\
\widetilde{\chi}_q(S(2\a_{n-2}+\a_{n-1}+\a_n))_{\le 2}\,,
\]
where, by Proposition~\ref{prop30}, 
$\widetilde{\chi}_q(S(\b-2\a_{n-2}-\a_{n-1}-\a_n))_{\le 2}$
does not contain $v_{n-2}, v_{n-1}, v_n$.
Moreover, $2\a_{n-2}+\a_{n-1}+\a_n$ is supported on a 
root system of type $A_3$, so by Example~\ref{ex121} and Proposition~\ref{prop30}
all the monomials of $\widetilde{\chi}_q(S(2\a_{n-2}+\a_{n-1}+\a_n))_{\le 2}$
containing $v_{n-1}$ and $v_{n}$ are of the form
$m'v_{n-2}^2v_{n-1}v_n$ where $m'$
is a monomial in $v_{n-3}$.

We now note that if $mv_{n-2}^2v_{n-1}v_n$ appears in 
$F_{\a}(v_1,\ldots,v_n)$
(where $m$ is a monomial in the variables $v_1,\ldots,v_{n-3}$), 
then $mv_{n-2}^2$ also occurs,
and with the same multiplicity. This follows again 
from the definition of an $\a$-acceptable vector $\ga$
and of the integer $e(\ga,\a)$.
We have seen that $mv_{n-2}^2$ occurs in 
$\widetilde{\chi}_q(S(\b))_{\le 2}$ and 
$F_{\a}(v_1,\ldots,v_n)$ with the same multiplicity.
Thus all we have to do is to show that 
$mv_{n-2}^2$ and
$mv_{n-2}^2v_{n-1}v_n$
occur in $\widetilde{\chi}_q(S(\b))_{\le 2}$
with the same multiplicity.
Let us denote these multiplicities by $a$ and $b$,
respectively.

Since $Y^\b mv_{n-2}^2$ is $\{n-1,n\}$-dominant,
$\widetilde{\chi}_q(S(\b))_{\le 2}$  
contains $mv_{n-2}^2(1+v_{n-1})(1+v_n)$ with multiplicity
at least $a$ by \ref{belongJ}, so $a\le b$.

Conversely, let $\b''=\b-\a_n-\a_{n-1}$.
The multiplicity $a$ of $mv_{n-2}^2$ in $\widetilde{\chi}_q(S(\b))_{\le 2}$
is equal to its multiplicity in $\widetilde{\chi}_q(S(\b''))_{\le 2}$.
This is a multiplicity in type $A_{n-2}$, so using Corollary~\ref{corA},
we can check that it coincides with the multiplicity of $mv_{n-2}^2$ 
in the product
\[
\widetilde{\chi}_q(S(\b''-2\a_{n-2}-\a_{n-3}))_{\le 2}\ 
\widetilde{\chi}_q(S(2\a_{n-2}+\a_{n-3}))_{\le 2} 
\]
or equivalently in the product
\[
\widetilde{\chi}_q(S(\b''-2\a_{n-2}-\a_{n-3}))_{\le 2}\ 
\widetilde{\chi}_q(S(2\a_{n-2}+\a_{n-3}+\a_{n-1}+\a_n))_{\le 2}. 
\]
Using the third formula of Lemma~\ref{longD4} and the fact that the first
factor contains no variable $v_{n-2}$, $v_{n-1}$, or $v_n$, we obtain
that in this product the multiplicity $a$ of $mv_{n-2}^2$ is equal
to the multiplicity of $mv_{n-2}^2v_{n-1}v_n$.
But, by \ref{sect23}, this is greater or equal to $b$. Hence $a\ge b$.
This concludes the proof of Theorem~\ref{th2proveD}.
\cqfd

Hence (\ref{eq2prove}) is verified, and this proves Conjecture~\ref{mainconjecture}~{(i)} in type $D_n$.

\subsection{}
In this section we take $\g$ of type $D_4$ and we prove that 
Conjecture~\ref{mainconjecture}~(ii) is verified.
For this we need to prove \ref{clusterexpa}~(i) (ii) and (iii). 

\subsubsection{}
The proof that $F_i\otimes F_j$ is simple is the same as in type $A_n$
(\cf \ref{proof92i}).
We then consider $F_i\otimes S(\a)$ for $\a\in \Phi_{\ge -1}$.
If $\a=-\a_i$ or if $\a$ is a multiplicity-free positive root, 
we can repeat the argument of \ref{proof92i}.
If $\a=\a_1+2\a_2+\a_3+\a_4$, we choose $I_1=\{2\}$, so that
$\tau_-(\a)=\a_1+\a_2+\a_3+\a_4$.
By Theorem~\ref{th2proveD}, we have
\[
\widetilde{\chi}_q(S(\a))_{\le 2} =
1+v_1+v_3+v_4+v_1v_3+v_1v_4+v_3v_4+v_1v_3v_4+v_1v_2v_3v_4. 
\]
It is easy to check that this is the result given by
the Frenkel-Mukhin algorithm, so we can argue again as in
\ref{proof92i}. This proves \ref{clusterexpa}~(i).

\subsubsection{}
Let $\a, \b \in \Phi_{\ge -1}$ be two compatible roots.
We want to show that $S(\a)\otimes S(\b)$ is simple.
If $\a$ or $\b$ is negative, we can repeat the argument
of \ref{proof92iiA} (a) (b). So let us suppose that 
$\a, \b \in \Phi_{>0}$.

If the union of the supports of $\a$ and $\b$ is 
strictly smaller than $I$, we can assume by symmetry
that it is contained in $\{1,2,3\}$.
Take $I_1 = \{2\}$.
By Proposition~\ref{prop30}, we then have
\[
\chi_q(S(\a))_{\le 2} = \vp_{\{1,2,3\}}(Y^\a)_{\le 2},
\quad 
\chi_q(S(\b))_{\le 2} = \vp_{\{1,2,3\}}(Y^\b)_{\le 2}, 
\]
that is, our $q$-characters are in fact of type $A_3$.
On the other hand, $\a$ and $\b$ are also compatible
as roots of type $A_3$, hence the equality
\[
\chi_q(S(\a))_{\le 2}\ \chi_q(S(\b))_{\le 2} = \chi_q(S(\a)\otimes S(\b))_{\le 2} 
\]
follows from \ref{proof92iiA}.

We are thus reduced to those pairs of compatible positive roots
$(\a,\b)$ whose union of supports is equal to $I$.
By \cite[Prop. 3.16]{FZ1}, these are, up to the 3-fold symmetry
$1\leftrightarrow 3 \leftrightarrow 4$,
\begin{itemize}
\item[(a)] $\a= \a_1$,\quad $\b = \a_1+\a_2+\a_3+\a_4$; 
\item[(b)] $\a= \a_1+\a_2$,\quad $\b = \a_1+2\a_2+\a_3+\a_4$;
\item[(c)] $\a= \a_1+\a_2+\a_3$,\quad $\b = \a_2+\a_3+\a_4$;
\item[(d)] $\a= \a_1+\a_2+\a_3$,\quad $\b = \a_1+\a_2+\a_3+\a_4$;
\item[(e)] $\a= \a_1+\a_2+\a_3$,\quad $\b = \a_1+2\a_2+\a_3+\a_4$;
\item[(f)] $\a=\b= \a_1+\a_2+\a_3+\a_4$;
\item[(g)] $\a=\b= \a_1+2\a_2+\a_3+\a_4$.
\end{itemize}

\medskip
(a) 
Take $I_0=\{2\}$.
Then $\chi_q(S(\a))_{\le 2} = Y_{1,3}$ and, since $\tau_-(\b)=\a_2$
we have 
\[
\chi_q(S(\b))_{\le 2} = Y_{1,3}Y_{2,0}Y_{3,1}Y_{4,1}(1+v_2),
\]
hence $\chi_q(S(\a)\otimes S(\b))_{\le 2}$ contains a unique
dominant monomial and $S(\a)\otimes S(\b)$ is simple.

\medskip
(b)
Take $I_1=\{2\}$. Then, by Theorem~\ref{th2proveD},
\[
\chi_q(S(\a)\otimes S(\b))_{\le 2} = Y^\a Y^\b
(1+v_1)
(1+v_1+v_3+v_4+v_1v_3+v_1v_4+v_3v_4+v_1v_3v_4+v_1v_2v_3v_4). 
\]
The only dominant monomials in this product are $Y^\a Y^\b$ and 
$Y^\a Y^\b v_1v_2v_3v_4$, hence it is enough to show that
$Y^\a Y^\b v_1v_2v_3v_4$ occurs in $\chi_q(L(Y^\a Y^\b))$.
Now $Y^\a Y^\b = Y_{1,0}^2Y_{2,3}^3Y_{3,0}Y_{4,0}$, and clearly
$m:=Y^\a Y^\b v_4 = Y_{1,0}^2Y_{2,1}Y_{2,3}^3Y_{3,0}Y_{4,2}^{-1}$
occurs in $\chi_q(L(Y^\a Y^\b))$. Since $m$ 
does not contain $v_1, v_2, v_3$, 
we know by \ref{belongJ} that $\vp_{\{1,2,3\}}(m)$ is contained in
$\chi_q(L(Y^\a Y^\b))$. 
To calculate $\vp_{\{1,2,3\}}(m)$, 
we write $m=Y^\ga (Y_{2,1}Y_{2,3})Y_{4,2}^{-1}$ 
where $\ga = 2\a_1+ 2\a_2 + \a_3$ is in the root lattice of type
$A_3$ and has the cluster expansion $\ga = (\a_1+\a_2)+(\a_1+\a_2+\a_3)$.
It then follows from Section~\ref{secttypA} that
\[
\vp_{\{1,2,3\}}(m) = m(1+v_1)(1+v_1+v_3+v_1v_3+v_1v_2v_3), 
\]
hence $mv_1v_2v_3=Y^\a Y^\b v_1v_2v_3v_4$ occurs in $\chi_q(L(Y^\a Y^\b))$.

\medskip
(c)
Take $I_0 = \{2\}$.
Then $\tau_-(\a)=\a_2+\a_4$ and $\tau_-(\b)=\a_1+\a_2$, so
\[
\chi_q(S(\a)\otimes S(\b))_{\le 2} = Y^\a Y^\b
(1+v_2+v_2v_4)(1+v_2+v_1v_2). 
\]
But by Section~\ref{secttypA}, this is equal to $\vp_{\{1,2,4\}}(Y^\a Y^\b)$,
which is contained in $\chi_q(L(Y^\a Y^\b))$.

\medskip
(d)
Take $I_0 = \{2\}$.
Then $\tau_-(\a)=\a_2+\a_4$ and $\tau_-(\b)=\a_2$, so
one can argue as in (c).

\medskip
(e)
Take $I_1=\{2\}$. Then, by Theorem~\ref{th2proveD},
\[
\chi_q(S(\a)\otimes S(\b))_{\le 2} = Y^\a Y^\b
(1+v_1+v_3+v_1v_3+v_1v_2v_3)
\qquad\qquad\qquad
\]
\[
\qquad\qquad\times(1+v_1+v_3+v_4+v_1v_3+v_1v_4+v_3v_4+v_1v_3v_4+v_1v_2v_3v_4). 
\]
The only dominant monomials other than $Y^\a Y^\b$ are
$Y^\a Y^\b v_1v_2v_3$, $Y^\a Y^\b v_1v_2v_3v_4$ which occurs
with coefficient 2, and
$Y^\a Y^\b v_1^2v_2^2v_3^2v_4$, hence it is enough to show that
they occur in $\chi_q(L(Y^\a Y^\b))$.
For the first one, which does not depend on $v_4$, one can
argue as in (c) or (d). 
Now $Y^\a Y^\b = Y_{1,0}^2Y_{2,3}^3Y_{3,0}^2Y_{4,0}$, and clearly
$m:=Y^\a Y^\b v_4 = Y_{1,0}^2Y_{2,1}Y_{2,3}^3Y_{3,0}^2Y_{4,2}^{-1}$
occurs in $\chi_q(L(Y^\a Y^\b))$. Since $m$ does not contain $v_1,
v_2, v_3$, 
we know by \ref{belongJ} that $\vp_{\{1,2,3\}}(m)$ is contained in
$\chi_q(L(Y^\a Y^\b))$. 
To calculate $\vp_{\{1,2,3\}}(m)$, 
we write $m=Y^\ga (Y_{2,1}Y_{2,3})Y_{4,2}^{-1}$ 
where $\ga = 2\a_1+ 2\a_2 + 2\a_3$ is in the root lattice of type
$A_3$ and has the cluster expansion $\ga = 2(\a_1+\a_2+\a_3)$.
It then follows from Section~\ref{secttypA} that
\[
\vp_{\{1,2,3\}}(m) = m(1+v_1+v_3+v_1v_3+v_1v_2v_3)^2, 
\]
hence $mv_1v_2v_3=Y^\a Y^\b v_1v_2v_3v_4$ 
and $mv_1^2v_2^2v_3^2=Y^\a Y^\b v_1^2v_2^2v_3^2v_4$
occur in $\chi_q(L(Y^\a Y^\b))$, the first one with
coefficient 2.

\medskip
(f)
Take $I_0=\{2\}$. Then $\chi_q(S(\a))_{\le 2} = Y^\a(1+v_2)$,
and its square has only one dominant monomial, namely $(Y^\a)^2$.

\medskip
(g)
Take $I_1=\{2\}$. Then, by Theorem~\ref{th2proveD},
\[
\chi_q(S(\a)^{\otimes 2})_{\le 2} = (Y^\a)^2
(1+v_1+v_3+v_4+v_1v_3+v_1v_4+v_3v_4+v_1v_3v_4+v_1v_2v_3v_4)^2. 
\]
The only dominant monomials other than $(Y^\a)^2$ are
$(Y^\a)^2 v_1v_2v_3v_4$ which occurs
with coefficient 2, and
$(Y^\a)^2v_1^2v_2^2v_3^2v_4^2$, hence it is enough to show that
they occur in $\chi_q(L((Y^\a)^2))$.
Now $m_1:=(Y^\a)^2v_4$ and $m_2:=(Y^\a)^2v_4^2$ both
occur in $\chi_q(L((Y^\a)^2))$, the first one with coefficient 2,
and both are $\{1,2,3\}$-dominant.
Considering $\vp_{\{1,2,3\}}(m_1)$ and $\vp_{\{1,2,3\}}(m_2)$
we can conclude as in (e).

\medskip
This finishes the proof of \ref{clusterexpa}~(ii).

\subsubsection{}
Finally \ref{clusterexpa}~(iii) follows from Corollary~\ref{corprime}
for all multiplicity-free roots. It remains to check it for the
longest root $\a=\a_1+2\a_2+\a_3+\a_4$. This can be done easily
using our explicit formulas for the truncated $q$-characters.
For example we see from Lemma~\ref{longD4} that 
the monomial $Y_{1,3}Y_{2,0}^2Y_{3,3}Y_{4,3}v_2v_3v_4$ occurs
in $\chi_q(L(Y_{1,3}Y_{2,0}))\chi_q(L(Y_{2,0}Y_{3,3}Y_{4,3}))$
but not in $\chi_q(S(\a))$. 
All other factorizations of $S(\a)$ can be ruled out in a similar
way, and we omit the details.

This concludes the proof of Conjecture~\ref{mainconjecture} in type $D_4$.
\cqfd 

\section{Applications}
\label{sectappl}

In this section we present some interesting consequences 
of our results concerning $\CC_1$.

\subsection{}
By construction, the cluster variables of a cluster algebra
satisfy some algebraic identities coming
from the mutation procedure. 
When we restrict to mutations on the {\em bipartite belt} \cite{FZ4}
these identities are similar to a $T$-system. 
In the case $\ell=1$ this $T$-system is periodic and involves
the classes of all the cluster simple objects of $\CC_1$,
as we shall now see.

Define $\tau = \tau_+\tau_-$ (see \ref{sectFpol1}).
For every $i\in I$, define a sequence $\ga_i(j)\ (j\in\Z)$ of elements of
$\Phi_{\ge -1}$ by
\begin{equation}
\ga_{i}(2j) = \tau^j(-\a_i);
\end{equation}
for odd $j$, $\ga_i(j)$ is defined via the parity condition
\begin{equation}\label{pargam}
\ga_i(j)=\ga_i(j+1) \quad \mbox{if\quad $\eps_i=(-1)^{j+1}$.}
\end{equation}
It follows from \cite{FZ1} that every element of $\Phi_{\ge-1}$ is of
the form $\ga_i(j)$ for some $i\in I$ and some $j\in [0,h+2]$.
For every $i\in I$, define a sequence $\b_i(j)\ (j\in\Z)$ of elements of
$\Phi_{\ge -1}$ by the 
initial condition
\begin{equation}
\b_i(0)=\a_{i}, 
\end{equation}
and the recursion
\begin{equation}
\b_i(j+1) = \tau_{(-1)^{j+1}}(\b_i(j)),\qquad (j\in\Z).  
\end{equation}
Define the following monomials in the classes of the frozen simple objects $F_i$:
\begin{equation}
 p_i^+(j) = \prod_{k\in I} [F_k]^{\max(0,[\b_k(j) : \a_i])},
\qquad
p_i^-(j) = \prod_{k\in I} [F_k]^{\max(0,-[\b_k(j) : \a_i])},
\qquad (i\in I,\ j\in\Z).
\end{equation}
The next result follows from Conjecture~\ref{mainconjecture}~{(i)},
hence it is now proved if $\g$ is of type $A_n$ or $D_n$.
\begin{theorem}\label{prop15}
The cluster simple objects $S(\a)\ (\a\in\Phi_{\ge -1})$ of $\CC_1$ 
satisfy the following system of equations in the Grothendieck ring $R_1$
\[
 [S(\ga_i(j+1))]\,[S(\ga_i(j-1))] = p_i^+(j) + p_i^-(j)\prod_{k\not = i} [S(\ga_k(j))]^{-a_{ik}},
\qquad (i\in I,\ j\in\Z).
\]
\end{theorem}

\begin{proof}
Set 
\begin{equation}
p_i(j) = \frac{p_i^+(j)}{p_i^-(j)} = \prod_{k\in I} [F_k]^{[\b_k(j) : \a_i]}, 
\end{equation}
a Laurent monomial in the $[F_i]$.
We will first show that
\begin{equation}\label{eqp}
p_i(j+1)p_i(j-1) = \prod_{k\not = i}(p^+_k(j))^{-a_{ik}}. 
\end{equation}
To do that, set
\[
q_i(j) = \frac{\prod_{k\not = i}(p^+_k(j))^{-a_{ik}}}{p_i(j+1)p_i(j-1)}. 
\]
For $i\in I$, define following \cite{FZ1} a piecewise-linear automorphism $\si_i$
of $Q$ by
\begin{equation}
[\si_i(\ga) : \a_j] = 
\left\{
\begin{array}{ll}
[\ga : \a_j] &\mbox{if} \ j\not = i,\\[5mm]
-[\ga : \a_i] - \displaystyle\sum_{k\not = i} a_{ik} \max(0, [\ga : \a_k])
&\mbox{if} \  j = i.
\end{array}
\right.
\end{equation} 
Then $\tau_\eps = \prod_{\eps_i=\eps}\si_i$.
Using the definition of $p^\pm_{i}(j)$ 
one can see
that the exponent of $[F_u]$ in $q_{i}(j)$ is equal to
\[
[\si_{i}(\b_u(j))+\b_u(j)-\b_u(j+1)-\b_u(j-1) : \a_{i}]
\qquad\qquad\qquad\qquad
\]
\[
\qquad\qquad\qquad\qquad
=\
[\si_{i}(\b_u(j))+\b_u(j)-\tau_{(-1)^{j+1}}(\b_u(j))-\tau_{(-1)^{j}}(\b_u(j)) : \a_{i}].
\]
Suppose that $\eps_i=(-1)^{j+1}$. Then $\si_{i}$ appears once in
$\tau_{(-1)^{j+1}}$ and does not appear in $\tau_{(-1)^{j}}$. 
It follows that
\[
[\tau_{(-1)^{j+1}}(\b_u(j))+\tau_{(-1)^{j}}(\b_u(j)) : \a_{i}]
=
[\si_{i}(\b_u(j))+\b_u(j) : \a_{i}],
\]
hence, the exponent of $[F_u]$ in $q_{i}(j)$ is $0$.
The case $\eps_i={(-1)^{j}}$ is identical.
Thus, $q_i(j)=1$ and (\ref{eqp}) is proved.

Set $y_i(j)=1/p_i(j)$. Then, using the
tropical semifield structure on the set of Laurent monomials
in the $[F_i]$ (see \ref{trop}), one has $1\oplus y_i(j) = 1/p_i^+(j)$,
hence 
\[
 p_i^+(j) = \frac{1}{1\oplus y_i(j)}, \qquad
 p_i^-(j) = \frac{y_i(j)}{1\oplus y_i(j)}.
\]
With this new notation, (\ref{eqp}) becomes
\begin{equation}\label{eq18}
 y_i(j+1)y_i(j-1) = \prod_{k\not =i}(1\oplus y_k(j))^{-a_{ik}},
\end{equation}
and the relation of Theorem~\ref{prop15} takes the form
\begin{equation}\label{eq19}
 [S(\ga_i(j+1))]\,[S(\ga_i(j-1))] = 
\frac{1 + y_i(j)\prod_{k\not = i}[S(\ga_k(j))]^{-a_{ik}}}{1\oplus y_i(j)}\,.
\end{equation}
Equations (\ref{eq18}) and (\ref{eq19}) coincide with formulas (8.11) and (8.12) of \cite{FZ4},
which proves the theorem.
\cqfd
\end{proof}
\begin{example}
{\rm
We take $\g$ of type $A_3$ and choose $I_0=\{1,3\}$, $I_1=\{2\}$.
Hence $\tau_+ = \si_1\si_3$, $\tau_-=\si_2$, and $\tau=\si_1\si_3\si_2$.
We have
\[
\begin{array}{lll}
\ga_1(0)=-\a_1, & \ga_2(0)=-\a_2, & \ga_3(0)=-\a_3, \\[2mm]
\ga_1(1)=\a_1, & \ga_2(1)=-\a_2, & \ga_3(1)=\a_3, \\[2mm]
\ga_1(2)=\a_1, & \ga_2(2)=\a_1+\a_2+\a_3, & \ga_3(2)=\a_3, \\[2mm]
\ga_1(3)=\a_2+\a_3, & \ga_2(3)=\a_1+\a_2+\a_3, & \ga_3(3)=\a_1+\a_2, \\[2mm]
\ga_1(4)=\a_2+\a_3, & \ga_2(4)=\a_2, & \ga_3(4)=\a_1+\a_2, \\[2mm]
\ga_1(5)=-\a_3, & \ga_2(5)=\a_2, & \ga_3(5)=-\a_1, \\[2mm]
\ga_1(6)=-\a_3, & \ga_2(6)=-\a_2, & \ga_3(6)=-\a_1, \\[2mm]
\end{array}
\]
and
\[
\begin{array}{lll}
\b_1(0)=\a_1, & \b_2(0)=\a_2, & \b_3(0)=\a_3, \\[2mm]
\b_1(1)=\a_1+\a_2, & \b_2(1)=-\a_2, & \b_3(1)=\a_2+\a_3, \\[2mm]
\b_1(2)=\a_2+\a_3, & \b_2(2)=-\a_2, & \b_3(2)=\a_1+\a_2, \\[2mm]
\b_1(3)=\a_3, & \b_2(3)=\a_2, & \b_3(3)=\a_1, \\[2mm]
\b_1(4)=-\a_3, & \b_2(4)=\a_1+\a_2+\a_3, & \b_3(4)=-\a_1, \\[2mm]
\b_1(5)=-\a_3, & \b_2(5)=\a_1+\a_2+\a_3, & \b_3(5)=-\a_1, \\[2mm]
\b_1(6)=\a_3, & \b_2(6)=\a_2, & \b_3(6)=\a_1. \\[2mm]
\end{array}
\]
The formulas of Theorem~\ref{prop15} read for $i=1,3$:
\[
\begin{array}{l}
[S(\a_1)]\,[S(-\a_1)]=[F_1]+[S(-\a_2)], \\[2mm]
[S(\a_2+\a_3)]\,[S(\a_1)]= [F_3]+[S(\a_1+\a_2+\a_3)], \\[2mm]
[S(-\a_3)]\,[S(\a_2+\a_3)]=[F_2]+[F_3]\,[S(\a_2)], \\[2mm]
[S(\a_3)]\,[S(-\a_3)]=[F_3]+[S(-\a_2)],\\[2mm]
[S(\a_1+\a_2)]\,[S(\a_3)]= [F_1]+[S(\a_1+\a_2+\a_3)],\\[2mm]
[S(-\a_1)]\,[S(\a_1+\a_2)]=[F_2]+[F_1]\,[S(\a_2)].
\end{array}
\]
and for $i=2$:
\[
\begin{array}{l}
[S(\a_1+\a_2+\a_3)]\,[S(-\a_2)]=[F_1][F_3]+[F_2]\,[S(\a_1)]\,[S(\a_3)],\\[2mm]
[S(\a_2)]\,[S(\a_1+\a_2+\a_3)]=[F_2]+[S(\a_1+\a_2)]\,[S(\a_2+\a_3)],\\[2mm]
[S(-\a_2)]\,[S(\a_2)]=[F_2]+[S(-\a_1)]\,[S(-\a_3)]. 
\end{array}
\]
}
\end{example}
\begin{remark}
{\rm
In type $E_n$, using Proposition~\ref{prop33}, we can still prove 
identities involving only multi\-pli\-ci\-ty-free roots. 
For example, one can show that for every $i\in I$,
\[
 \begin{array}{rcl}
[S(\a_i)][S(-\a_i)]&=&[F_i]+\ds\prod_{j\not = i}[S(-\a_j)]^{-a_{ij}},\\[5mm]
[S(\a_i-\sum_{j\not = i}a_{ij}\a_j)][S(-\a_i)]&=&\ds\prod_{j\not = i}[F_j]^{-a_{ij}}
+[F_i]\ds\prod_{j\not=i}[S(\a_j)]^{-a_{ij}}.
 \end{array}
\]
The first formula is a classical $T$-system, but not the second one.
} 
\end{remark}

\subsection{}
If $\g$ is of type $A_n$, we have the following well-known
duality for the characters of the finite-dimensional irreducible
$\g$-modules. If $\l=(\l_1,\ldots,\l_n)$ and 
$\mu=(\mu_1,\ldots,\mu_n)$ are two dominant weights, identified
in the standard way
with partitions in at most $n$ parts,
there holds
\[
\dim V(\l)_\mu = [V(\mu_1)\otimes\cdots \otimes V(\mu_n)\ : \ V(\l)], 
\]
where $V(\l)$ stands for the irreducible $\g$-module with highest
weight $\l$.
In plain words, the weight multiplicities of $V(\l)$ coincide
with the multiplicities of $V(\l)$ as a direct summand 
of certain tensor products.
This is sometimes called {\em Kostka duality} and is specific to type $A$.

We find it interesting to note that Eq.~(\ref{eq2prove}) 
yields a similar duality for the truncated $q$-characters
of the cluster simple objects $S(\b)$ for any $\g$.
Indeed assume that (\ref{eq2prove}) holds, namely that
\[
\widetilde{\chi}_q(S(\b))_{\le 2}\ =\ F_{\tau_-(\b)}(v_1,\ldots,v_n),\qquad (v_i=A_{i,\xi_i+1}^{-1}), 
\]
(this is proved in type $A$ and $D$ and for all multiplicity-free 
roots in type $E$).
Then writing for short $\a=\tau_-(\b)=\sum_i a_i\a_i$, 
$\vv^\ga = \prod_{i\in I}v_i^{c_i}$ for $\ga = \sum_i c_i\a_i\in Q$, and
\[
F_{\a}(v_1,\ldots,v_n)=\sum_\ga n_{\b,\ga} \vv^\ga, 
\]
we have that $n_{\b,\ga}$ is the multiplicity of the $l$-weight
$Y^\b \vv^\ga$ in $S(\b)$.
On the other hand, define 
\[
T(\a) := \bigotimes_{i\in I} S(-\eps_i\a_i)^{\otimes a_i}. 
\]
Note that $\chi_q(T(\a))_{\le 2} = \prod_i Y_{i,\xi_i+2}^{a_i}$
is reduced to a single monomial, so $T(\a)$ is simple.
(From the cluster algebra point of view, $(x[-\eps_i\a_i]; i\in I)$
is a cluster of $\A$.)
Put
\[
d_i = 
\left\{
\begin{array}{ll}
-\ds\sum_{j\not = i} c_j a_{ij} & \mbox{if $i\in I_0$,}\\[5mm]
\ds\sum_{j\not = i} (c_j-a_j) a_{ij} & \mbox{if $i\in I_1$,} 
\end{array}
\right. ,
\qquad
e_i = 
\left\{
\begin{array}{ll}
a_i-c_i & \mbox{if $i\in I_0$,}\\[5mm]
-c_i-\ds\sum_{j\not = i} c_j a_{ij} & \mbox{if $i\in I_1$,} 
\end{array}
\right..
\]
If $n_{\b,\ga}\not = 0$ then $d_i$ and $e_i$ are nonnegative
and we can consider the simple module 
\[
U(\ga):=\bigotimes_{i\in I} S(-\eps_i\a_i)^{\otimes d_i} \otimes F_i^{\otimes e_i}.
\]
\begin{proposition}
Assume that Eq.~(\ref{eq2prove}) holds for $S(\b)$.
Then the multiplicity of $U(\gamma)$
as a composition factor of the tensor product
$S(\b)\otimes T(\a)$ 
is equal to the $l$-weight multiplicity $n_{\b,\ga}$.
\end{proposition}
\begin{proof}
This is a direct calculation using Eq.~(\ref{equay}) and
the same idea as in the proof of
Proposition~\ref{positivity}.
The details are left to the reader. 
\cqfd 
\end{proof}

\subsection{}
So far we have only used some combinatorial and representation-theoretical
techniques. However, our results have geometric consequences.
Indeed, by work of Fu and Keller \cite{FK}, the
$F$-polynomials have nice geometric descriptions in terms of
quiver grassmannians. This goes as follows.

Let $C$ be the oriented graph obtained from the Dynkin diagram of $\g$ by
deciding that the vertices in $I_1$ are sources and those in $I_0$
are sinks. So $C$ is a Dynkin quiver, and we can associate to every
positive root $\a$ the unique (up to isomorphism) indecomposable
representation $M[\a]$ of $C$ over $\C$ with dimension vector~$\a$.
Regarding an element $\ga=\sum_i c_i\a_i$ of the root lattice with nonnegative coordinates
$c_i$ as a dimension vector for $C$, we can consider for every representation
$M$ of $C$ the quiver grassmannian
\[
\Gr_\ga(M) := \{N \mid N \mbox{ is a subrepresentation of $M$ with dimension $\ga$}\}.
\]
This is a closed subset of the ordinary grassmannian of subspaces 
of dimension $\sum_i c_i$ of the complex vector space $M$. 
So in particular, $\Gr_\ga(M)$ is a projective variety.
Denote by $\chi(\Gr_\ga(M))$ its topological Euler characteristic.
Then we have the following formula, inspired from a similar formula 
of Caldero and Chapoton for cluster expansions of cluster variables \cite{CC}.
\begin{theorem}{\rm \cite[Th. 6.5]{FK}}\label{thFK}
For $\a\in\Phi_{>0}$ we have
\[
F_\a(v_1,\ldots,v_n) = \sum_\ga \chi(\Gr_\ga(M[\a])) \,\vv^\ga. 
\]
\end{theorem}
This yields immediately
\begin{theorem}\label{thgeom}
If Eq.~(\ref{eq2prove}) holds for $\b\in\Phi_{\ge -1}$, then
\begin{equation}\label{eqgeom}
\chi_q(S(\b))_{\le 2}\ = 
Y^\b \sum_\ga \chi(\Gr_\ga(M[\tau_-(\b)])) \,\vv^\ga,
\qquad (v_i=A_{i,\xi_i+1}^{-1}). 
\end{equation}
\end{theorem}
\begin{example}
{\rm 
Take $\g$ of type $D_4$ and choose $I_0 = \{2\}$,
so that $C$ is the quiver of type $D_4$ with its three
arrows pointing to the trivalent node 2.
Let $\b=\a_1+2\a_2+\a_3+\a_4$ be the highest root.
We have $\tau_-(\b)=\b$.
The representation $M[\b]$ of $C$ is of dimension $5$.
There are thirteen non-empty quiver grassmannians corresponding
to the dimension vectors
\[
(0,0,0,0),\
(0,1,0,0),\
(0,2,0,0),\
(1,1,0,0),\
(0,1,1,0),\
(0,1,0,1),\ 
(1,2,0,0),
\]
\[
(0,2,1,0),\
(0,2,0,1),\
(1,2,1,0),\
(1,2,0,1),\
(0,2,1,1),\
(1,2,1,1). 
\]
The variety $\Gr_{(0,1,0,0)}(M[\b])$ is a projective line, hence
its Euler characteristic is equal to 2.
The twelve other grassmannians are reduced to a point.
Therefore we obtain that
\[
\chi_q(S(\b))_{\le 2}\ = 
Y_{1,3}Y_{2,0}^2Y_{3,3}Y_{4,3}
\left(1 + 2v_2 + v_2^2 + 
v_1v_2 + v_2v_3 + v_2v_4 + v_1v_2^2 + v_2^2v_3
+ v_2^2v_4\right.
\]
\[
\qquad\qquad
\qquad\qquad
\qquad\qquad
\left. +\ v_1v_2^2v_3 + v_1v_2^2v_4 + v_2^2v_3v_4 
+v_1v_2^2v_3v_4\right), 
\]
in agreement with Lemma~\ref{longD4}.
}
\end{example}

Note that if moreover Conjecture~\ref{mainconjecture}~(ii) holds, then
we can write any simple module $L(m)$ in $\CC_1$ as a tensor product of 
cluster simple objects. Taking into account the additivity properties of 
the Euler characteristics and the results of \cite{CK}, this gives for 
$\chi_q(L(m))_{\le 2}$ a formula similar to (\ref{eqgeom}), in which 
the indecomposable representation $M[\tau_-(\b)]$ is replaced by a generic 
representation of~$C$ (or equivalently a representation without self-extension).  

\begin{example}
{\rm 
Take $\g$ of type $A_2$ and choose $I_0 = \{1\}$,
so that $C$ is the quiver $1 \longleftarrow 2$.
Consider the simple module $S=L(Y_{1,0}^2Y_{2,3})$.
We have seen that $S\cong L(Y_{1,0}Y_{2,3})\otimes L(Y_{1,0})$.
This corresponds to the fact that the generic representation 
of $C$ of dimension vector $2\a_1 + \a_2$ is 
\[
M=(\C \stackrel{\rm id}{\longleftarrow} \C) \oplus (\C \stackrel{0}{\longleftarrow} 0).
\]
There are five non-empty quiver grassmannians for $M$ corresponding
to the dimension vectors
\[
(0,0),\quad
(1,0),\quad
(2,0),\quad
(1,1),\quad
(2,1).
\]
The variety $\Gr_{(1,0)}(M)$ is a projective line, hence
its Euler characteristic is equal to 2.
The four other grassmannians are reduced to a point.
Therefore we obtain that
\[
\chi_q(S)_{\le 2}\ = 
Y_{1,0}^2Y_{2,3}
\left(1 + 2v_1 + v_1^2 + 
v_1v_2 + v_1^2v_2\right),
\]
in agreement with Example~\ref{exa28}.
}
\end{example}

Theorem~\ref{thgeom} is very similar to a formula of Nakajima
\cite[\S 13]{N1} for the $q$-character of a standard module.
Indeed, as shown by Lusztig \cite{Lu}, the lagrangian quiver varieties  
used in Nakajima's character formula are isomorphic to grassmannians of submodules
of a projective module over a preprojective algebra.
There are however two important differences. In our case
the geometric formula gives only the truncated $q$-character
(but this is enough to determine the full $q$-character of
an object of $\CC_1$). More importantly, Theorem~\ref{thgeom}
concerns {\em simple} modules and not {\em standard} modules.
In Nakajima's approach, the $q$-characters of the simple
modules are obtained as alternating sums of $q$-characters
of standard modules using intersection cohomology methods.

\section{General $\ell$}\label{sect9}

We now consider the category $\CC_\ell$ for an arbitrary integer $\ell$.

\subsection{}
We define a quiver $\Gamma_{\ell}$ with vertex set
$
\{(i,k) \mid i\in I,\ 1\le k\le \ell+1\}.
$
The arrows of $\Gamma_{\ell}$ are given by the following
rule. Suppose that $(i,k)$ is such that $i\in I_0$ and $k$ is odd,
or $i\in I_1$ and $k$ is even. Then the arrows adjacent to 
$(i,k)$ are
\begin{itemize}
\item[(h)] the horizontal arrows $(i,k-1)\to (i,k)$ if $k>1$
and $(i,k+1)\to (i,k)$ if $k\le \ell$;
\item[(v)] the vertical arrows $(i,k)\to (j,k)$ where
$a_{ij}=-1$ and $k\le\ell$.
\end{itemize}
All arrows are of this type.
\begin{example}
{\rm
Take $\g$ of type $A_3$ and choose $I_0 = \{1,3\}$ and $I_1=\{2\}$.
The quiver $\Gamma_3$ is then
\[
\matrix{(1,1)& \leftarrow & (1,2) & \rightarrow &
 (1,3) &  \leftarrow & (1,4)  \cr
         \downarrow &&\uparrow &&  \downarrow &&   \cr
 (2,1) & \rightarrow & (2,2) & \leftarrow &
  (2,3) &  \rightarrow & (2,4)   \cr
         \uparrow &&\downarrow &&  \uparrow &&   \cr
(3,1)& \leftarrow & (3,2) & \rightarrow &
  (3,3) &  \leftarrow & (3,4)   
}
\]
}
\end{example} 

\subsection{}
Let $\widetilde{B}_{\ell}$ be the $n(\ell+1)\times n\ell$-matrix 
with set of column indices $I\times [1,\ell]$ and set of row indices
$I\times [1,\ell+1]$. 
The entry $b_{(i,k),(j,m)}$ is equal to $1$ if there is an arrow from
$(j,m)$ to $(i,k)$ in $\Gamma_\ell$, to $-1$ if there is an arrow from
$(i,k)$ to $(j,m)$, and to $0$ otherwise.

Let $\A_\ell = \A(\widetilde{B}_\ell)$ be the cluster algebra attached
to the initial seed $(\xx, \widetilde{B}_\ell)$, where
\[
\xx=(x_{(i,k)}\mid i\in I,\ 1\le k \le \ell+1).
\]
This is a cluster algebra of rank $n\ell$, with $n$ frozen variables
$f_i := x_{(i,\ell+1)}\ (i\in I)$.
It follows easily from \cite{FZ2} that $\A_\ell$ has in general
infinitely many cluster variables. The exceptional pairs $(\g,\ell)$
for which $\A_\ell$ has finite cluster type are listed in Table~\ref{table1}.
\begin{table}[t]
\begin{center}
\begin{tabular}
{|c|c|c|}
\hline
Type of $\g$ &  $\ell$ & Type of $\A_\ell$\\
\hline
$A_1$ & $\ell$ & $A_\ell$ \\
\hline
$X_n$ & $1$ & $X_n$ \\
\hline
$A_2$ & $2$ & $D_4$ \\
$A_2$ & $3$ & $E_6$\\
$A_2$ & $4$ & $E_8$ \\
\hline
$A_3$ & $2$ & $E_6$ \\
\hline
$A_4$ & $2$ & $E_8$ \\
\hline
\end{tabular}
\end{center}
\caption{\small \it Algebras ${\cal A}_\ell$ of finite cluster type.
\label{table1}}
\end{table}

\subsection{}
For $i\in I$ and $k\in[1,\ell+1]$, define
\[
 r(i,k) = \left\{
\begin{array}{ll}
2\ds\left\lceil\frac{\ell-k+1}{2}\right\rceil & \mbox{if $i\in I_0$,}\\[5mm]
2\ds\left\lceil\frac{\ell-k+2}{2}\right\rceil -1& \mbox{if $i\in I_1$,}
\end{array}
\right.
\]
where $\lceil x \rceil$ denotes the smallest integer $\ge x$.
These integers satisfy
\begin{itemize}
\item[(a)] $r(i,k)\ge r(i,k+1) \ge r(i,k+2)=r(i,k)-2$, 
\item[(b)] if $a_{ij}=-1$ then $r(j,k)$ is the unique integer
strictly between $r(i,k)$ and $r(i,k)+2(-1)^k\eps_i$.
\end{itemize}

Recall from \ref{sectKR} the Kirillov-Reshetikhin modules 
$W_{k,a}^{(i)}\ (i\in I,\ k\in\N,\ a\in\C^*)$.
The Kirillov-Reshetikhin modules in $\CC_\ell$ have spectral
parameters of the form $a=q^r$ for some integer $r$ between $0$ and $\ell+1$.
To simplify notation {\em we shall write $W_{k,r}^{(i)}$ instead of $W_{k,q^r}^{(i)}$}. 
We can now state our main conjecture, which generalizes Conjecture~\ref{mainconjecture}
to arbitrary $\ell$.
\begin{conjecture}\label{mainconjecture2}
The map $x_{(i,k)} \mapsto [W^{(i)}_{k,\,r(i,k)}]$ extends to 
a ring isomorphism $\iota$ from the cluster algebra $\A_\ell$ 
to the Grothendieck ring $R_\ell$ of $\CC_\ell$.
%
If we identify $\A_{\ell}$ with $R_\ell$ via $\iota$, $\CC_\ell$
becomes a monoidal categorification of $\A_{\ell}$.
\end{conjecture}
The idea to choose this initial seed for $\A_\ell$ comes from the $T$-systems.
Indeed, after replacing $x_{(i,k)}$ by $[W^{(i)}_{k,\,r(i,k)}]$,
the exchange relations (\ref{mutationformula}) 
for the initial cluster variables 
become,
\[
[W_{k,\,r(i,k)+2(-1)^k\eps_i}^{(i)}] =
\frac{[W_{k-1,\,r(i,k-1)}^{(i)}] [W_{k+1,\,r(i,k+1)}^{(i)}] 
+ 
\ds\prod_{a_{ij}=-1}[W_{k,\,r(j,k)}^{(j)}]}{[W_{k,\,r(i,k)}^{(i)}]} 
\qquad (i\in I,\ k\le \ell),
\]
which is an instance of Eq.~(\ref{eqTsystem}).

\subsection{}\label{CPtypeA1}
It follows from the work of Chari and Pressley \cite{CP2} that
Conjecture~\ref{mainconjecture2} holds for $\g=\Sl_2$ and any $\ell$
(see Example~\ref{exKRsl2}).
In this case $\A_\ell\equiv R_\ell$ is a cluster algebra of finite cluster
type $A_\ell$, with one frozen variable $[W_{\ell+1,\,0}]$.
The cluster variables are the classes of the other Kirillov-Reshetikhin  
modules of $\CC_\ell$, namely
\[
[W_{k,\,2s}],\qquad (1\le k \le \ell,\quad 0\le s\le \ell-k+1). 
\]
To determine the compatible pairs of cluster variables, one can
use again the geometric model of \cite{FZ1} and attach to each
cluster variable a diagonal of the $(\ell+3)$-gon $\P_{\ell+3}$
as follows:
\[
[W_{k,2s}] \longmapsto [s+1,\,s+k+2].
\]
We then have that $W_{k,2s}\otimes W_{k',2s'}$ is simple if and only
if the corresponding diagonals do not intersect in the interior
of $\P_{\ell+3}$. 

\subsection{}\label{A2l2}
Take $\g$ of type $A_2$ and $\ell =2$. We choose $I_0=\{1\}$
and $I_1 = \{2\}$.
In this case Conjecture~\ref{mainconjecture2} holds (see below \ref{typeA4}).
The cluster algebra $\A_2$ has finite cluster type $D_4$, hence every
simple object of $\CC_2$ is isomorphic to a tensor product of cluster
simple objects and frozen simple objects.

The sixteen cluster simple objects are
\[
L(Y_{1,0}),\  L(Y_{1,2}),\ L(Y_{1,4}),\ L(Y_{2,1}),\ L(Y_{2,3}),\ L(Y_{2,5}),\
L(Y_{1,0}Y_{1,2}),\ L(Y_{1,2}Y_{1,4}),\ L(Y_{2,1}Y_{2,3}),\ L(Y_{2,3}Y_{2,5}),
\]
\[
L(Y_{1,0}Y_{2,3}),\ L(Y_{1,2}Y_{2,5}),\ L(Y_{1,4}Y_{2,1}),\
L(Y_{1,0}Y_{1,2}Y_{2,5}),\  L(Y_{1,0}Y_{2,3}Y_{2,5}),\ 
L(Y_{1,0}Y_{1,2}Y_{2,3}Y_{2,5}).
\]
They have respective dimensions 
\[
3,\ 3,\ 3,\ 3,\ 3,\ 3,\ 6,\ 6,\ 6,\ 6,\ 8,\ 8,\ 8,\ 15,\ 15,\ 35. 
\]
Note that only the first ten are Kirillov-Reshetikhin modules.
The next five are evaluation modules.
The last module $L(Y_{1,0}Y_{1,2}Y_{2,3}Y_{2,5})$ is not an evaluation
module. Its restriction to $U_q(\g)$ is isomorphic to 
$V(2\varpi_1+2\varpi_2)\oplus V(\varpi_1+\varpi_2)$.

The two frozen modules are the Kirillov-Reshetikhin modules
$L(Y_{1,0}Y_{1,2}Y_{1,4}),\  L(Y_{2,1}Y_{2,3}Y_{2,5})$,
of dimension 10.

By \cite{FZ2} there are fifty clusters, \ie fifty factorization
patterns of a simple object of $\CC_2$ as a tensor product
of cluster simple objects and frozen simple objects.

\subsection{}
Take $\g$ of type $A_4$ and $\ell = 3$. 
In this case $\A_3$ has infinitely many cluster variables.
Choose $I_0 = \{1,3\}$ and $I_1 = \{2,4\}$. 
Thus the simple module $L(Y_{1,4}Y_{2,1}Y_{2,7}Y_{3,4})$
belongs to $\CC_3$. 
However, it is not a real simple object
\cite[\S4.3]{L}
 because
in the Grothendieck ring $R_3$ we have
\[
[L(Y_{1,4}Y_{2,1}Y_{2,7}Y_{3,4})]^2 = [L(Y_{1,4}^2Y_{2,1}^2Y_{2,7}^2Y_{3,4}^2)]
+ [L(Y_{2,1}Y_{2,3}Y_{2,5}Y_{2,7}Y_{4,3}Y_{4,5})]. 
\]
If $\CC_3$ is indeed a monoidal categorification of $\A_3$, 
then $[S]$ cannot be a cluster monomial.

For $\g$ of type $A_3$ and $\ell = 3$, there is a similar
example.
The simple module $L(Y_{1,4}Y_{2,1}Y_{2,7}Y_{3,4})$
belongs to $\CC_3$, and we have 
\[
[L(Y_{1,4}Y_{2,1}Y_{2,7}Y_{3,4})]^2 = [L(Y_{1,4}^2Y_{2,1}^2Y_{2,7}^2Y_{3,4}^2)]
+ [L(Y_{2,1}Y_{2,3}Y_{2,5}Y_{2,7})]. 
\]

We expect the existence of non real simple objects in $\CC_\ell$ 
whenever $\A_\ell$ does not have finite cluster type.

\subsection{}\label{typeAl}
Take $\g$ of type $A_n$ and let $\ell$ be arbitrary.
Let us sketch why in this case Conjecture~\ref{mainconjecture2} would
essentially follow from a conjecture of \cite{GLS1} (see also \cite[\S
23.1]{GLS2}).

First, we can use the quantum affine analogue of the Schur-Weyl duality
\cite{CP4, C, GRV} to relate the finite-dimensional representations of $U_q(\hg)$
with the finite-dimensional representations of the affine Hecke
algebras $\H_m(t)$ of type $A_m\ (m\ge 1)$ with parameter $t=q^2$. 
More precisely, for every $m$ we have a functor $\F_{m}$ from
$\md\,\H_m(t)$ to the category $\CC$ of finite-dimensional
representations of $U_q(\hg)$, which maps every simple module of $\H_m(t)$
to a simple module of $U_q(\hg)$ or to the zero module.
The simple $U_q(\hg)$-module with highest $l$-weight
$\prod_{i=1}^n\prod_{r=1}^{k_i}Y_{i,a_i}$
is the image by $\F_m$ of a simple $\H_m(t)$-module, where
$m=\sum_i ik_i$. 

Moreover, the functors $\F_m$ are multiplicative in the following sense:
for $M_1$ in $\md\, \H_{m_1}(t)$ and $M_2$ in $\md\, \H_{m_2}(t)$ one has
\[
\F_{m_1+m_2}(M_1\odot M_2) =
\F_{m_1}(M_1)\otimes\F_{m_2}(M_2)\,,
\]
where $-\odot-$ denotes the induction product from
$\md\,\H_{m_1}(t) \times \md\,\H_{m_2}(t)$ to
$\md\,\H_{m_1+m_2}(t)$.
Let $\RR$ be the sum over $m$ of the Grothendieck groups of the
categories $\md\,\H_m(t)$ endowed with the multiplication
induced by $\odot$.
The functors $\F_m$ thus induce a surjective ring homomorphism 
$\Psi : \RR \to R$, which maps classes of simples to 
classes of simples or to zero.

Let $\D_{m,\ell}$ denote the full subcategory of $\md\,\H_{m}(t)$
whose objects are those modules on which the generators 
$y_1,\ldots, y_m$ of the maximal commutative subalgebra of $\H_{m}(t)$
have all their eigen\-values in
\[
\left\{t^k \mid k\in\Z,\ \frac{1-n}{2}\le k\le \frac{n}{2}+\ell\right\},
\]
(see \cite{L}).
It is easy to check that every simple object of $\CC_\ell$ 
is of the form $\F_m(M)$ for some $m$ and some simple object $M$ of $\D_{m,\ell}$.
Therefore, denoting by $\RR_\ell$ the sum over $m$ of the Grothendieck groups
of the categories $\D_{m,\ell}$, we see that $\Psi$ restricts to 
a surjective ring homomorphism from $\RR_\ell$ to $R_\ell$.
 
By a dual version of Ariki's theorem \cite{A,LNT} the $\Z$-basis 
of $\RR_\ell$ given by the classes of the simple objects can be
identified with the dual canonical basis of the coordinate
ring $\C[N]$ of a maximal unipotent subgroup $N$ of $SL_{n+\ell+1}(\C)$. 

So, to summarize, for $\g$ of type $A_n$, Conjecture~\ref{mainconjecture2}
can be reformulated as a conjecture about multiplicative properties
of the dual canonical basis of $\C[N]$.
In \cite{GLS1}, a cluster algebra structure on $\C[N]$ has been
studied in relation with the representation theory of preprojective
algebras. It was shown that the cluster monomials belong to the
dual of Lusztig's {\em semicanonical basis} of $\C[N]$ \cite{Lu2}. 
More precisely they are the elements parametrized by the irreducible
components of the nilpotent varieties with an open orbit.
It was also conjectured that these elements of the dual semicanonical
basis belong to the dual canonical basis, hence, by Ariki's theorem,
are classes of irreducible representations of some $\H_m$.

Finally, one can check that the initial seed of the cluster algebra 
$\A_\ell$ given by Conjecture~\ref{mainconjecture2} is the image 
under $\Psi$ of a seed of $\C[N]\equiv \RR_\ell$.
So if the conjecture of \cite{GLS1} was proved, by applying $\Psi$
we would deduce that all cluster monomials of $\A_\ell$ are classes
of simple objects of~$\CC_\ell$.
To finish the proof of Conjecture~\ref{mainconjecture2} one would
still have to explain why {\em all classes of real simple objects}
are cluster monomials. 
 
\subsection{}\label{typeA4}
In \cite{GLS0} it was shown that if $N$ is of type $A_r\ (r\le 4)$
the dual canonical and dual semi\-canonical basis of $\C[N]$
coincide. Moreover these are the only cases for which $\C[N]$
has finite cluster type. It then follows from \ref{typeAl}
that Conjecture~\ref{mainconjecture2} holds 
if $n+\ell\le 4$, and moreover in this case all simple objects
of $\CC_\ell$ are real. This proves the conjecture for $\g$
of type $A_2$ and $\ell = 2$ (see \ref{A2l2}).

\subsection{}
For $\g$ of type $A_n$, there is an interesting relation between the
cluster algebra $\A_\ell$ and the grassmannian $\Gr(n+1,n+\ell+2)$
of $(n+1)$-dimensional subspaces of $\C^{n+\ell+2}$.
Indeed, the homogeneous coordinate ring $\C[\Gr(n+1,n+\ell+2)]$
has a cluster algebra structure \cite{S} with an initial seed
given by a similar rectangular lattice (see also \cite{GSV,GLSpf}).
More precisely, denote the Pl\"ucker coordinates 
of $\C[\Gr(n+1,n+\ell+2)]$ by
\[
[i_1,\ldots,i_{n+1}], \qquad ( 1\le i_1<\cdots < i_{n+1} \le n+\ell+2).
\]
The $\ell + 2$ Pl\"ucker coordinates 
\[
 [1,2,\ldots,n+1],\ [2,3,\ldots,n+2],\ \ldots ,\ [\ell+2,\ell+3,\ldots,n+\ell+2],
\]
belong to the subset of frozen variables of the cluster algebra
$\C[\Gr(n+1,n+\ell+2)]$.
Hence, the quotient ring $S_\ell$ of $\C[\Gr(n+1,n+\ell+2)]$
obtained by specializing these variables to 1 is also
a cluster algebra, with the same principal part.
By comparing the initial seed of $\A_\ell$ with the initial
seed of $S_\ell$ obtained from \cite[\S4,\S5]{S}, we see immediately that these two
cluster algebras are isomorphic.

So we can reformulate Conjecture~\ref{mainconjecture2} for  $\g=\Sl_{n+1}$
by stating that {\em $\CC_\ell$ should be a monoidal categorification
of the quotient ring $S_\ell$ of 
$\C[\Gr(n+1,n+\ell+2)]$ by the relations
\[
[1,2,\ldots,n+1]= [2,3,\ldots,n+2]=\cdots =[\ell+2,\ell+3,\ldots,n+\ell+2]=1. 
\]
}
Note that for $\g=\Sl_2$ we recover the situation of \S\ref{CPtypeA1}.
Note also that \cite{GLSpf} provides an {\em additive} categorification of
the cluster algebra
$\C[\Gr(n+1,n+\ell+2)]$, as a Frobenius subcategory of the module category
of a preprojective algebra of type $A_{n+\ell+1}$.
\begin{example}
{\rm
Take $n=2$ and $\ell=2$ (see \S\ref{A2l2}). 
In this case $S_2$ is the ring obtained from $\C[\Gr(3,6)]$
by quotienting the following relations
\[
[1,2,3] = [2,3,4] = [3,4,5] = [4,5,6] = 1. 
\]
For simplicity, we denote again by $[i,j,k]$
the image of the Pl\"ucker coordinate in the quotient $S_2$.
Then the identification of the Grothendieck ring $R_2$ with $S_2$
gives the following identities of cluster (and frozen) variables:
\[
\begin{array}{lll}
[L(Y_{1,0})]=[3,4,6],
&[L(Y_{1,2})]=[2,3,5], 
&[L(Y_{1,4})]=[1,2,4],
\\[2mm]
[L(Y_{2,1})]=[3,5,6],
&[L(Y_{2,3})]=[2,4,5], 
&[L(Y_{2,5})]=[1,3,4],
\\[2mm]
[L(Y_{1,0}Y_{1,2})]=[2,3,6],
&[L(Y_{1,2}Y_{1,4})]=[1,2,5], 
&[L(Y_{1,0}Y_{1,2}Y_{1,4})]=[1,2,6],
\\[2mm]
[L(Y_{2,1}Y_{2,3})]=[2,5,6],
&[L(Y_{2,3}Y_{2,5})]=[1,4,5], 
&[L(Y_{2,1}Y_{2,3}Y_{2,5})]=[1,5,6],
\\[2mm]
[L(Y_{1,0}Y_{2,3})]=[2,4,6],
&[L(Y_{1,2}Y_{2,5})]=[1,3,5], 
&[L(Y_{1,4}Y_{2,1})]=[1,3,4][2,5,6]-[1,5,6],
\\[2mm]
[L(Y_{1,0}Y_{1,2}Y_{2,5})]=[1,3,6],
&[L(Y_{1,0}Y_{2,3}Y_{2,5})]=[1,4,6], 
&[L(Y_{1,0}Y_{1,2}Y_{2,3}Y_{2,5})]=[2,3,6][1,4,5]-1.
\end{array}
\]
Moreover, as is easily checked, the dimension of a simple module in $\CC_2$ 
is obtained by evaluating the corresponding cluster monomial in $S_2$
on the matrix
\[
\left[
\matrix{
1&1&1&1&1&1\cr
0&1&2&3&4&5\cr 
0&0&1&3&6&10
}
\right] 
\]
Thus, 
\[
\dim L(Y_{10}Y_{12}Y_{23}Y_{25}) = 
\left|
\matrix{
1&1&1\cr
1&2&5\cr 
0&1&10
}
\right|  
\times
\left|
\matrix{
1&1&1\cr
0&3&4\cr 
0&3&6
}
\right|
-1
= 35.
\] 
} 
\end{example}


\bigskip
\small
\noindent
\begin{tabular}{ll}
David {\sc Hernandez} : &
CNRS et Ecole Normale Sup\'erieure Paris\\
& DMA, 45, rue d'Ulm, 75005 Paris, France \\
&email : {\tt David.Hernandez@ens.fr}\\[5mm]
Bernard {\sc Leclerc} :&
LMNO, CNRS UMR 6139,
Universit\'e de Caen,\\
& 14032 Caen cedex, France\\
&email : {\tt leclerc@math.unicaen.fr}
\end{tabular}

\end{document}